\documentclass[12pt, letterpaper]{article}

\usepackage[longnamesfirst]{natbib}
\bibpunct[ ]{(}{)}{,}{a}{}{,}

\usepackage{palatino,amsmath,amssymb,mathrsfs,amsthm,amsfonts,mathtools,caption,bm}
\usepackage[inline]{enumitem}
\usepackage[usenames,dvipsnames]{xcolor}
\usepackage{hyperref}
\usepackage{graphicx}
\usepackage{stmaryrd}
\usepackage[percent]{overpic}
\usepackage{comment}
\usepackage{tikz}
\usetikzlibrary{matrix,graphs,arrows,positioning, shapes, calc,decorations.markings,decorations.pathmorphing,shapes.symbols,angles}
\hypersetup{
colorlinks=true, linkcolor=black, 
citecolor=OliveGreen
}

\newtheorem{theorem}{Theorem}[section]

\newtheorem{lemma}[theorem]{Lemma}
\newtheorem{proposition}[theorem]{Proposition}
\newtheorem{corollary}[theorem]{Corollary}
\theoremstyle{definition}
\newtheorem{definition}[theorem]{Definition}
\newtheorem{conjecture}[theorem]{Conjecture}

\newtheorem{remark}[theorem]{Remark}
\newtheorem{example}[theorem]{Example}
\numberwithin{equation}{section}
\allowdisplaybreaks

\numberwithin{equation}{section}

\renewcommand{\S}{\mathcal{S}}
\newcommand{\mx}{0}
\newcommand{\gr}{\mathfrak g}

\newcommand{\e}{\varepsilon}

\renewcommand{\L}{\mathcal{L}}
\newcommand{\xdownarrow}[1]{{\left\downarrow\vbox to #1{}\right.\kern-\nulldelimiterspace}}
\newcommand{\xuparrow}[1]{{\left\uparrow\vbox to #1{}\right.\kern-\nulldelimiterspace}}

\newcommand{\Z}{\mathbb{Z}}
\newcommand{\R}{\mathbb{R}}
\newcommand{\N}{\mathbb{N}}

\renewcommand{\Pr}{{P}}	
\newcommand{\Ex}{{E}}

\newcommand{\m}{\mathsf}

\renewcommand{\hat}{\widehat}

\renewcommand{\bar}{\overline}

\newcommand{\pj}{Q}

\newcommand{\ga}{\gamma}
\newcommand{\de}{\delta}
\newcommand{\ep}{\varepsilon}
\newcommand{\al}{\alpha}

\newcommand{\Ga}{\Gamma}
\newcommand{\cL}{\mathcal{L}}
\newcommand{\cK}{\mathcal{K}}
\newcommand{\cD}{\mathcal{D}}
\newcommand{\cE}{\mathcal{E}}
\newcommand{\fd}{\bm{d}}

\newcommand{\Rd}{{\mathbb{R}}^4_\uparrow}
\newcommand{\eqd}{\overset{d}{=}}

\newcommand{\pn}{\mathcal P}

\title
{Upper tail large deviations of the directed landscape}

\author{Sayan Das \and Duncan Dauvergne \and B\'alint Vir\'ag}
\date{}
\begin{document}
\maketitle
\begin{abstract} Starting from one-point tail bounds, we establish an upper tail large deviation principle for the directed landscape at the metric level. Metrics of finite rate are in one-to-one correspondence with measures supported on a set of countably many paths, and the rate function is given by a certain \textit{Kruzhkov entropy} of these measures. As an application of our main result, we prove a large deviation principle for the directed geodesic.
\end{abstract}


\maketitle
{
	\hypersetup{linkcolor=black}
	\setcounter{tocdepth}{1}
	\tableofcontents
}

\section{ Introduction}

The directed landscape describes  random planar geometry:  it is expected to be the scaling limit of planar first passage percolation, and has been proven to be the scaling limit of a rich class of models in the KPZ (Kardar-Parisi-Zhang) universality class, including integrable models of last passage percolation (\cite{DOV18, dv22}), coloured tasep (\cite{aggarwal2024scaling}), and the continuum directed random polymer and the KPZ equation (\cite{wu2023kpz}). 

The directed landscape is best thought of as a random directed metric on the space-time plane. More precisely, define $\mathcal E$ be the set of all continuous functions $e$ from $\Rd=\{(x,s;y,t)\in \mathbb R^4, s<t\}$ to $\R$ satisfying the (reverse) triangle inequality 
$$
e(p;q)+e(q;r) \le e(p;r)
$$
for $(p;q),(q;r)\in \Rd$. The directed landscape is a random element $\cL \in \mathcal E$. Because of the triangle inequality, we can think of an element $e \in \mathcal E$ as a \textbf{(directed) metric}, and view $e$ as assigning distances between pairs of points $(x, s), (y, t)$ in the space-time plane. The condition that $e$ only assigns a distance when $s < t$ makes the second-coordinate time-like, and because of this restriction, there is no contradiction in allowing our directed metrics to be real-valued. The fact that the triangle inequality is reversed for the directed landscape is simply a convention.

For a fixed point pair $u=(x,s;y,t)\in \Rd$ the law of $\cL(u)$ is given by
\begin{equation}\label{e:landscape-Dirichlet}
	\mathcal L(u)\eqd (t-s)^{1/3}\mathrm{TW} + d(u), \qquad d(u)=-\frac{(y-x)^2}{t-s},
\end{equation}
where TW is a GUE Tracy-Widom random variable. We call $d$ the {\bf Dirichlet metric} (with negative sign). In particular, it can be shown that after scaling  
\begin{align}
	\label{E:L-epsilon}
	\L_\e(x,s;y,t):=\e\L(x\e^{-1/2},s;y\e^{-1/2},t)
\end{align}
as $\e\to 0$ we have $\L_\e \to d$  uniformly on bounded subsets (see Corollary 10.7 in  \cite{DOV18} for a quantitative version). But how likely is it that $\cL_\e$ is close to a  different metric? 

Our main theorem answers this question. For this theorem and throughout the paper, we equip $\mathcal E$ with the  topology of uniform convergence on bounded sets. This is stronger than uniform convergence on compact sets, since it requires uniform convergence near the boundary $s=t$ in $\Rd$.

\begin{theorem}\label{t:main} There exists a lower semicontinuous function $I : \mathcal E \to [0,\infty]$ such that for every Borel measurable  $A\subset \mathcal E$, as $\e\to 0$ we have
	\begin{equation*}   
		\exp((o(1)-\inf_{A^\circ} I)\e^{-3/2}) \le P( \L_\e\in A) \le \exp((o(1)-\inf_{\bar{A}}I)\e^{-3/2}).
	\end{equation*}
	Moreover, $I^{-1}[0,a]$ is compact for every  $a<\infty$,  $I^{-1}(0)=d$, and $I$ is strictly increasing on $I^{-1}[0,\infty)$.
\end{theorem}
In other words, the family $\{\L_\e\}$ satisfies a large deviation principle with speed $\e^{-3/2}$ and good rate function $I$. 


Large deviations at speed $\ep^{-3/2}$ are typically called upper-tail large deviations in the KPZ literature. There is no notion of ``upper tail'' involved in the statement of the theorem.  Lower tail events simply have infinite rate at this speed. There is no contradiction to having different large deviation principles at different speeds, and we expect to see another principle at speed $\e^{-3}$ corresponding to the lower tail. Indeed, in the closely related setting of first passage percolation, a lower tail  large deviation principle at the metric level was recently established by \cite{verges2024large}. What is special for this particular speed is that $I$ is a good rate function: the  sub-level sets are compact. At higher speeds sub-level sets must contain our entire $I^{-1}[0,\infty)$. 

Monotonicity of the rate function is a feature of large deviations at the metric level. It is not expected to hold for important contractions, such as the upper-tail large deviations for the Airy process. Another feature of the metric-level large deviations is a simple characterization of the rate function, which we describe next. 

\subsection{The rate function}
\label{SS:rate-function-intro}

For $e \in \mathcal E$ and a continuous function $\ga:[s, t] \to \R$, henceforth referred to as a path, define the \textbf{length} $|\ga|_e$ of $\ga$ with respect to $e$ by
\begin{equation}
	\label{E:ga-length}
	|\ga|_e = \inf_{k \in \N} \inf_{s = r_0 < r_1 < \dots < r_k} \sum_{i=1}^k e(\ga(r_{i-1}), r_{i-1}; \ga(r_i), r_i)).
\end{equation}
The triangle inequality guarantees that $|\ga|_e \le e(\ga(s), s; \ga(t), t)$, and we call a path $\ga$ an \textbf{$e$-geodesic} if this is an equality.

For the Dirichlet metric, the length is a negative Dirichlet energy: when $\ga$ is absolutely continuous, we have $|\ga|_d = -\int_s^t \ga'^2$, and otherwise $|\ga|_d = -\infty$. Let $H^1$ be the set of all paths $\ga$ with $|\ga|_d > -\infty$. The metrics in Theorem \ref{t:main} with finite rate can be described planting good regions in $\R^2$ on top of the Dirichlet metric.

Given a singular measure $\mu$ on $\R^2$, define a directed metric $e_\mu$ as follows. For a path $\ga:[s, t] \to \R$, let $\gr \gamma=\{(\gamma(r),r):r\in[s,t]\}$ denote the \textbf{graph} of $\gamma$, and define
\begin{equation}
	\label{e:emu}
	e_\mu(x,s;y,t)=\sup_\gamma \mu(\gr \gamma)+|\gamma|_d, 
\end{equation}
where the supremum is over all $\gamma\in H^1$ with domain $[s,t]$ satisfying $\gamma(s)=x$, $\gamma(t)=y$. In words, paths are rewarded by covering sets of large $\mu$-measure with their graphs, but are penalized by the Dirichlet metric for excessive undulation. 

Such metrics were studied by  \cite{bakhtin2013burgers} in the case when $\mu$ is a Poisson point process. For us, $\mu$ will be deterministic and without atoms.  We note in passing that under some technical conditions,  as a function of $y,t$ the metric $e_{\mu}(x,s;y,t)$ satisfies a Hamilton-Jacobi equation (an Eikonal equation) with ``forcing term'' given by $\mu$. The $y$-derivative then satisfies a forced inviscid Burgers' equation.

More precisely, our measures $\mu$ will be supported on a countable union of graphs $\gr \gamma, \ga \in \Ga \subset H^1$ and have the property that for every $\ga \in \Ga$, the time-marginal of the restricted measure $\mu|_{\gr \ga}$ has a Lebesgue density.
Finally, all our measures will have finite {\bf Kruzhkov entropy} $\mathcal K(\mu)$. For this, define the {\bf temporal density} $\rho_\mu(x,t)$ of $\mu$, by setting $\rho_\mu(x,t) = 0$ for $(x, t) \notin \text{supp}(\mu)$, and whenever $(x,t)\in \gr\gamma$ for some $\gamma\in \Gamma$, letting $\rho_\mu(x,t)$ be the Lebesgue density of the $t$-marginal of $\mu|_{\gr \ga}$.
As long as $\rho_\mu$ exists, it is well-defined $\mu$-a.e. Then 
$$
\mathcal K(\mu)=\int \sqrt{\rho_\mu}d\mu.
$$
The name is motivated by Kruzhkov's analysis of Burgers' equation and its generalizations, \cite{kruzhkov1970first}. The definition is related to the concept of R\'enyi entropy, where a power of the density is integrated against the measure. 
We call a measure satisfying the above conditions a {\bf planted network measure}. Then for all $e\in \mathcal E$,
$$
I(e)=\begin{cases}
	\frac43 \mathcal K(\mu)\qquad &\text{if }e=e_\mu \text{ for some planted network measure $\mu$}
	\\ \infty &otherwise. 
\end{cases}
$$
We will also show that in any finite rate metric, the length of a path $\gamma\in H^1$ is given by $\mu(\gr \gamma)+|\gamma|_d$. 
This representation quickly implies the last two claims of Theorem \ref{t:main}.

\subsection*{Applications, an alternate characterization, and the proof}
In Section \ref{sec:app}, a continuation of this introduction, we give several applications of Theorem \ref{t:main} to describe large deviations for marginals of the directed landscape. Among other results, we describe the joint large deviations of $\mathcal L(u_1), \dots , \cL(u_k)$ for any finite collection of points $u_1, \dots, u_k$. 

We also present a full large deviation principle for the directed geodesic $\Pi$. This is the a.s.\ unique $\cL$-geodesic between the points $(0,0)$ and $(0, 1)$. By the symmetries of the directed landscape, any $\cL$-geodesic between deterministic points is equal in law to an affine function of $\Pi$. The large deviations of the directed geodesic are described by a good rate function $J:C([0, 1]) \to [0, \infty]$. It turns out that $J(f)$ solves a variational problem, and its value is between $4/3$ times the cubes of the $L^{3/2}$ and $L^3$ norms of $f$, and that both of these norms are attained for different choices of $f$. We use our large deviation principle to show that as $a\to\infty$ we have $P(\sup_{[0,1]}|\Pi|\ge a)=e^{-\frac{32}3a^3+o(a^3)}$, to partially solve a conjecture of \cite{liu2022one}, and to present a new conjecture. 

A natural question that arises from our description of the rate function is how to systematically determine $\mu$ from the measure $e$. This can be done by differentiation. For $e \in \mathcal E$ and $p \in \R^2$, define 
\begin{equation}
	\label{E:sup-theta-R}
	\rho_e(p) = \sup_{\theta \in \R} \lim_{t \to 0^+} \frac{e(p; p + (\theta t, t)) - d(p; p + (\theta t, t))}{t}.
\end{equation}
The limit above can be thought of as quantifying how much excess density the metric $e$ has at $p$ when compared with the Dirichlet metric. When $e = e_\mu$ for a planted network measure $\mu$, then $\rho_e$ is the temporal density of $\mu$. Another upshot of this differentiation approach to understanding finite rate metrics is a formula for the rate function in terms of an integrated Dirichlet energy. This interpretation is closely related to the predicted large deviations of the Airy process, which should be the same as those of a Brownian motion conditioned to remain above a parabola. See Section \ref{sec:dirichlet} for details.

Our proof relies only on single-point bounds for the Tracy-Widom random variables, and builds the geometric picture on top of this by relying on deterministic topological arguments and the metric structure of the directed landscape. Indeed, most of the proof is devoted to understanding the deterministic structure of finite-rate metrics in $\mathcal E$. Because we use so little about the structure of the directed landscape, we believe that the methods should extend quite generally to other first and last passage percolation models, even non-integrable ones under some mild conditions. 

\subsection*{A partial review of  previous work}

We give a partial literature review here, focusing on work in the KPZ universality class related to the study of large deviations. 
See \cite*{quastel2011introduction}; 
\cite*{corwin2012kardar};  \cite*{zygouras2018some};
\cite*{ganguly2021random} and references therein for a broader review of the KPZ universality class.

Large deviations in the KPZ universality class have been studied by many authors. The upper tail bound for the  Tracy-Widom law  can be deduced from results in the original paper of \cite{tracy1994level}. A probabilistic derivation using the stochastic Airy operator was given by \cite{ramirez2011beta}.

One-point large deviations in exactly solvable last passage percolation models have been explored by \cite{ls77,sep98a,sep98b,dz99,joh00} using a range of methods. At the process level, large deviations for tasep (equivalently exponential last passage percolation) in the upper-tail regime were studied by \cite{jensen2000asymmetric} and \cite{varadhan2004large}, but the full proof in this regime had to wait until the recent paper of \cite{quastel2021hydrodynamic}. A version of Kruzhkov entropy appears in that paper. To express the rate function there, one has to solve a Burgers' equation. In our case,  the formulas simplify because we consider the higher, metric level of large deviations. Also, at this level, we only need to rely on single-point tail bounds as opposed to the KPZ fixed point-like determinantal formulas used in \cite{quastel2021hydrodynamic}. 


Process level large deviations at different speed (those corresponding to lower tail) have been studied by \cite{ot19} for tasep. The recent exciting preprint of \cite{verges2024large} proves a lower tail large deviation principle in first passage percolation on $\Z^d$ at the metric level. A one-point lower tail large deviation principle for this model with $d=2$ was previously shown by \cite{bgs17}. 


A study of what happens to the directed landscape and its geodesics at finer fluctuation scales under large deviation events was initiated in \cite{liu2022geodesic,liu2024conditional}. These papers obtained one-point and multi-point Gaussian convergence for the geodesic and the KPZ fixed point, respectively, using exact formulas from \cite{liu2022one}. Relying on probabilistic and geometric arguments, a Brownian bridge limit for the geodesic under an upper-tail large deviation event was proven in \cite{ganguly2023brownian} using precise bounds from \cite{ganguly2022sharp}.

Recently, there has been significant interest within both the mathematics and physics communities in understanding large deviations of the  KPZ equation (\cite{kpz86}). The works of \cite{le2016large,kamenev2016short,sasorov2017large,corwin2018coulomb,krajenbrink2018systematic,Le_Doussal_KP_large_deviations,kraj19,cg20b,lwtail,dt21,tsai_lower_tail,cafasso_claeys_KPZ,gl20} obtain results on one-point tails of KPZ equation. Process level limits and large deviations for the KPZ equation were studied by \cite{lin2023spacetime} and \cite{gaudreau2023kpz}. 
For some other integrable models in the KPZ universality class, one-point large deviation principles have been established by \cite{gs13,jan15,ej17,jan19,asep,qpng}.

\subsection*{Organization} The rest of the paper is organized as follows. In Section \ref{sec:app}, we describe various applications of our main theorem. In Section \ref{sec:basic}, we collect preliminary properties and estimates for the directed landscape and Airy sheet, and establish simple extensions of the one-point bounds of Theorem \ref{T:tracy-widom-tails}. Woven through Section \ref{sec:basic} is a heuristic derivation of the large deviation principle, which is naturally suggested by the preliminaries.
In Section \ref{sec:top}, we introduce a subspace of $\mathcal E$, where the rate function is finite, and prove properties of it. The rate function and various approximations for metrics are defined and analyzed in Sections \ref{sec:rate} and \ref{sec:plant}. In Sections \ref{sec:upbd} and \ref{sec:lwbd}, we prove the upper and lower bounds for the large deviations. Finally, in Section \ref{sec:dirichlet}, we express the rate function as an integrated Dirichlet energy.

\section{Applications} \label{sec:app}
In this section, we study the minimizing metric and its rate function for various optimization problems. We first prove a proposition which describing a generic structure of the minimizer for a particular class of optimization problems.

A finite or countable collection of paths is {\bf internally disjoint} if for all distinct $\ga, \ga'$ in the collection with domains $[a, b], [c, d]$ and $r \in (a, b) \cap (c, d)$ we have $\ga(r) \ne \ga'(r)$. We will use the intuitive notion of rightmost geodesics, see the discussion before Lemma \ref{L:geodesic-space0} for the precise definition, but otherwise the proposition and its proof only appeal to Theorem \ref{t:main} and the structure of the rate function introduced in Section \ref{SS:rate-function-intro}.

\begin{proposition}[Multi-point upper tail decay for the directed landscape]\label{P:stline} 
	\ \\
	For $i=1,\ldots, k$, let $(p_i;q_i)\in  \mathbb{R}_{\uparrow}^4$ and $\al_i\in \mathbb R$. Then  
	\begin{equation}
		\label{E:PLE}
	P(\L_\e(p_i;q_i)\ge \alpha_i, i=1,\ldots, k)=\exp(-\beta \e^{-3/2}+o(\e^{-3/2})), \qquad \text{as} \quad \e\to 0.
	\end{equation}
 Here 
	$\beta<\infty$ solves the variational problem
	\begin{align*}
		\beta := \inf_{e \in \mathcal M} I(e), \qquad \mathcal M=\big\{e\in \mathcal E:  e(u_i) \ge \alpha_i \ \mbox{ for }i=1,\ldots,k\big\}.
	\end{align*}
	This infimum is achieved by at least one $e \in \mathcal M$, and any such $e$ equals $e_\mu$ for some measure $\mu$ with the following structure.
	\begin{enumerate}[label=(\alph*),leftmargin=18pt]
		\item \label{propa} The support of $\mu$ is contained in a set of the form $\bigcup_{i=1}^k \gr \pi_i$ where each $\pi_i$ is the rightmost $e_\mu$-geodesic from $p_i$ to $q_i$.
		\item \label{propb} There is a finite collection of line segments $\ga_1, \dots, \ga_m$ such that $\bigcup_{i=1}^k \gr \pi_i = \bigcup_{i=1}^m \gr \gamma_i$, and $\rho_\mu$ is $\mu$-a.e. constant on each of the line segments $\gr \gamma_i$.
		\item \label{propc} If $\gamma_i(t)=\gamma_j(t)$  for $i\neq j$, then either $(\gamma_i(t),t) \in \bigcup_{\ell = 1}^k \{p_\ell, q_\ell\}$, or at least three segments meet at the point $(\ga_i(t), t)$. That is, there exists an index $\ell$ distinct form $i$ and $j$ so that $\gamma_i(t)=\gamma_j(t)=\gamma_\ell(t)$.
		
	\end{enumerate}
\end{proposition}

\begin{proof} 
	To see that $\beta < \infty$,  observe that we can find a finite rate metric in $\mathcal M$ by considering a measure $\mu$ supported on straight line segments between $p_i$
	to $q_i$ and setting $\rho_\mu$ to be a high enough constant on the graphs of these segments. 
	
	Now suppose $e_n \in \mathcal M$ is a sequence of metrics with $I(e_n) \to \beta$ as $n \to \infty$. Since $\beta < \infty$ and sub-level sets of $I$ are compact, $e_n$ has a subsequential limit $e^*$, which lies in $\mathcal M$ since this set is closed. Moreover, $I(e^*) \le \beta$ by lower semicontinuity of $I$. Hence $e^*=e_\mu$ for some $\mu$ and $I(e_\mu)=\beta$.
	
	Next, let $\pi_i$ be the rightmost geodesic from $p_i$ to $q_i$ under $e^*$ (these exist by Lemma \ref{L:geodesic-space}). The metric $\mu$ must be supported inside $\bigcup_{i=1}^k\gr \pi_i$, since the metric induced by $\mu|_{\bigcup_{i=1}^k\gr \pi_i}$ is still in $\mathcal M$, and has strictly smaller Kruzhkov entropy if $\mu$ has mass outside this set. This gives \ref{propa}.
	
		Let us consider $\nu$ defined via $\rho_{\nu}:=\rho_{\mu}+c$ with $c>0$. Note that $e_{\nu} \in \mathcal M^{\circ}$ and $I(e_{\nu})-I(e_{\mu})$ can be made arbitrarily small by choosing $c$ small enough as $\mu$ is supported on finitely many paths. Thus $\beta=\inf_{\mathcal M^\circ} I$, and consequently \eqref{E:PLE} is now a direct application of the Theorem \ref{t:main}. 
	
	Now, any two distinct paths $\pi_i, \pi_j$ must overlap on a closed interval, a simple fact about rightmost geodesics. Therefore there is a collection of internally disjoint paths $\gamma_i, i = 1, \dots, m$ such that $\bigcup_{i=1}^k\gr \pi_i = \bigcup_{i=1}^m \gr \ga_i$. If we choose $m$ minimally, then \ref{propc} must hold for this collection. Indeed, if there were a pair $\ga_i, \ga_j$ with $\ga_i(t) = \ga_j(t)$, then internal disjointness guarantees that $t$ is an endpoint of the domains $[a_i, b_i], [a_j, b_j]$ of $\ga_i, \ga_j$. If $t = a_i = b_j$ or $t = a_j = b_i$, then we can concatenate $\ga_i, \ga_j$ to give an internally disjoint collection of $m-1$ paths whose graphs cover the same region $\bigcup_{i=1}^k\gr \pi_i$. If $t = a_i = a_j$ or $t = b_i = b_j$, then either $(\ga_i(t), t)$ is an endpoint of one of the paths $\pi_i$, or else there is a third path $\ga_\ell$ with $\ga_i(t) = \ga_j(t) = \ga_\ell(t)$.

	Finally, we claim that each $\ga_i$ must be a straight line, and that $\rho_\mu$ must be constant along $\gr \ga_i$. It is enough to check this for $\ga_1:[a, b]\to \R$. Define $\mu^* = \mu|_{\bigcup_{i=2}^m \gr \ga_i}$, and for $\e \ge 0$ let $\nu_\e$ be the measure supported on the line segment $\ga_1^*$ from $(\ga_1(a), a)$ to $(\ga_1(b), b)$ and with $\rho_\nu$ constant along $\ga_1^*$, and equal to $(\al^{3/2} - \ep)^{2/3}$, where $\al$ is the average of $\rho_\mu(\ga_1(t), t), t \in [a, b]$. Let $\mu^*_\ep$ be the measure with $\rho_{\mu^*_\ep} = \rho_{\mu^*} \vee \rho_{\nu_\ep}$. By a short computation,
	\begin{align*}
		\mathcal K(\mu) - \mathcal K(\mu^*_\ep) &\ge \mathcal K(\mu|_{\gr \ga_1}) - \mathcal K(\nu_\ep) \\
		&= \int_a^b \rho_\mu(\ga_1(t), t)^{3/2} dt - \frac1{(b-a)^{1/2}} \int_a^b \rho_\mu(\ga_1(t), t) dt + (b-a) \ep \\
		&\ge (b-a) \ep,
	\end{align*}
	where the final bound uses Jensen's inequality. Moreover, the $e_\mu$-length of any concatenation of paths $\ga_{i_1}, \dots, \ga_{i_\ell}$ is bounded above by the $e_{\mu^*_\ep}$-length of any concatenation of paths $\ga_{i_1}^*, \dots, \ga_{i_\ell}^*$ where $\ga_i^* = \ga_i$ for $i \ge 2$, as long as $|\ga_1^*|_{e_{\mu^*_\e}} \ge |\ga_1|_{e_\mu}$. Using the characterization of length in terms of the measure $\mu$ discussed at the end of Section \ref{SS:rate-function-intro}, this inequality is equivalent to $\mu^*_\e(\ga_1^*) + |\ga_1^*|_d \ge \mu(\ga_1) + |\ga_1|_d$, which is implied by the claim that 
	$$
	(b-a)(\al^{3/2} - \ep)^{2/3} - (b-a)\al + |\ga_1^*|_d \ge |\ga_1|_d.
	$$
	This inequality implies that if $\ga_1$ is not a straight line, then $e_{\mu^*_\ep} \in \mathcal M$ for some $\e > 0$, which is a contradiction since $\mathcal K (\mu_\e^*) < \mathcal K(\mu_\e)$ for $\e > 0$. If $\ga_1$ is a straight line, then we still have that $\mathcal K (\mu_0^*) \in \mathcal M$ and so the inequalities in the previous display must all be equalities, forcing $\rho_\mu$ to be constant along $\ga_1$. 
\end{proof}

We use Proposition \ref{P:stline} to resolve two optimization problems.

\begin{example}[Back to a single point]\label{e:one-point} For $\al \ge 0$, consider the one point variational problem:
	\begin{align*}
		\inf \big\{I(e) :  e(0,0;0,1) \ge \alpha\big\}= \tfrac43\alpha^{3/2}.
	\end{align*}
	By Proposition \ref{P:stline}, a measure $\mu$ for the optimizing metric is supported on the  straight line from $(0,0)$ to $(0,1)$ with $\rho_\mu = \al_1$, a constant, along this line. The bound  $e_\mu(0,0;0,1)\ge \alpha$ implies $\alpha_1\ge \alpha$. Thus $I(e_\mu)=\frac{4}{3}\alpha_1^{3/2}$, which is minimized when $\al_1 = \al$.
\end{example}

\begin{example}[V or Y?]\label{e:two-point}Let us consider the two-point variational problem:
	\begin{align*}
		\inf \big\{I(e) :  e(0,0;-1,1) \ge \alpha, e(0,0;1,1) \ge \alpha\big\}.
	\end{align*}
	By Proposition \ref{P:stline}, the optimizing measure $\mu$ with temporal density $\rho$ is supported on three line segments that meet at some point $(r,t)$:
	\begin{align*}
		\begin{array}{lll}
			\ga:[0,t]\to\mathbb{R}, \qquad \qquad
			& \ga(s)=rs/t, \qquad \qquad \qquad &\rho|_{\gr \gamma}-(\gamma')^2\equiv p, 
			\\
			\ga_1: [t,1]\to\mathbb{R}, &\ga_1(s)=r-(s-t)\frac{r+1}{1-t}, &\rho|_{\gr \gamma_1}-(\gamma'_1)^2 \equiv q_1, 
			\\
			\ga_2: [t,1]\to\mathbb{R}, &\ga_2(s)=r-(s-t)\frac{r-1}{1-t}, &\rho|_{\gr \gamma_2}-(\gamma'_2)^2\equiv q_2.
		\end{array}
	\end{align*}
	Note that $e_\mu(0,0;-1,1)=pt+q_1(1-t) \ge \alpha$ which implies $q_1 \ge (\alpha-pt)/(1-t)$. Similarly, $q_2 \ge (\alpha-pt)/(1-t)$. We have
	\begin{align*}  I(e_\mu) & =\tfrac43\int_0^t(p+\ga'(s)^2)^{3/2}ds+\tfrac43\int_t^1(q_1+\ga_1'(s)^2)^{3/2}ds+\tfrac43\int_t^1(q_2+\ga_2'(s)^2)^{3/2}ds \\ & 
		\ge \tfrac43\left[
		\frac{(tp+\frac{r^2}{t})^{3/2}}{\sqrt{t}}
		+
		\frac{(\alpha-pt+\frac{(r+1)^2}{1-t})^{3/2}}{\sqrt{1-t}}
		+
		\frac{(\alpha-pt+\frac{(r-1)^2}{1-t})^{3/2}}{\sqrt{1-t}}\right].
	\end{align*}
	with equality if and only if $e_\mu(0,0;-1,1) = e_\mu(0,0;1,1) = \al$. Therefore the optimizer $e_\mu$ must have $q_1 = q_2 = (\alpha-pt)/(1-t)$. To find the optimizer from this point forward, we shall minimize the right-hand side above over $p,r,t$. As a function of $r$, the above expression is minimized when $r=0$. Thus,
	\begin{align*}
		I(p,t):=I(e_\mu)=\tfrac43\left[tp^{3/2}+2\frac{(\alpha-pt+(1-t)^{-1})^{3/2}}{\sqrt{1-t}}\right].
	\end{align*}
	As a function of $p$, one can check that the above expression is minimized when $$p=p^*(t):=\frac{4(\alpha+(1-t)^{-1})}{1+3t}.$$
	Finally, $I(p^*(t),t)$ is minimized at $t^*$ where
	\begin{align*}
		t^*=\begin{cases}
			0, & \alpha\in [-1,0], \\
			(\sqrt{1+1/\alpha}-\sqrt{1/\alpha})^2,  & \alpha>0.
		\end{cases}
	\end{align*}
	\begin{figure}[t]
		\centering
		\captionsetup{width=0.9\linewidth}
		\begin{tikzpicture}[line cap=round,line join=round,>=triangle 45,x=1.5cm,y=3cm]
			\draw (-2,0)--(2,0);
			\draw (-2,1)--(2,1);
			\draw (0,0)--(-1,1);
			\draw (0,0)--(1,1);
			\draw (3,0)--(7,0);
			\draw (3,1)--(7,1);
			\draw (5,0)--(5,0.4);
			\draw (5,0.4)--(4,1);
			\draw (5,0.4)--(6,1);
			\begin{scriptsize}
				\node at (1,1.1) {$(1,1)$};
				\node at (-1,1.1) {$(-1,1)$};
				\node at (4,1.1) {$(1,1)$};
				\node at (6,1.1) {$(-1,1)$};
				\node at (0,-0.1) {$(0,0)$};
				\node at (6.2,0.4) {$\left(0,\left(\sqrt{1+\frac1\alpha}-\sqrt{\frac1\alpha}\right)^2\right)$};
				\node at (5,-0.1) {$(0,0)$};
			\end{scriptsize}
		\end{tikzpicture}
		\caption{Optimizers in Example \ref{e:two-point} for $\alpha\in [-1,0]$ (left) and $\alpha>0$ (right).}
		\label{fig:2path}
	\end{figure}
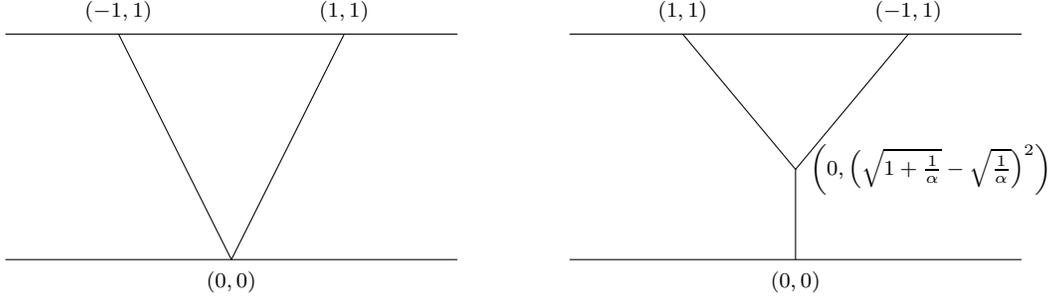
	We have
	\begin{align*}
		I(p^*(t^*),t^*)= \begin{cases}
			\frac83(1+\alpha)^{3/2}, & \alpha\in [-1,0], \\
			\frac43+2\alpha+\frac43(1+\alpha)^{3/2},  & \alpha>0.
		\end{cases}
	\end{align*}
	The optimizing metric thus takes a `V' structure for $\alpha\in [-1,0]$ and a 'Y' structure for $\alpha>0$, see Figure \ref{fig:2path}.
\end{example}

The remaining examples study variational problems with different forms than Proposition \ref{P:stline}. 

Given a finite or countable collection of internally disjoint paths $\{\ga_i : [a_{\ga_i},b_{\ga_i}]\to \R\}_{i\in I} \subset H^1$ with corresponding \textit{weight functions} $\{w_i : [a_{\ga_i},b_{\ga_i}]\to \R\}_{i\in I}$ satisfying $w_i'\ge -(\ga_i')^2$, we define a measure $\mu$ on $\R^2$  by setting its temporal density to be $\rho_{\mu}=w_i'+\ga_i'^2$ on each $\gr\ga_i$ and zero everywhere else. We shall call the corresponding metric $e_{\mu}$ (defined via \eqref{e:emu}) to be the metric obtained by planting paths $\{\ga_i\}_{i\in I}$ with corresponding weight functions $\{w_i\}_{i\in I}$. For such a metric $e_{\mu}$ we have
\begin{align}\label{e:plantrate}
	I(e_{\mu})=\frac43\sum_{i\in I} \int_{a_{\ga_i}}^{b_{\ga_i}}\big(w_i'(t)+(\ga_i'(t))^2\big)^{3/2}dt.
\end{align}

\begin{example}[Weight function large deviations]
	Consider the optimization problem:
	\begin{align*}
		\inf \big\{I(e) :  e(0,0;0,t)=w(t), \mbox{ for }t\in [0,1]\big\},
	\end{align*}
	where $w:[0, 1] \to \R$ is an absolutely continuous function with $w(0)=0$, $w' \ge 0$, and $\int_0^1|w'(t)|^{3/2} dt < \infty$. Then the above infimum is achieved by the metric $e_{\mu_*}$, where $\rho_{\mu_*}(0, t) = w'(t)$ for $t \in [0, 1]$, and $\rho_{\mu_*} = 0$ elsewhere. From the fact that $|\ga|_{e_{\mu_*}} = \mu_*(\gr \ga) + |\ga|_d$, it is easy to check that $e_{\mu_*}(0,0;0,t)=w(t)$ for $t \in [0, 1]$. 
	
	On the other hand, let $e$ be any finite rate metric satisfying $e(0,0;0,t)=w(t)$ for $t\in [0,1]$. 
	For each $n\ge 1$ and $k\in \{1,2,\ldots,2^n\}$, let $\pi_{k,n}$ be the rightmost geodesic from $(0,0)$ to  $(0, k2^{-n})$ under $e$. Again these exist by Lemma \ref{L:geodesic-space}.  Let $e_n$ be the metric obtained from $\mu$ restricted to $\bigcup_{k=1}^{2^n} \gr\pi_{k,n}$.  Then $\{e_n\}_{n\ge 1}$ is an increasing sequence of metrics bounded above by $e$. In particular, the sequence $e_n$ is precompact and monotone, so $e_n \uparrow e^*$ for some metric $e^*$.
	We have
	$$
	\lim_{n\to\infty} I(e_n)=I(e^*)\le  I(e)
	$$ and $e^*(0,0;0,t)=w(t)$ for $t\in [0,1]$.
	Let $d_n$ be the metric obtained by planting straight lines from $(0,(k-1)2^{-n})$ to $(0,k2^{-n})$ with weight $w_k$ satisfying $w_k' := 2^n[w(k2^{-n})-w((k-1)2^{-n})]$. Then $d_n(0,0;0,k2^{-n})=w(k2^{-n})$ for all $k\in \{1,2\ldots,2^{n}\}$. Using a similar argument as in the proof of Proposition \ref{P:stline}, one can check that $I(d_n) \le I(e_n)$. We have 
	$$
	I(d_n)=\frac{4}{3} 2^{n/2}\sum_{k=1}^{2^n} [w(k2^{-n})-w((k-1)2^{-n})]^{3/2},
	$$
	and so $I(d_n) \to I(e_{\mu_*})$ as $n \to \infty$, and hence $I(e_{\mu_*}) \le I(e)$.
\end{example}

\subsection{Large deviations of the directed geodesic}

In this section, we present a large deviation principle for geodesics in the directed landscape. 
Let $u=(x,s;y,t)\in \Rd$, let $u_\ep=(x/\ep,s;y/\ep,t)$
and let $\gamma_{u_\ep}:[s, t] \to \R$ denote the (a.s.\ unique) $\cL$-geodesic with endpoints $u_\ep$ as an element of the space of continuous functions $C([s, t])$ from $[s, t] \to \R$ with respect to uniform convergence. 
\begin{theorem}\label{T:dg-ldp}
	For every Borel measurable set $A \subset C([s, t])$ we have
	\begin{equation*}   
		\exp((o(1)-\inf_{A^\circ} J_u)\e^{-3}) \le  P( \ga_{u_\e}\in A) \le \exp((o(1)-\inf_{\bar{A}}J_u)\e^{-3}).
	\end{equation*}
	Here $J_u:C([s, t]) \to [0, \infty]$ is a good rate function satisfying $J_u(f) = \infty$ when $f \notin H^1$ and
	\begin{align}\label{e:J-def}
		{J}_{u}(f) := \inf \big\{I(e) :  e \in \mathcal D_u(f)\big\}
	\end{align}
	for $f \in H^1$, where $\mathcal D_u(f) \subset \mathcal E$ is the set of metrics where $f$ is a geodesic with endpoints $u$. 
\end{theorem}
We will use the abbreviations ${J}={J}_{(0,0;0,1)}, \mathcal D(f) = D_{(0,0;0,1)}(f)$. Theorem \ref{T:dg-ldp} is essentially an application of the contraction principle from large deviation theory, but we need to be a bit careful since there is not a natural continuous map taking metrics to a geodesic from $(0,0) \to (0, 1)$. We postpone the proof to Section \ref{s:gamma-LDP}.  Here we first establish a few natural properties of the rate function ${J}$, and then use these to understand $J(f)$ for some specific examples.

\begin{lemma}\label{l:j-facts}
	\begin{enumerate}[label=\arabic*.]
		
		\item The infimum in \eqref{e:J-def} is always achieved by a metric $e \in \mathcal E$ whenever it is finite. Moreover, any metric achieving this infimum is given by planting the single path $f$ with some weight function $w$.
		\item ${J}_{(x, s;y, t)}(f)={J}(Af)$ where $Af:[0, 1] \to \R$ is given by
		$$
		Af(r) = (t-s)^{-1/3}(f(s + (t-s) r) - [x + r (y-x)]).
		$$
		In other words, $Af$ is the function given by transforming the plane under the landscape symmetry mapping $(x, s) \to (0,0), (y, t) \to (0,1)$.
		
		\item If $f(0)=f(1)=0, f \in H^1$ then ${J}(f)$ is the value of the optimization problem: \begin{align*}
			&\text{minimize } \qquad \tfrac{4}{3}\int_0^1 (w'(s)+f'(s)^2)^{3/2}\,ds, \\
			&\text{subject to }\;\; \qquad  w(t)-w(s)\ge-\frac{ (f(t)-f(s))^2}{t-s}, \quad 0\le s<t\le 1.      
		\end{align*}
		
		\item  Assume $f(0)=f(1)=0, f \in H^1$. Think of $f'$ as a random variable defined on $[0,1]$ with Lebesgue measure. Then ${J}(f)$ is the value of the optimization problem
		$$
		\text{minimize } \tfrac{4}{3}\Ex \rho^{3/2}\quad \text{subject to }\;\; \Ex \big[\rho\,\big|\,[s,t]\big] \ge \operatorname{Var}\!\big[f'\,\big|\,[s,t]\big] \quad 0\le s<t\le 1.
		$$
		over all random variables $\rho:[0,1] \to [0, \infty)$. The expectation and variance are conditioned on the event $[s,t]$.
		\item For $a\in \mathbb R$ we have ${J}(af)=|a|^3 {J}(f)$.
		
		\item Consider a finite or countable disjoint collection of subintervals $\{[a_i, b_i), i \in F\}$ of $[0, 1]$. 
		For any path $f$ from $(0,0)$ to $(0, 1)$ we have ${J}(f)\ge \sum_{i \in F} {J}_{p_i;q_i}(f|_{[a_i,b_i]})$ where $p_i=(a_i,f(a_i)),q_i=(b_i,f(b_i))$.
	\end{enumerate}
\end{lemma}

\begin{proof}
	For part $1$, let $e_n \in \mathcal D(f)$ be a sequence of metrics with $I(e_n) \downarrow {J}(f) < \infty$. Then since the sub-level sets of $I$ are compact, all the metrics $e_n$ are contained in a common compact set and so there exists a subsequential limit $e$. Since $I$ is lower semicontinuous, $I(e) \le {J}(f)$. From the definition of path length, it is easy to see that $f$ must also be a geodesic in $e$ and so $e \in \mathcal D(f)$ and hence $I(e)$ achieves the infimum \eqref{e:J-def}. Since $e$ is a finite rate metric, $e=e_{\mu}$ for some planted network measure $\mu$. Consider $\nu=\mu|_{\gr f}$, and observe that $e_{\nu} \in D(f)$. Therefore $I(e_{\nu}) = I(e)$ and so $e_{\nu} = e$ since $I$ is strictly monotone. Part $2$ follows from the symmetries of $\cL$ (see Lemma \ref{L:invariance}).
	
	Next, by part $1$ and the formula from \eqref{e:plantrate} it is enough to minimize
	$$
	\frac43\int_0^1(w'(t)+f'(t)^2)^{3/2}dt$$
	over all absolutely continuous weight functions with $w'\ge -f'^2$ that make $f$ a geodesic in $e$. Claim $3$ spells out this condition.  Note that $w'\ge -f'^2$ is implied by the constraint there. Part $4$ is just a reformulation of $3$ with $\rho=w'+f'^2$, and part $5$ follows from part $4$, since if $\rho$ solves the optimization problem in $4$ for $f$, then for any $a \in \R$, $a^2 \rho$ solves the optimization problem for $a f$.

	The inequality in $6$ follows by $3$. Indeed, the right hand side of the inequality can be written as the optimization problem
	$$
	\text{minimize } \sum_{i\in F} \tfrac{4}{3}\int_{a_i}^{b_i} (w'(s)+f'(s)^2)^{3/2}\,ds, \quad \text{subject to }\;\; w(t)-w(s)\ge-\frac{ (f(t)-f(s))^2}{t-s},
	$$
	where the constraint is only imposed for pairs $s, t$ with $a_i \le s < t \le b_i$ for some $i \in F$.
\end{proof}
\begin{lemma}[$L^{3/2}$ and $L^3$-norm bounds]\label{l:J-norms}
	We have
	\begin{equation}
		\label{E:lower-upper}
		\tfrac43 \left(\int_0^1 |f'(t)|^{2}dt\right)^{3/2}\le {J}(f)\le \tfrac43\int_0^1|f'|^3 dt.
	\end{equation}
	In particular ${J}(f)=\tfrac43\int_0^1|f'|^3 dt$ whenever $|f'|$ is constant almost everywhere.
	
	If we can find a finite or countable disjoint collection of sets $[a_i, b_i), i \in F$, whose union $\bigcup_{i \in F} [a_i, b_i)$, is a subset of $[0, 1]$ of Lebesgue measure $1$, such that $f(a_i) = f(b_i) = 0$ for all $i$, and such that
	$|f'|$ is constant Lebesgue-a.e.\ on each of the intervals $[a_i, b_i]$, then the upper bound in \eqref{E:lower-upper} is an equality.
	
	If the function $f'$, viewed as a random variable defined on $[0, 1]$ with Lebesgue measure satisfies the conditional variance bound
	$$
	\Ex  |f'|^2 = \operatorname{Var} (f') \ge \operatorname{Var}\!\big[f'\,\big|\,[s,t]\big], \quad 0\le s<t\le 1,
	$$
	then the lower bound in \eqref{E:lower-upper} is an equality.
\end{lemma}
\begin{proof}
	Taking $w=0$ in Lemma \ref{l:j-facts}.3 clearly satisfies the constraint and gives the present upper bound.
	
	The lower bound follows from Jensen's inequality and the bound $w(1) \ge d(0,0; 0, 1) = 0$:
	\begin{equation}
		\label{E:J-lower}
		\int_0^1(w'(t)+f'(t)^2)^{3/2}dt \ge \left(\int_0^1 w'(t)+f'(t)^2\,dt\right)^{3/2} \ge \left(\int_0^1 f'(t)^2dt\right)^{3/2}.
	\end{equation}
	In the case when $|f'|$ is constant, these two bounds agree, and we have identified ${J}(f)$.
	For the second claim, we have
	$$
	{J}(f) \ge \sum_{i\in F} {J}_{(0, a_i), (0, b_i)}(f|_{[a_i, b_i]}) = \sum_{i\in F} \frac{4}{3} \int_{a_i}^{b_i} |f'(t)|^3 dt = \frac43\int_0^1|f'|^3 dt,
	$$
	where the first bound uses Lemma \ref{l:j-facts}.6, and the first equality uses Lemma \ref{l:j-facts}.2 and the constant $|f'|$ case of the present lemma.
	For the final claim, if the conditional variance bound holds, then turning to Lemma \ref{l:j-facts}.4, we can set $\rho \equiv \operatorname{Var}(f') = \Ex |f'|^2$ to solve the optimization problem. The quantity $\tfrac{4}{3} \Ex  \rho^{3/2}$ equals the lower bound in \eqref{E:lower-upper}.
\end{proof}

The next lemma computes the rate function for a piecewise linear $f$ with two pieces. The result is somewhat surprising. 
\begin{lemma}\label{l:twopiece}
	Let $f_a(0)=f_a(1)=0$,  $f_a(a)=1$, and let $f_a$ be linear on $[0,a]$ and $[a,1]$.
	Then $J(f_a)=J(f_{1-a})$, and for $a\in [0,1/2]$ we have
	$J(f_a)=\frac{3-4 a^2}{6 (1-a)^3 a^2}$.
	\\In particular, $J(f_a)$ is not thrice differentiable  at  $a=1/2$.
\end{lemma}
\begin{proof}
	Let $a \le 1/2$ and set $b=(1/a+1/(1-a))^2$, the squared difference of slopes of the two pieces of $f$. 
	Applying the optimality condition of Lemma \ref{l:j-facts}.4  to intervals of the form $[0,t], t \in [2a, 1]$, we get
	\begin{equation}\label{e:looser-cond}
		\rho\ge 0, \qquad 
		\int_0^t (\rho-\lambda)(s)ds\ge 0, \qquad \lambda(t)= b(\tfrac{1}4 \wedge (a^2/t^2)).
	\end{equation}
	A simple computation shows that   $\rho=\lambda$ satisfies all the constraints of the optimization problem of Lemma \ref{l:j-facts}.4.
	Next we will solve the problem of minimizing $\tfrac43\int_0^1 \rho^{3/2}$ subject to the less restrictive conditions \eqref{e:looser-cond}.
	
	To show that a minimizer exists, let $\rho_n$
	be a sequence satisfying \eqref{e:looser-cond} so that $\int_0^1 \rho_n^{3/2}$ converges to the infimum $q$. Since $\rho_n$ has bounded $L^{3/2}$-norm, and hence bounded $L^{1}$-norm, it has an $L^1$-weakly convergent subsequence $\rho_n\to\rho$ with $\int_0^1 \rho^{3/2}\le q$ by Fatou's lemma. Weak convergence implies that \eqref{e:looser-cond} holds for $\rho$.
	
	We claim that $\rho\neq \lambda$ cannot be a minimizer. Consider $\rho$ satisfying \eqref{e:looser-cond} and let $f_\rho(t) = \int_0^t (\rho - \lambda)(s) ds$. Suppose that $f_\rho(s) = \de > 0$ for some $s \in (0, 1]$. Let 
	$$
	s_0 = \inf \{r \le s : f_\rho(r) \ge \de/2\}, \qquad s_1 = \sup \{r \ge s : f_\rho(r) \ge \de/2\}.
	$$
	Then there is a set $A \subset [s_0, s]$ of positive Lebesgue measure on which $\rho \ge \lambda + \de/4$. If $s_1 = 1$, then $\rho^* = \rho - (\de/4) \mathbf{1}_A$ still satisfies the conditions \eqref{e:looser-cond} since 
	$$
	f_{\rho^*}(u) \ge f_\rho(u) - \tfrac{\de}{4} \mathbf{1}(u \ge s_0) \ge 0,
	$$ and moreover $\int_0^1 (\rho^*)^{3/2} < \int_0^1 \rho^{3/2}$. Hence $\rho$ cannot be a minimizer. If $s_1 < 1$, then we can find another set $B \subset [s, s_1]$ of positive Lebesgue measure on which $\rho \le \lambda - \de/4$. By possibly reducing the size of either $A$ or $B$, we may assume that they have the same positive Lebesgue measure $|A| = |B|$. Let $\eta=\mathbf{1}_B-\mathbf{1}_A$. Then for all $r\in [0,\delta/4]$, $\rho+r\eta$ satisfies \eqref{e:looser-cond} since $|f_\rho(u) - f_{\rho+r\eta}(u)| \le r \mathbf{1}(u \in [s_0, s_1])$.
	On the other hand,
	$$
	\partial_r \int_0^1(\rho+r \eta)^{3/2}\Big\vert_{r=0}=\tfrac 32\int \eta\rho^{1/2}\le \tfrac32|A|\left(-(\lambda(s)+\delta)^{1/2}+(\lambda(s)-\delta)^{1/2}\right)<0,
	$$
	so for some small $r>0$, $\rho+r\eta$ is a better candidate than $\rho$. Thus $\rho=\lambda$ is the unique minimizer of the less restrictive problem \eqref{e:looser-cond}, and $\lambda$ satisfies all conditions of the stricter Lemma \ref{l:j-facts}.4. We compute the $L^{3/2}$-norm of $\lambda$ to get the claim.
\end{proof}
\begin{corollary}\label{c:geodesic-sup}
	Let $\pi$ be the $\cL$-geodesic from $(0,0)$ to $(0,1)$. Then as $a\to\infty$ $$P(\pi(1/2)\ge a)=e^{-\frac{32}3a^3+o(a^3)}, \qquad P(\sup_{[0,1]}|\pi|\ge a)=e^{-\frac{32}3a^3+o(a^3)}.$$
\end{corollary}
\begin{proof}
	Any function $f$ with $f(0)=f(1)=0$ and $\sup_{[0,1]}|f|\ge a$ must have 
	$$\int_0^1 |f'|^2 \ge (\int_0^1 |f'|)^2 = 4a^2
	$$ by Jensen's inequality. 
	The upper bounds on the above probabilities then follow from the left bound in \eqref{E:lower-upper}. The matching lower bound follows from the construction in Lemma \ref{l:twopiece}.
\end{proof}
\begin{remark}[Liu's conjecture]
	The first claim partially proves Conjecture 1.5  by \cite{liu2022one} for 
	the special case $t=1/2$. In forthcoming work, R. Basu and co-authors (\cite{basu2024personal}), show that the rest of the conjecture is not correct. They also give an independent proof of the $t=1/2$ case starting from (and extending to) exponential last passage percolation. Conjecture 1.5 of \cite{liu2022one} gives the correct value for having a zero-length path that takes the value $a$ at time $t$. When $t\neq1/2$, the zero-length path with the conjectured rate will not be a geodesic:  some shortcuts will have positive length.
	
	Heuristics using large deviations for the Airy process suggest  that geodesic under the  large deviation event $\pi(t)>a$ for $t\neq 1/2$ will not follow a piecewise linear path. In particular, it will not be $f_t$ from Lemma \ref{l:twopiece}! Instead, it will have two linear pieces and a parabola. Solving the optimization problem (or even describing the heuristic behind it) is beyond the scope of this paper, but we can state the resulting formula  as a new conjecture.
\end{remark}

\begin{conjecture}For $t\in (0,1/2]$,
	as $a\to\infty$ we have $P(\pi(t)\ge a)=e^{-\iota(t) a^3+o(a^3)}$ where
	$$\iota(t)=\frac{-(2t)^{5/2} (9 b+4)+6 t^2 (25 b+13)-2(2 t)^{3/2} (26 b+19)-48 \sqrt{2t}+24}{3 \left(3-\sqrt{8t}\right)^3 (1-t)^2 t^2},
	$$
	with 
	$$
	b=\frac{\sqrt{72 t^2+6 (2t)^{3/2}-143 t-12 \sqrt{2t}+72}}{(9-8 t) \sqrt{t}}.
	$$
	In particular, as $t\to 0$ we have $\iota(t)=\frac8{27}/t^{2}+o(1/t^2)$.
\end{conjecture}
The last part should recover the rate function for the semi-infinite geodesic, see \cite{rahman2021infinite} for the definition.

Lemma \ref{l:J-norms} allows us to construct classes of functions $f$ where either the upper or lower bound in \eqref{E:lower-upper} is satisfied. We finish this section by giving a few specific examples of such functions.

\begin{lemma}
	\label{L:jump-functions}
	Fix $\beta \in (0, 1/2)$ and $\al > 0$, and consider the continuous function $f$ which is linear on each of the intervals $[0, \beta], [\beta, 1 - \beta], [1 - \beta, 1]$ and satisfies $f(0) = f(1) = 0, f(\beta) = f(1-\beta) = \al$. Then if $\beta \ge 1/8$, ${J}(f)$ is given by the lower bound in \eqref{E:lower-upper}.
\end{lemma}

\begin{proof}
	By Lemma \ref{l:j-facts}.5, it suffices to prove this when $\al = 1$. Observe that $|f'(t)| = 1/\beta$ on $[0, \beta] \cup [1 - \beta, 1]$, whereas $f' = 0$ on $[\beta, 1-\beta].$ Therefore we can compute that
	$
	\operatorname{Var}(f') = 2/\beta.
	$
	On the other hand, for $[s, t]$ with either $s > \beta$ or $t < 1- \beta$, $f'$ is a Bernoulli-$p$ random variable multiplied by $\pm 1/\beta$ for some $p \in [0, 1]$. Then
	$$
	\operatorname{Var}\!\big[f'\,\big|\,[s,t]\big] \le \frac{1}{4 \beta^2},
	$$
	with equality if and only if $p = 1/2$. The condition $\beta \ge 1/8$ implies that this is less than or equal to $2/\beta$. Finally, if $s \le \beta < 1 - \beta \le t$, then
	\[
	\operatorname{Var}\!\big[f'\,\big|\,[s,t]\big] \le \Ex  \big[|f'|^2\,\big|\,[s,t]\big] \le \Ex  |f'|^2 = 2/\beta. 
	\]
	The lemma follows from the last claim of Lemma \ref{l:J-norms}.
\end{proof}

It is not difficult to construct other explicit examples where the lower bound is attained. One simple example is $f(x) = x(1-x)$. The next lemma gives an example where the upper bound is attained.
\begin{lemma}
	\label{L:l3l2}
	There exist functions $f$ with $f(0)=f(1)=1$, rate ${J}(f)=\infty$ and $\int_0^1f'(t)^2dt<\infty$.
\end{lemma}

\begin{proof}
	The following function $f$ is defined through the condition $f(0) = 0$ and its derivative. For $j \in \N$, set
	$$
	f'(s) = \begin{cases}
		j^{1/3}, \qquad &s \in \left[\frac{1}{j+1}, \frac{1 + 1/(2j)}{j+1}\right) \\
		-j^{1/3}, \qquad &s \in \left[\frac{1 + 1/(2j)}{j+1}, \frac{1}{j}\right).
	\end{cases}
	$$
	Then $\int_0^1 |f'|^2<\infty$.
	However, $|f'|$ is constant almost everywhere between consecutive zeros of $f$, and so by Lemma \ref{l:J-norms} we have ${J}(f)=\frac{4}{3}\int_0^1 |f'|^3=\infty$.
\end{proof}

Lemma \ref{L:l3l2} constructs an example of a function which will have finite length in any directed metric with finite rate, but will not be a geodesic in any of these metrics!

\section{Preliminaries}
\label{sec:basic}

In this section we list the preliminaries needed to prove the large deviation principle for $\cL$. These preliminaries also naturally suggest what the large deviation rate function should be, and so woven through this section will be a heuristic argument for our large deviation principle. We start by defining directed metrics.

\begin{definition}
	\label{D:directed-metric}
	A \textbf{directed metric of positive sign} on a set $X$ is a function $e:X \times X \to \R \cup \{\infty\}$ such that $e(x, x) = 0$ for all $x \in X$ and $e$ satisfies the triangle inequality
	$
	e(x, z) \le e(x, y) + e(y, z).
	$
	We say that $e:X \times X \to \R \cup \{-\infty\}$ is a \textbf{directed metric of negative sign} if $-e$ is a directed metric of positive sign.
\end{definition}
Section 5 of \cite{dv22} builds up a general theory of directed metrics. They are a natural generalization of metrics, when the symmetry condition $e(x, y) = e(y, x)$ is removed. In our context, any function $e:\Rd \to \R$ in the set $\mathcal E$ can be extended to a directed metric (of negative sign) on all of $\R^2$ by setting $e(p; p) = 0$ for all $p \in \R^2$ and $e(p; q) = -\infty$ when $p \ne q$ and $(p; q) \notin \Rd$. Because of this, we can always think of $\mathcal E$ as the set of continuous directed metrics on the space-time plane, and we refer to elements of $\mathcal E$ simply as metrics. The path length formula \eqref{E:ga-length} for $e \in \mathcal E$ is a special case of the usual formula for the length of a path in a (directed) metric.

The directed landscape is a random directed metric. It is built from the Airy sheet $\S$, which is a random continuous function $\S:\R^2 \to \R$, defined precisely, for example, in \cite{dv22}, Definition 1.22.   In this paper we take the law of the Airy sheet as a black box as we require only a few properties of the object. We call 
$$
\S_\sigma(x, y) := \sigma \S(x/\sigma^2, y/\sigma^2)
$$
an \textbf{Airy sheet of scale $\sigma$}.
\begin{definition}
	\label{D:L-unique}
	The directed landscape $\cL:\Rd \to \R$ is the unique random continuous function satisfying
	\begin{enumerate}[label=\Roman*.]
		\item (Airy sheet marginals) For any $t\in \R$ and $s>0$ we have
		$$
		\mathcal{\cL}(x, t; y,t+s^3) \eqd \S_s(x, y)
		$$
		jointly in all $x, y$. That is, the increment over time interval $[t,t+s^3)$ is an Airy sheet of scale $s$.
		\item (Independent increments) For any disjoint time intervals $\{[t_i, s_i] : i \in \{1, \dots k\}\}$, the random functions
		$
		\{\cL(\cdot, t_i ; \cdot, s_i) : i \in \{1, \dots, k \}\}
		$
		are independent.
		\item (Metric composition law) Almost surely, for any $r<s<t$ and $x, y \in \R$ we have that
		\begin{equation}
			\label{E:MC-law}
			\cL(x,r;y,t)=\max_{z \in \mathbb R} [\cL(x,r;z,s)+\cL(z,s;y,t)].
		\end{equation}
	\end{enumerate}
\end{definition}
The triangle inequality for $\cL$ is equivalent to the weaker claim that LHS$\ge$RHS in the metric composition law \eqref{E:MC-law}. The fact that the metric composition law is an equality implies that $\cL$ defines a geodesic space: almost surely, for any point pair $(p; q) \in \Rd$, we have
$$
\cL(p; q) = \max_{\ga:p \to q} |\ga|_\cL.
$$
For fixed $(p; q)$, this maximum is almost surely uniquely attained (i.e. there is a unique geodesic). Existence and uniqueness of $\cL$-geodesics are shown in \cite{DOV18}, Theorem 12.1 and Lemma 13.2.

Like many scaling limits, the directed landscape satisfies many distributional symmetries, which we use throughout.

\begin{lemma} [Lemma 10.2, \cite{DOV18}]
	\label{L:invariance} We have the following equalities in distribution as random continuous functions from $\Rd \to \R$. Here $r, c \in \R$, and $q > 0$.
	\begin{enumerate}[label=\arabic*.]
		\item (Time stationarity)
		$$
		\displaystyle
		\cL(x, t ; y, t + s) \eqd \cL(x, t + r ; y, t + s + r).
		$$
		\item (Spatial stationarity)
		$$
		\cL(x, t ; y, t + s) \eqd \cL(x + c, t; y + c, t + s).
		$$
		\item (Shear stationarity)
		$$
		\cL(x, t ; y, t + s) \eqd \cL(x + ct, t; y + ct + sc, t + s) + s^{-1}[(x - y)^2 - (x - y - sc)^2].
		$$
		\item (KPZ rescaling)
		$$
		\cL(x, t ; y, t + s) \eqd  q \cL(q^{-2} x, q^{-3}t; q^{-2} y, q^{-3}(t + s)).
		$$
	\end{enumerate}
\end{lemma}

By Lemma \ref{L:invariance}, we can observe that for any fixed $u = (x, s; y, t) \in \Rd$, the law of the random variable $\cL(u)$ is simply a shifted and rescaled version of the law of $\cL(0,0; 0, 1)$:
\begin{equation}
	\label{E:one-point}
	\cL(u) \eqd (t-s)^{1/3}\cL(0,0; 0, 1) - \frac{(x-y)^2}{t-s}.
\end{equation}
Equation \eqref{E:one-point} provides a good way of thinking about $\cL$: it consists of the `Dirichlet part'
$$
d(x, s; y, t) = - \frac{(x-y)^2}{t-s},
$$
and a noise part, which consists of a Tracy-Widom random variable. One way of thinking about the rescaling $\cL_\ep$ is that it provides one way of scaling down the strength of the noise part to have size $\ep$. Indeed, at the level of one-point distributions, by Lemma \ref{L:invariance} we have
\begin{equation}
	\label{E:one-point-law}
	\cL_\ep(u) \eqd \ep (t-s)^{1/3}\cL(0,0; 0, 1) + d(x, s; y, t)
\end{equation}
In order to understand the large deviation behaviour of the whole directed landscape, we should first try to understand the large deviation behaviour at the level of single points. By \eqref{E:one-point}, it suffices to understand the tails of the Tracy-Widom random variable $\cL(0,0; 0, 1)$.
\begin{theorem}[see Theorem 1.3 in \cite{ramirez2011beta}]
	\label{T:tracy-widom-tails}
	Let $X = \cL(0,0; 0, 1)$. Then we have the following asymptotics as $m \to \infty$:
	$$
	\Pr (X > m) = e^{-[\frac{4}{3} + o(1)] m^{3/2}}, \qquad \Pr (X < -m) = e^{-[\frac{1}{12} + o(1)] m^{3}}.
	$$
\end{theorem}
From Theorem \ref{T:tracy-widom-tails} we can observe that it is much easier to make the directed landscape large at a single point than to make it small. Heuristically, this is fairly easy to see if we think of $\cL$ as a path metric: to make $\cL(0,0; 0, 1)$ large we need to plant a \textit{single good path}, whereas to make $\cL(0,0; 0, 1)$ small we need to make \textit{all paths bad}. This phenomena extends all the way through to our large deviation principle for $\cL$. Indeed, at the $\ep^{3/2}$-scale our rate function will assign infinite rate to any metric $e$ with $e(u) < d(u)$ for some $u \in \Rd$, and will assign finite rates only to metrics than can be achieved by planting countably many paths.

Let us try to describe this more precisely, in a way that derives the rate function $I$. Let $e$ be any directed metric on $\Rd$ with $e \ge d$. We will explore the possibility that $\cL_\ep$ can be close to $e$ by looking at the behaviour of $\cL_\ep$ restricted to a single path $\ga:[s, t] \to \mathbb R$. For $r < r' \in [s, t]$ let $u_{r, r'} = (\ga(r), r; \ga(r'), r')$, and
consider the event where for fixed $\de > 0$, we have
\begin{align}
	\label{E:weight-de}
	|\cL_\ep(u_{r, r'}) - e(u_{r, r'})| < \de
\end{align}
for all $s \le r < r' \le t$. In other words, this is the event where $\cL_\ep$ is close to $e$ along $\ga$. Now, by Theorem \ref{T:tracy-widom-tails} and \eqref{E:one-point-law}, the probability of \eqref{E:weight-de} for fixed $r, r'$ is given by
$$
\exp \left(-\frac{4}{3}\ep^{3/2}\Theta(e, u_{r, r'}) + O(\de \ep^{3/2} (r' - r)^{-1/2}) \right),
$$
where for $u = (x, s; y, t)$,
\begin{equation}
	\label{e:dThetaeu}
	\Theta(e, u) := \frac{[e(u)-d(u)]_+^{3/2}}{(t-s)^{1/2}}=\left[\frac{e(u)}{t-s} + \frac{(x-y)^2}{(t-s)^2}\right]_+^{3/2}(t-s).
\end{equation}
Moreover, by the temporal independence in $\cL$ (Property II of Definition \ref{D:L-unique}), the events in \eqref{E:weight-de} are independent for disjoint intervals $[r, r']$, and so the probability that \eqref{E:weight-de} holds for all consecutive pairs $r, r'$ on a partition $r_0 < r_1 < \dots < r_k$ of $[s, t]$ is given by
$$
\exp \left(-\frac{4}{3}\ep^{3/2}\sum_{i=1}^k \Theta(e, u_{r_{i-1}, r_i}) + O(k \de \ep^{3/2} \max_{1 \le i \le k} (r_{i} - r_{i-1})^{-1/2}) \right),
$$
which for a fine enough partition, $\de$ sufficiently small, and $\ga, e$ sufficiently nice is well approximated by $e^{- \ep^{3/2}I(\ga, e)}$, where
$$
I(\ga, e) = \frac{4}{3} \int_s^t (w'(r) + [\ga'(r)]^2)^{3/2}, \qquad w(r) = \|\ga|_{[s, r]}\|_e.
$$
This computation only gives a lower bound on the chance that $\cL_\ep$ is close to $e$ since we have only compared the metrics along the single path $\ga$. This gives a lower bound on $I(e)$. The general form of the rate function in Theorem \ref{t:main} comes from comparing the metrics on arbitrary collections of disjoint paths:
$$
I(e) = \sup_{\ga_1, \dots, \ga_k} \sum_{i=1}^k I(\ga_i, e),
$$
where the supremum is over all finite collections of disjoint paths $\ga_1, \dots, \ga_k$. The fact that the rate function is additive over disjoint paths loosely follows from the fact that the landscape $\cL$ is derived as a limit of last passage models built on independent noise, so the behaviour of $\cL$ along two disjoint paths should be independent. To prove this, we will need an asymptotic independence proposition for the Airy sheet.

\subsection{Quantitative approximate independence in the Airy sheet}

\begin{proposition}
	\label{l:4.4} Fix $\Delta>0$. Fix any $a_1<b_1<a_2<b_2<\cdots<a_k<b_k$ and $c_1<d_1<c_2<d_2<\cdots<c_k<d_k$, and suppose that
	\begin{align}
		\label{e:cond}
		c_{i+1} - d_i > 2 \Delta, \qquad a_{i+1} - b_i > 2 \Delta
	\end{align}
	for all $i \in \{1, \dots, k-1\}$.
	For each $i = 1, \dots, k$, let $A_i$ be a Borel subset of $C([a_i,b_i]\times [c_i,d_i])$. Then
	\begin{equation}
		\label{e:l4.4}
		\begin{aligned}	\left|\Pr\left(\bigcap_{i=1}^k\{\S|_{[a_i,b_i]\times[c_i,d_i]}\in A_i\}\right)-\prod_{i=1}^k\Pr\left(\S|_{[a_i,b_i]\times[c_i,d_i]} \in A_i\right)\right|&\le 2k \Pr(\|\Pi\|_\infty \ge \Delta) \\&\le 4 k e^{- c' \Delta^3},
		\end{aligned}
	\end{equation}
	where
	$\Pi$ is the a.s.\ unique $\cL$-geodesic in $\cL$ from $(0,0)$ to $(0, 1)$, and $c' > 0$ is an absolute constant.
\end{proposition}
We refer to Figure \ref{fig:my_label3} for a visualization of the above lemma.

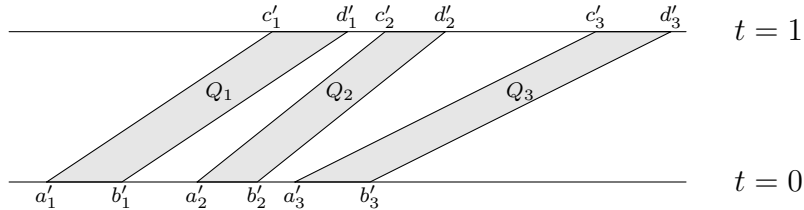
\begin{figure}[h!]
	\centering
	\captionsetup{width=0.9\linewidth}
	\begin{tikzpicture}[line cap=round,line join=round,>=triangle 45,x=1cm,y=2cm]
		\draw[fill=gray!20!white] (-0.5,0)--(0.5,0)--(1.5+2,1)--(0.5+2,1)--(-0.5,0);
		\draw[fill=gray!20!white] (1.5,0)--(2.3,0)--(2.8+2,1)--(2+2,1)--(1.5,0);
		\draw[fill=gray!20!white] (2.8,0)--(3.8,0)--(5.8+2,1)--(4.8+2,1)--(2.8,0);
		\node at (9.1,1.01) {$t=1$};
		\node at (9.1,0.01) {$t=0$};
		\draw (-1,0)--(8,0);
		\draw (-1,1)--(8,1);
		\begin{scriptsize}
			\node at (-0.5,-0.1) {$a_1'$};
			\node at (0.5,-0.1) {$b_1'$};
			\node at (1.5,-0.1) {$a_2'$};
			\node at (2.3,-0.1) {$b_2'$};
			\node at (2.8,-0.1) {$a_3'$};
			\node at (3.8,-0.1) {$b_3'$};
			\node at (2.5,1.1) {$c_1'$};
			\node at (3.5,1.1) {$d_1'$};
			\node at (4,1.1) {$c_2'$};
			\node at (4.8,1.1) {$d_2'$};
			\node at (6.8,1.1) {$c_3'$};
			\node at (7.8,1.1) {$d_3'$};
			\node at (1.8,0.6) {$Q_1$};
			\node at (5.8,0.6) {$Q_3$};
			\node at (3.4,0.6) {$Q_2$};
		\end{scriptsize}
	\end{tikzpicture}
	\caption{By Condition \eqref{e:cond} the open quadrilaterals $Q_1$, $Q_2$, $Q_3$ are disjoint.}
	\label{fig:my_label3}
\end{figure}

\begin{proof}[Proof of Proposition \ref{l:4.4}] The second bound in \eqref{e:l4.4} follows from Proposition 12.3 in \cite{DOV18}, so it suffices to show the first bound. The main strategy behind the proof of Proposition \ref{l:4.4} is to construct a coupling of $k+1$ landscapes $(\L_0,\L_1,\ldots,\L_k)$ such that $(\L_i)_{i=1}^k$ are all independent and with a high explicit probability $\L_{\mx}(x,0;y,1)=\L_i(x,0;y,1)$ for all $x\in [a_i,b_i]$ and $y\in [c_i,d_i]$. This is a more quantitative variant of the proof of Proposition 2.6 in \cite{dauvergne2022non}.
	
	To do this, we first construct a coupling of copies of exponential last passage percolation. Let $(\xi_{v}^{(i)})_{v\in \Z^2, i=1,\ldots,k}$ be collections of independent rate-one exponential random variables. For $i = 1, \dots, m$ let
	$a_i'=a_i-\Delta, b_i'=b_i+\Delta, c_i'=c_i-\Delta, d_i'=d_i+\Delta$, and let
	$Q_i$ be the open parallelogram with vertices
	$$
	(a_i',0), \qquad (b_i', 0), \qquad(c_i', 1), \qquad (d_i', 1),
	$$
	see Figure \ref{fig:my_label3}. Due to the condition \eqref{e:cond}, the parallelograms $Q_i$ are disjoint (see Figure \ref{fig:my_label3}). For $n \in \N$, define $A_n:\R^2 \to \Z^2$ by letting $A_n(x, s) = (\lfloor xn^{2/3}+sn \rfloor, \lfloor sn \rfloor)$, and let
	\begin{align*}
		\xi_{v}^{(0)}(n):=\begin{dcases}
			\xi_v^{(i)} & \mbox{ if }v\in A_n(Q_i)  \mbox{ for some }i=1,2,\ldots,k \\
			\xi_v^{(1)} & \mbox{ if }v\notin \bigcup_{i=1}^k A_n(Q_i) .
		\end{dcases}
	\end{align*}
	Loosely speaking, $\xi_v^{(0)}(n)$ is formed by `stitching' together different collections of exponential weights. The linear scaling operator $A_n$ is chosen to set up convergence to the directed landscape.

	
	We now consider $k+1$ copies of exponential last passage percolation using $k+1$ collection of weights: $(\xi_v^{(\mx)}(n))_{v\in \Z^2}$ and $(\xi_v^{(i)})_{v\in \Z^2}$ where $i=1,2,\ldots,k$. For any points $v_1,v_2\in \Z^2$ with $v_{1, i} \le v_{2, i}$ for $i = 1, 2$, let
	$
	T_n^{(i)}(v_1,v_2)
	$
	denote the last passage value from $v_1$ to $v_2$ computed using the $i$-th collection of exponential weights for $i=0,1,\ldots,k$. That is,
	\begin{equation}
		\label{E:T-def}
		T_n^{(i)}(v_1,v_2)  = \max_{\pi:v_1 \to v_2} \sum_{j=0}^{\|v_1 - v_2\|_1} \xi_{\pi_j}^{(i)}(n)
	\end{equation}
	where $\xi_{\pi_j}^{(i)}(n) = \xi_{\pi_j}^{(i)}$ for $n \ge 1$, and the maximum is over all paths $\pi = (\pi_0, \dots, \pi_{\|v_1 - v_2\|_1})$ with $\pi_0 = v_1, \pi_{\|v_1 - v_2\|_1} = v_2$ and $\pi_i - \pi_{i-1} \in \{(1, 0), (0, 1)\}$ for all $i = 1, \dots, \|v_1 - v_2\|_1$. Almost surely, for all $v_1, v_2$ there is a unique geodesic $\pi^{(i)}(v_1, v_2)$ achieving the maximum in \eqref{E:T-def}. Set
	\begin{align*}
		L_n^{(i)}(x,s;y,t) & := 4n^{-1/3}[T_n^{(i)}(A_n(x, s) ; A_n(y, t)) -4n(t-s)-2n^{2/3}(y-x)].
	\end{align*}
	By Theorem 1.7 in \cite{dv22}, for each $i\in \{0,1,\ldots,k\}$, $L_n^{(i)} \stackrel{d}{\to} \L$, uniformly on compact subsets of $\Rd$. Since the prelimits are all defined on the same probability space, along a subsequence $(n_r)_{r=1}^\infty$ we have $$(L_{n_r}^{(0)}, L_{n_r}^{(1)},\ldots, L_{n_r}^{(k)}) \stackrel{d}{\to} (\mathcal{L}_0,\mathcal{L}_1,\ldots,\mathcal{L}_k).$$
	where each $\L_{i}$ is a directed landscape and the last $k$ copies, $\L_1,\L_2,\ldots,\L_k$ are independent. Letting $\S_i = \L_i(\cdot, 0; \cdot, 1)$, the left-hand side of \eqref{e:l4.4} is then equal to
	$$	
	\left|\Pr\left(\bigcap_{i=1}^k\{\S_0|_{[a_i,b_i]\times[c_i,d_i]}\in A_i\}\right)-\Pr\left(\bigcap_{i=1}^k\{\S_i|_{[a_i,b_i]\times[c_i,d_i]}\in A_i\}\right)\right|,
	$$	
	which by a union bound is bounded above by
	$
	\sum_{i=1}^k \Pr(\S_0|_{[a_i,b_i]\times[c_i,d_i]} \ne \S_i|_{[a_i,b_i]\times[c_i,d_i]}).
	$
	Therefore letting $\S^{(i)}_n$ denote the prelimiting version of $\S_i$, to finish the proof, it suffices to show that for all $i = 1, \dots, k$ we have
	\begin{equation}
		\label{E:n-infty}
		\limsup_{n \to \infty} \Pr(\S^{(0)}_n|_{[a_i,b_i]\times[c_i,d_i]} \ne \S^{(i)}_n|_{[a_i,b_i]\times[c_i,d_i]}) \le  2\Pr(\|\Pi\|_\infty \ge \Delta).
	\end{equation}
	We will have $\S^{(i)}_n|_{[a_i,b_i]\times[c_i,d_i]} \ne \S^{(0)}_n|_{[a_i,b_i]\times[c_i,d_i]}$ if and only if one of geodesics $\pi^{(i)}(A_n(x, 0); A_n(y, 1))$ or $\pi^{(0)}(A_n(x, 0); A_n(y, 1))$ for $(x, y) \in [a_i,b_i]\times[c_i,d_i]$ exits the set $A_n(Q_i)$. By planarity, these geodesics satisfy the following ordering property: all of the geodesics $\pi^{(i)}(A_n(x, 0); A_n(y, 1))$ for $(x, y) \in [a_i,b_i]\times[c_i,d_i]$ are contained in the subset of the strip $\R \times [0, n]$ bounded on the left and right by $\pi^{(i)}(A_n(a_i, 0); A_n(b_i, 1))$ and $\pi^{(i)}(A_n(c_i, 0); A_n(d_i, 1))$ respectively, and similarly for the geodesics $\pi^{(0)}$. Therefore \eqref{E:n-infty} is implied by claim that for $j = 0, i$ we have
	\begin{equation}
		\label{E:n-infty'}
		\limsup_{n \to \infty} \Pr(\pi^{(j)}((a_i, 0)_n; (b_i, 1)_n) \cup \pi^{(j)}((c_i, 0)_n; (d_i, 1)_n) \not \subset A_n(Q_i)) \le \Pr(\|\Pi\|_\infty \ge \Delta).
	\end{equation}
	Using Theorem 1.7/1.8 in \cite{dv22} for $j = 0, i$, the rescaled sets
	$$
	A_n^{-1}[\pi^{(j)}((a_i, 0)_n; (b_i, 1)_n)] \cup A_n^{-1}[\pi^{(j)}((c_i, 0)_n; (d_i, 1)_n)\}]
	$$
	converge in law as $n \to \infty$ in the Hausdorff topology on compact sets to the set
	$$
	\mathfrak{g} \Pi_1 \cup \mathfrak{g} \Pi_2 :=\{(\Pi_1(r), r) : r \in [0, 1]\} \cup \{(\Pi_2(r), r) : r \in [0, 1]\},
	$$
	where $\Pi_1$ is the a.s.~unique geodesic from $(a_i,0)$ to $(c_i,1)$ in $\L$, and $\Pi_2$ is the a.s. unique geodesic from $(b_i,0)$ to $(d_i,1)$. Therefore the limsup in \eqref{E:n-infty} is bounded above by
	$
	\Pr( \mathfrak{g} \Pi_1 \cup \mathfrak{g} \Pi_2 \not \subset Q_i).
	$
	Noting that $\Pi_1 \le \Pi_2$, and letting $L_1(t) = b_i' t + a_i'(1-t)$ and $L_2(t) = c_i' t + d_i'(1-t)$, we have that
	\begin{align*}
		\Pr( \mathfrak{g} \Pi_1 \cup \mathfrak{g} \Pi_2 \not \subset Q) \le &\;\Pr( \Pi_1(r)\le L_1(r) \text{ for some }r\in[0,1]) \\ &+  \Pr(L_2(r)\le \Pi_2(r) \text{ for some }r\in[0,1])\le 2 \Pr(\|\Pi\|_\infty \ge \Delta),
	\end{align*} where the final inequality uses the shear and translation invariance of $\cL$ (Lemma \ref{L:invariance}), to translate to statements about geodesics from $(0,0)$ to $(0, 1)$. This yields \eqref{E:n-infty'}, completing the proof. 
\end{proof}

\begin{remark} The second bound in \eqref{e:l4.4} where we have appealed to Proposition 12.3 in \cite{DOV18} uses more than one-point bounds.  However, we note in passing that the argument for showing $A_\e' \subset A_\e$ in Proposition \ref{P:landscape-in-E*} yields a proof of the geodesic tail bound needed here. This proof relies only on Propositions \ref{p:d-close}, \ref{p:d-dominant}, which we build from the one-point bounds in Theorem \ref{T:tracy-widom-tails}. Note that a fortiori, Corollary \ref{c:geodesic-sup} determines the optimal rate in the stretched exponential tail bound for $\|\Pi\|_\infty$. 
\end{remark}

\begin{proposition}[Airy sheet tails]\label{p:Airy-tails}
	There is a universal constant $c_0 > 0$ so that the following holds. 
	Let $s<t$, $\Delta>0$ and $U=\{(x_i,s;y_i,t):i=1,\ldots,k\}$ with $x_i+\Delta\le x_{i+1}$ and $y_i+\Delta\le y_{i+1}$ for all $i$. Let  $r:U \to \mathbb R$ with $r\ge d$.
	Let $\theta=\sum_{u\in U}\Theta(r,u)$. Then as
	as $\e\to 0$ we have
	\begin{align*}
		P(\mathcal L_\e|_U\ge r)=
		\exp((o(1)-\tfrac43\theta)
		\e^{-3/2}), \qquad \text{ if }\theta< c_0\frac{\Delta^3}{(t-s)^2}.
	\end{align*}
\end{proposition}
\begin{proof}
	Let $\Delta_0=\Delta/((t-s)^{2/3}\ep^{1/2}).$ 
	By Proposition \ref{l:4.4} and scaling,
	$$
	\bigg|P(\mathcal L_\e|_U\ge r)-\prod_{u\in U}P(\mathcal L_\e(u)\ge r(u))\bigg|\le 4k\exp(-c\Delta_0^3)=
	\exp\bigg((o(1)-c\frac{\Delta^3}{(t-s)^2})\e^{-3/2}\bigg).$$
	The claim follows if we bound the product using the Tracy-Widom upper tail bound in Theorem \ref{T:tracy-widom-tails} and scaling. 
\end{proof}

\subsection{Neighborhood bounds}

The goal of this section is to use to go from pointwise bounds to uniform bounds over compact or bounded subsets. The methods here are standard: the lower bound follows by a quick chaining argument, while the upper bound follows from the lower bound and the triangle inequality.  Let $u_0=(0,0;0,1)$.

\begin{proposition}\label{p:neighborhood-l}
	With a universal $c_0>0$,
	for every $\eta\in (0,1/12)$ there is $c_\eta>0$ so that
	$$
	P\bigg(\inf_{u:\|u-u_0\|_\infty\le \ep}\mathcal L(u)<-a\bigg)\le c_\eta \exp(-a^3(\eta-c_0\e^{1/3})), \qquad \text{ for all } a>0, \ep\in(0,1/5].
	$$
\end{proposition}

\begin{proof}
	This is a standard chaining argument. We specify a nice countable set $U\subset \Rd$ so that  a dense set of point pairs in the domain of the $\inf$ above can be connected by some concatenation of $d$-geodesics for $u\in U$.  Then we use the union bound and the one-point bound to show that elements of $U$ are unlikely to have a small $\mathcal L$-value. 
	
	To define the point pairs, let $\mathbb D=\{i/2^k\in [0,1) :i,k\in \Z\}$ denote the dyadic rationals in $[0,1)$. For $x\in \mathbb D$ let $0.x_1x_2\ldots$ denote its binary expansion with  $x_i=0$ eventually. Define $s(x)$ so that $x_{s(x)}$ is the last nonzero bit in the expansion. 
	Define the map
	$$\alpha:\mathbb D^2\setminus\{(0,0)\}\to \mathbb D^2, \quad 
	\alpha(x,y)=(0.x_1\dots x_\sigma, 0.y_1\dots y_\sigma), \quad \sigma=\max(s(x),s(y))-1.
	$$
	Iterating $\alpha$ will eventually take any point in $\mathbb D$ to $(0,0)$. This is how we will build our chains   up to simple similarity  transformations $\varphi_i$. 
	
	Next, we define two similarity transformations $\varphi_i:\mathbb R^2\to\mathbb R^2$, $i=0,1$ by the following properties. Let $R_\theta:\R^2 \to \R^2$ denote the counterclockwise rotation by $\theta$, and set $$\varphi_0(x) = \sqrt{2}\ep R_{-3\pi/4} x + (0, \ep), \qquad \varphi_1(x) = \sqrt{2}\ep R_{\pi/4} x + (0, 1- \ep).
	$$
	In other words, $\varphi_i$ takes $[0,1]^2$ to the closed $L^1$-ball of radius $\e$ about $(0,i)$,  $\varphi_0(0,0)=(0,\e)$, and $\varphi_1(0,0)=(0,1-\e)$.  Let $u_1=(0,\e;0,1-\e)$, and set
	$$U=
	\Big\{(\varphi_0(p) ; \varphi_0(\alpha(p))), 
	p\in \mathbb D^2\Big\} \cup \{u_1\} \cup 
	\Big\{(\varphi_1(\alpha(p)) ; \varphi_1(p)), 
	p\in \mathbb D^2\Big\}, 
	$$
	and let $U_k=\{(x,s;y,t)\in U:t-s=\e 2^{-k/2}\}$. 
	By construction, we have the following:
	$$
	\{u_1\} \cup \bigcup_{k\ge 0}U_k=U\subset \Rd, \quad |U_k|\le 2^{2k+1}, \quad \frac{|x-y|}{t-s}\in \{0,1\} \text{ for all }(x,s;y,t)\in U.
	$$
	Let $A_i=\varphi_i(\mathbb D^2)$. For every $p\in A_0$, iterating $\varphi_0\circ \alpha \circ \varphi_0^{-1}$ we get a sequence $p=p_0,p_1,\ldots, p_\ell=(0,\e)$ so that $(p_i;p_{i+1})\in U$. Similarly, for every $q\in A_1$, there is a sequence 
	$(0,1-\e)=q_0,\ldots, q_\ell=q$ so that $(q_i;q_{i+1})\in U$. Concatenating the two sequences to connect $p$ to $q$, and using the triangle inequality we get
	$$
	\inf_{A_0\times A_1}  \L\ge \L(u_1)+2b+ 2S, \quad S=\sum_{k=0}^\infty \min_{U_k}(\L-d),\quad 
	b=\sum_{k=0}^\infty \min_{U_k}d=\frac{-\e}{1-2^{-1/2}}> -4\e. $$
	where  $d$ is the Dirichlet metric \eqref{e:landscape-Dirichlet}. By the union bound and scaling properties  of the directed landscape (Lemma \ref{L:invariance}) for any  sequence $\beta_k\ge 0$ with total sum $\beta$ we have
	$$
	P(S\le-a\beta \ep^{1/3})\le \sum_{k=0}^\infty 2^{2k+2}P(\mathcal L(u_0)\le -a \beta_k2^{k/6}).
	$$
	Without loss of generality, let $a\ge 2$ and let $\beta_k=2^{2-k/12}$. Then 
	$$
	P(S\le-a\beta \ep^{1/3})\le c'_\eta e^{-\eta a^3}\sum_{k=0}^\infty 2^{2k+2}\exp(-7\eta \beta_k^3 2^{k/2})=c''_\eta e^{-\eta a^3}
	$$
	by the lower tail bound of Theorem \ref{T:tracy-widom-tails}. By the same bound again, 
	\begin{align*}
		P(\inf_{A_0\times A_1} \L \le -a) &\le P(\mathcal \L(u_1)\le -a(1-2\beta\e^{1/3})+8\e)+P(S\le -a\beta \e^{1/3})\\&\le c'_\eta e^{-\eta a^3(1-2\beta \e^{1/3}-8\ep)^3}+c''_\eta e^{-\eta a^3}
	\end{align*}
	Since the closure of $A_0\times A_1$ contains all $u$ with $\|u-u_0\|_\infty\le \e/2$, the claim follows. 
\end{proof}

We will now use the lower bound and the triangle inequality to get an upper tail bound. Let $u_0=(0,0;0,1)$.
\begin{proposition}\label{p:neighborhood-u}
	With a universal $c_0>0$,
	for every $\eta\in (0,4/3)$ there is $c_\eta>0$ so that
	$$
	P\Big(\sup_{u:\|u-u_0\|_\infty\le \e}\L(u)>a\Big)\le c_\eta \exp(-a^{3/2}(\eta-c_0\e^{1/3})), \qquad \text{ for all } a>0, \ep\in(0,1/5].
	$$
\end{proposition}
\begin{proof}
	Let $S$ denote the $\sup$ above, let $b>0$ large to be chosen later, and $p_1=(0,-b^3\e),q_1=(0,1+b^3\e)$. By the triangle inequality, we have
	$$
	S \le \L(p_1;q_1)-S_0-S_1, 
	\quad S_0=\inf_{p:\|p\|_{\infty}\le\e} \L(p_1;p), \quad S_1= \inf_{q:\|q-(0,1)\|_{\infty}\le\e} \L (q;q_1). 
	$$
	By the symmetries of the directed landscape (Lemma \ref{L:invariance}), after scaling time by $b^{-3}\e^{-1}$ and shifting, we have
	$$
	S_1\eqd S_0\eqd b\e^{1/3}\inf_{|x|\le \e^{1/3}/b^2 ,|s-1|\le 1/b^3} \L(0,0;x,s). 
	$$
	For $b$  large enough, by the lower tail bound, Proposition \ref{p:neighborhood-l}, we have $P(S_0<-r)\le c\exp(-r^3/(30b^3\e))$ all $r>0,\e\in(0,1/5]$. With $r=40^{1/3}b a^{1/2}\e^{1/3}$, a union bound gives
	$$P(S\ge a)\le P(\L(p_1;q_1)>a(1-2b^{3}40^{1/3}  \e^{1/3}a^{-1/2}))+2ce^{-(4/3)a^{3/2}}
	$$
	With large enough $b$, for $a\ge 1$ the claim now follows from the one-point upper tail bound of Theorem \ref{T:tracy-widom-tails}. The $a\le 1$ case will hold automatically if we  properly  adjust $c_\eta$.
\end{proof}

\subsection{Efficient covers}

The neighborhood bounds in the previous section will lead to tightness bounds for $\L_\e$ over bounded sets, as long as such sets can be covered efficiently with translates of a neighborhood under the symmetry  group of the landscape.
The last claim \eqref{e:cover-open} of the next lemma does show exactly this. The lemma uses the compact sets 
\begin{equation}\label{e:tildeBn}
	\tilde B_n=\{u = (x, s; y, t) \in \Rd : t -s \ge 1/n, \|u\|_\infty \le n\}, \qquad \bigcup_{n\in \mathbb N} \tilde B_n =\Rd,
\end{equation}
and the boxes $\Lambda_{a, b} = ([-a, a] \times [-b, b])^2$.
\begin{lemma}\label{l:cover}
	Let
	$
	\alpha_{i,j,k,\ell}$ denote the map $(x,s)\mapsto (2^{-2\ell}(x+i+k s),2^{-3\ell}(s+j))$, and for $(p; q) \in \Rd$ and $n \in \Z^4$ set $\alpha_n(p ;q)=(\alpha_n(p);\alpha_n(q))$ for the diagonal map on $\Rd$. Then there exist absolute constants $r_0,c_0>0$ so that 
	\begin{equation}\label{e:cover-full}
		\bigcup_{n\in \mathbb Z^4} \alpha_n(\tilde B_{r_0})=\Rd,
	\end{equation}
	and for all  $\ell\in \mathbb Z$, $a, b > 0$ we have 
	\begin{equation}\label{e:cover-compact}
		\#\{n\in \Z^3:\alpha_{n,\ell}(\tilde B_{r_0})\cap  \Lambda_{a, b}\not=\emptyset\} \le  c_0(a^22^{4\ell}+1)b2^{3\ell}.
	\end{equation}
	Moreover, the left-hand side above is $0$ if $c_0 b  \le 2^{-3 \ell}$. 
	
	Similarly, for any open neighborhood $U$ of $(0,0;0,1)$ there exists $c_U>0, m_U\in \mathbb N$ and $S\subset \mathbb (\tfrac{1}{m_U} \Z)^4$  so that for any $a, b > 0$ we have  $\Lambda_{a, b} \subset \bigcup_{n\in S}\alpha_n(U)$ and for all $\lambda \in \mathbb Z$ we have
	\begin{equation}\label{e:cover-open}
		\#\{(i,j,k,\lambda)\in S:\lfloor \lambda\rfloor =\ell \} \le c_U(a^22^{4\ell}+1)b2^{3\ell}
	\end{equation}
	and the left-hand side equals $0$ if $c_Ub \le 2^{-3\ell}$.
\end{lemma}
It is possible, but cumbersome, to prove the lemma by hand. Instead we use a bit of group theory for a quick proof.
\begin{proof}
	The maps $\{\alpha_{i,j,k,1} :(i,j,k)\in \mathbb R^3\}$ form the real Heisenberg group $H_\mathbb R$, as can be seen via the matrix representation
	$$
	\begin{bmatrix}
		x \\
		s  \\
		1 
	\end{bmatrix}
	\mapsto 
	\begin{bmatrix}
		1 & k & i \\
		0 & 1 & j \\
		0 & 0 & 1
	\end{bmatrix}
	\begin{bmatrix}
		x \\
		s  \\
		1 
	\end{bmatrix}=
	\begin{bmatrix}
		x+i+ks \\
		s+j  \\
		1 
	\end{bmatrix}.
	$$
	The discrete Heisenberg group $H_\mathbb Z=\{\alpha_{i,j,k,1} : (i,j,k)\in \mathbb Z^3\}$ forms a co-compact lattice in $H_\R$, so there is a compact subset $K\subset H_\R$ with $\bigcup_{\alpha\in H_\mathbb Z} \alpha (K)=H_\R$. Let $S=\{(0,0;0,r):r\in [1,4]\}$.
	Since the diagonal action of $H_\R$ is transitive on each set $O_r=\{(x,s;y,t)\in \Rd: t-s=r\}$, we have 
	\begin{align}
		\label{E:RdRd}
		\Rd&=\bigcup_{\ell \in \mathbb Z, h\in H_\R} \alpha_{0,0,0,\ell}\circ h(S)=\bigcup_{\ell \in \mathbb Z, \alpha \in H_\Z} \bigcup_{ h\in K} \alpha_{0,0,0,\ell}\circ \alpha \circ h(S)=\bigcup_{n\in \Z^4} \alpha_n(C), 
	\end{align}
	with $C=\big\{h(p): (h,p)\in K\times S\big\}.$
	The map $(h,p)\mapsto h(p)$ is continuous, so $C$ is compact. 
	Since every compact set is contained in some $\tilde B_r$, \eqref{e:cover-full} follows. 
	
	For a nonempty intersection in \eqref{e:cover-compact}, $i,j,k$ must satisfy
	$$
	2^{-2\ell}|x+i+sk|\le a,\; 2^{-3\ell}|s+j|\le b, \;2^{-2\ell}|x+sk-y-tk|\le 2a, \;2^{-3\ell}(t-s)\le 2b
	$$
	for some $(x,s;y,t)\in \tilde B_{r_0}$. Since $t-s$ is bounded below,  the third inequality limits the number of choices for $k$ to at most  $c(a2^{2\ell}+1)$. Given $k$, the first inequality allows at most $c(a2^{2\ell}+1)$ choices for $i$. The second inequality gives at most $c{b2^{3\ell}+1}$ choices for $j$. Since $t-s>1/r_0$, by the last inequality there are no solutions if $2^{3\ell}b<c$, and  \eqref{e:cover-compact} and the subsequent `Moreover' follows. 
	
	For the `Similarly' claim, by \eqref{E:RdRd} and rescaling, for large enough $m_0 \in \N$ there exists a compact $C' \subset U$ such that $\Rd = \bigcup_{n \in (m_0^{-1}\Z)^4} \al_n(C')$. Therefore there exists a finite set $F\subset (m_0^{-1}\mathbb Z)^4$ so that $\tilde B_{r_0}\subset \bigcup_{n\in F} \alpha_n(U)$. The remaining claims follow similarly to \eqref{e:cover-compact}.
\end{proof}

\subsection{Exponential tightness bounds for $\L_\e$}

In this section, we deduce two exponential tightness bounds for $\L_\e$. These bounds do not rely on the fine topological structure explored in the coming sections. 

\begin{proposition}\label{p:d-close}
	For every bounded $B\subset \Rd$,  $a,\delta>0$ we have 
	$$
	\limsup_{\e\to 0} \e^{3/2} \log P\Big(|\L_\e(u)-d(u)|\ge a(t-s)^{1/3}+\delta\text{ for some } u=(x,s;y,t)\in B\Big) \le - \tfrac43a^{3/2}.
	$$
\end{proposition}
\begin{proof}
	It suffices to prove the claim for $B=[-n,n]^4\cap \Rd$.
	For any open or closed  $D$, let
	\begin{align*}
		A_{\ep,1}(D)&=\Big\{\e|\L(u)-d(u)|\ge a(t-s)^{1/3}\text{ for some } u=(x,s;y,t)\in D\Big\}.
		\\
		A_{\ep,2}(D)&=\Big\{\e|\L(u)-d(u)|\ge \delta \text{ for some } u=(x,s;y,t)\in D\Big\}.
	\end{align*}
	
	Let $\eta<4/3$.  By Propositions \ref{p:neighborhood-l} and \ref{p:neighborhood-u} there exists $c>0$ and an open neighborhood $U$ of $(0,0;0,1)$ so that for $i=1,2$ and all  $\e>0$ we have $P\big(A_{\ep,1}(U)\big) < ce^{-\eta (a/\e)^{3/2}}$ and $ P\big(A_{\ep,2}(U)\big) < ce^{- (\delta/\e)^{3/2}} $. 
	
	Let $B^\e=([-n/\e^{1/2}, n/\ep^{1/2}]\times[- n, n])^2\cap \Rd$. Lemma \ref{l:cover} provides a set $S \subset (\tfrac{1}{m_U} \Z)^4$ so that $B^\e \subset \bigcup_{n\in S}\alpha_n(U)$. Let $A$ be the event  in the proposition. Then 
	$$
	P(A)\le P\big(\bigcap_{i=1,2} A_{\e,i}(B^\e)\big)\le \sum_{n\in S}\min_{i=1,2}P\big(A_{\e,i}(\alpha_n(U))\big)=\sum_{n\in S}\min\big(P\big(A_{\e, 1}(U)\big),P\big(A_{\e/2^{n_4}, 2}(U)\big)\big)
	$$
	where $n_4$ denotes the last coordinate of $n$, and $c'$ is a constant depending on $a, \de$. Here the final inequality uses the symmetries of the directed landscape (Lemma \ref{L:invariance}). Thus, by the cardinality bounds \eqref{e:cover-open}, for any $\ell_0\in \mathbb N$,
	$$
	P(A)\le \frac{c2^{7\ell_0}}{ \ep}\,e^{-\eta (a/\e)^{3/2}}+\frac{c}{\e}\sum_{\ell=\ell_0}^\infty2^{7\ell} e^{- (\delta 2^{\ell}/\ep)^{3/2}}\le \frac{c'2^{\ell_0}}{\e}(e^{-\eta (a/\e)^{3/2}}+e^{-(\delta2^{ \ell_0}/\e)^{3/2}}),
	$$
	where $c,c'$ depends on $a, n, \eta,\delta$ but not on $\e$ or $\ell_0$. Set $\ell_0$ large, then let $\ep\downarrow 0$ and finally let $\eta\uparrow\frac43$ to get the result. 
\end{proof}

We will only need a very weak version of the following proposition in the proof of Theorem \ref{t:main}.

\begin{proposition}\label{p:d-dominant}
	For every compact $K\subset \Rd$, let $r=\min\{t - s: (x,s;y,t)\in K \}$, and let $a>0$. We have
	$$
	\limsup_{\e\to 0}\e^{3}\log P\Big(\L_\e(u)\le d(u)-a\text{ for some } u\in K\Big) \le-\frac{a^3}{12r}. $$
\end{proposition}

\begin{proof}
	It suffices to prove the claim for $\tilde B_n$ defined in \eqref{e:tildeBn}.
	For any open or closed  $D$, let
	\begin{align*}
		A_{\ep}(D)&=\Big\{\ep\L(u)\le d(u)-a\text{ for some } u\in D\Big\}.
	\end{align*}
	Let $\eta<1/12$.  By Proposition \ref{p:neighborhood-l}, for every $\delta>0$ there exists $c>0$ and an open neighborhood $U$ of $(0,0;0,1)$ of diameter at most $\delta$ so that for and all  $\e>0$ we have $P(A_{\ep}(U))  < ce^{- \eta(a/\ep)^{3}} $. 
	
	Let $B^\e=\{(x,s;y,t): (\e^{1/2}x,s;\e^{1/2}y,t)\in \tilde B_n\}$. By 
	Lemma \ref{l:cover} we can find $m_0\in \mathbb N$ and sets  $S_\e'\subset (\tfrac{1}{m_0}\mathbb Z)^4$ so that for every $\e>0$ we have $B^\e \subset \bigcup_{n\in S_\e'}\alpha_n(U)$.  Let $S_\e$ denote the set of  $n\in S_\e'$ for which $\alpha_n(U)\cap B^\e \neq \emptyset$. By the compactness of $\tilde B_n$, the set $N_{4,U}=\{n_4:n\in S_\ep,\ep>0\}$ is finite (here $n_4$ denotes the last coordinate). Let $A_\e$ be the event in the proposition. Then 
	$$
	P(A_\e)=P \big(A_{\e}(B^\e)\big)\le \sum_{n\in S}P\big(A_{\e}(\alpha_n(U))\big)=\sum_{n\in S}P\big(A_{\e/2^{n_4}}(U)\big)\le c\sum_{n\in S}e^{-\eta (a2^{n_4}/\e)^3},
	$$
	where the final equality is by the symmetries of the directed landscape (Lemma \ref{L:invariance}). By the cardinality bounds \eqref{e:cover-open}, $\# S_\ep\le c/\ep$, and so letting $\e\to 0$  we get
	$$
	\limsup_{\e\to 0}\e^{3}\log P(A_\e)\le -\eta \min_{k\in N_{4,U}}(a2^{k})^3.$$
	As $\delta\to 0$, the minimum above converges to $a^3/r$.
	Letting  $\eta\uparrow\frac1{12}$ yields  the result. 
\end{proof}

\section{Topology}
\label{sec:top}

In this section, we introduce the space of functions on which we define the rate function and prove some basic properties of the metrics in this space. First, recall from the introduction that $\mathcal E$ is the space of continuous functions $e:\R^4_\uparrow \to \R$ with the topology of uniform convergence on bounded sets, satisfying the reverse triangle inequality $e(p; q) + e(q; r) \le e(p; r)$ for all triples $(p; q), (q ;r), (p; r) \in \Rd$. The space $\mathcal E$ is completely metrizable with the following metric $\fd$. Let $B_n = [-n, n]^4 \cap \R^4_\uparrow$, and for $e, e' \in \mathcal E$, let
\begin{align}\label{e:metric}
	\fd (e, e') = \sum_{n=1}^\infty 2^{-n} \frac{\fd _n(e, e')}{1 + \fd _n(e, e')}, \qquad \text{ where } \quad \fd _n(e, e') = \|e|_{B_n} - e'|_{B_n}\|_{\infty}.
\end{align}
A general directed metric $e \in \mathcal E$ does not have enough structure for us to define the rate function. To work around this, we will define our rate function on a well-behaved subset $\mathcal D \subset \mathcal E$ and simply set it to be $\infty$ elsewhere. We call $e\in \mathcal E$ {\bf Dirichlet-dominant} if $e \ge d$ everywhere. For such $e$ and a point $u = (x, s; y, t)\in \mathbb R^4_\uparrow$, recall the definition of $\Theta$ from \eqref{e:dThetaeu}, 
and define
$$
\Theta(e) = \sup_{u_1, \dots, u_k} \sum_{i=1}^k \Theta(e, u_i),
$$
where the supremum is over all finite sets of points $u_i = (x_i, s_i; y_i, t_i), i= 1, \dots, k$ such that the intervals $(s_i, t_i)$ are disjoint. The quantity $\Theta(e)$ can be thought of as a measure of how close $e$ is to the Dirichlet metric, and is closely related to the rate function.
\begin{definition}
	\label{D:Dm-def}
	For $m > 0$, let $\mathcal D_m \subset \mathcal E$ be the set of functions $e$ satisfying the following three conditions:
	\begin{enumerate}[label=(\roman*),leftmargin=18pt]	
		\item\label{eddom} $e$ is Dirichlet-dominant: $e \ge d$.
		\item\label{edclose} $e$ is Dirichlet-close: $\Theta(e) \le m$.
		\item\label{emet} $e$ satisfies the metric composition law:
		for any points $x, y \in \R$ and $s < r < t$ we have
		\begin{align}
			\label{e:metmax}
			e(x, s; y, t) = \max_{z \in \R} e(x, s; z, r) + e(z, r; y, t).
		\end{align}
	\end{enumerate}
	We set $\mathcal D:= \bigcup_{m > 0} \mathcal D_m.$
\end{definition}
Property \ref{emet} can be thought of as ensuring that $e$ defines a geodesic space, and later on we will show that this is indeed the case.

The rate function is only finite on a subset of $\mathcal D$. In the remainder of this section we collect some basic structural properties of $\mathcal D$. These properties can loosely be summarized as saying that each $\mathcal D_m$ is a compact, and that all metrics $e \in \mathcal D$ define geodesic spaces with quantitatively controlled path lengths.

The first three facts we record are immediate from considering the above definitions of $\Theta(e), \mathcal D_m$. We leave the proofs to the reader.

\begin{lemma}
	\label{L:basic-facts}
	\begin{enumerate}[label=(\roman*)]
		\item\label{L:basic-lsc} The function $e\mapsto \Theta(e)$ is lower semi-continuous on the space of continuous functions $e:\mathbb R^4_\uparrow \to \R$ satisfying $e \ge d$.
		\item\label{L:basic-mbound} If $e \in \mathcal D_m$, then for all $u = (x, s; y, t) \in \mathbb R^4_\uparrow$, since $\Theta(e) \le m$ we have
		$$
		e(u) - d(u) \le m^{2/3} (t-s)^{1/3}.
		$$
		\item\label{L:basic-invariant} For all $m > 0$, the set $\mathcal D_m$ is invariant under the four landscape symmetries from Lemma \ref{L:invariance}.
	\end{enumerate}
\end{lemma}

For some purposes, it will be more convenient to work with metrics that satisfy the condition in Lemma \ref{L:basic-facts}\ref{L:basic-mbound} rather than the bound on $\Theta(e)$ in Definition \ref{D:Dm-def}\ref{edclose}. For this reason, we let $\mathcal D_m^*$ be the set of all $e \in \mathcal E$ satisfying Definition \ref{D:Dm-def}\ref{eddom}\ref{emet} and Lemma \ref{L:basic-facts}\ref{L:basic-mbound}. Note that Lemma \ref{L:basic-facts}\ref{L:basic-invariant} also holds for $\mathcal D_m^*$. We next prove that $\mathcal D_m$ is compact, which will be necessary for eventually establishing our rate function is good. 
We start with an equicontinuity lemma.  We state this lemma in more generality than we need it here as it will later be used to prove exponential tightness for the upper bound in Theorem \ref{t:main}.

\begin{lemma}\label{L:equicontinuity} 
	Let $n, m, \ep > 0$, and suppose that $e \in \mathcal E$ satisfies the estimates
	\begin{equation}
		\label{E:d-e}
		d(u) - \ep \le e(u) \le d(u) + m^{2/3}(t-s)^{1/3} + \ep
	\end{equation}
	for all $u = (x, s; y, t) \in B_n$. Suppose also that for some $1 \le \ell$ and $n \ge m \vee \ell$, for any points $x, y \in [-\ell, \ell]$ and $s < r_1< r_2 < t \in [-\ell, \ell]$ we have
	\begin{align}
		\label{E:met-compact}
		e(x, s; y, t) = \max_{z_1,z_2 \in [-n, n]} e(x, s; z_1, r_1) +e(z_1,r_1;z_2,r_2)+ e(z_2, r_2; y, t).
	\end{align}
	Consider $u_1 = (p_1; q_1) = (x_1, s_1; y_1, t_1), u_2 = (p_2; q_2) = (x_2, s_2; y_2, t_2)\in B_{\ell}$ with $\|u_1 - u_2\|_\infty < [(t_1 - s_1)^3 \wedge (t_2 - s_2)^3]/64$. Then
	$$
	|e(u_1) - e(u_2)| \le 4 \ep + 12 n^2 \|u_1 - u_2\|_\infty^{1/9}.
	$$
\end{lemma}

\begin{proof}
	Set $\al = \|u_1 - u_2\|_\infty$. Choose $s_3 \in [s_1 + \al^{1/3}, s_1 + 2 \al^{1/3}) \cap [s_2 + \al^{1/3}, s_2 + 2 \al^{1/3})$ and $t_3 \in (t_1 - 2 \al^{1/3}, t_1 -  \al^{1/3}] \cap (t_2 - 2 \al^{1/3}, t_2 -  \al^{1/3}]$. Our restriction on $\al$ guarantees that $s_3 < t_3$. By \eqref{E:met-compact} we have that 
	\begin{align*}
		e(u_i) = \max_{z_1, z_2 \in [-n, n]} e(p_i; z_1, s_3) + e(z_1, s_3; z_2, t_3) + e(z_2, t_3; q_i).
	\end{align*}
	Therefore by \eqref{E:d-e} we have the estimate
	\begin{align*}
		|e(u_1) - e(u_2)| &\le 4 \ep + 2^{4/3} m^{2/3} \al^{1/9} \\
		&+ \sup_{z, z' \in [-n, n]} |d(p_1; z, s_3) - d(p_2; z, s_3)| + |d(z', t_3; q_1) - d(z', t_3; q_2)|.
	\end{align*}	
	By a straightforward computation, the above supremum is bounded by $8 n^2 \al^{1/3}$.
\end{proof}

We are now ready to prove our main compactness tool.

\begin{proposition}
	\label{P:compactness} Each of the sets $\mathcal D_m, \mathcal D_m^*$ is compact in $\mathcal E$.
\end{proposition}

\begin{proof}
	First let $\mathcal E_*$ denote the space $\mathcal E$ but with the topology of uniform convergence on compact, rather than bounded, sets. We first check that $\mathcal D_m^*$ is precompact in $\mathcal E_*$. For this, by the Arzel\`a-Ascoli theorem it is enough to check boundedness and equicontinuity on every compact set $\tilde B_\ell, \ell \in \N$. Boundedness follows from Dirichlet dominance and Lemma \ref{L:basic-facts}\ref{L:basic-mbound}. Equicontinuity follows from Lemma \ref{L:equicontinuity}, since \eqref{E:d-e} holds with $\ep = 0$ on $\Rd$, and in the metric composition law in Definition \ref{D:Dm-def}\ref{emet}, by \eqref{E:d-e}, on $B_\ell$ we have 
 \begin{align*}
   e(x, s; y, t) & = \max_{z_1,z_2 \in \R} e(x, s; z_1, r_1) +e(z_1,r_1;z_2,r_2)+ e(z_2, r_2; y, t) \\ & = \max_{z_1,z_2 \in [-n, n]} e(x, s; z_1, r_1) +e(z_1,r_1;z_2,r_2)+ e(z_2, r_2; y, t),   
 \end{align*}
for $n=(m+1)(\ell+1)$.

	Next we upgrade this to compactness in $\mathcal E_*$. That is, for a sequence $e_n \in \mathcal D_m^*$ we need to check that any  subsequential limit $e$ is in $\mathcal D_m^*$ as well. The conditions in Definition \ref{D:Dm-def}\ref{eddom} and Lemma \ref{L:basic-facts}\ref{L:basic-mbound} are closed conditions in the $\mathcal E_*$-topology, so $e$ must satisfy these conditions.
	
	To check that Definition \ref{D:Dm-def}\ref{emet} holds for $e$, observe that Lemma \ref{L:basic-facts}\ref{L:basic-mbound} and Definition \ref{D:Dm-def}\ref{emet} imply the following bound on $\mathcal D_m^*$. For any points $x, y \in \R$ and $s < r < t$ there exists $\al = \al(x, y, s, r, t)$ such that for all $e' \in \mathcal D_m^*$ we have
	\begin{align*}
		e'(x, s; y, t)&= \max_{z \in [-\al, \al]} e'(x, s; z, r) + e'(z, r; y, t).\\
		e'(x, s; y, t) &\ge 1 + \sup_{z \notin [-\al, \al]} e'(x, s; z, r) + e'(z, r; y, t).
	\end{align*}
	These two conditions are also closed conditions in the $\mathcal E_*$-topology, so both must be satisfied by $e$. Moreover, these conditions imply Definition \ref{D:Dm-def}\ref{emet} as desired.
	
	To complete the proof of the proposition, it suffices to show that any $\mathcal E_*$-convergent sequence $e_i$ in $\mathcal D_m^*$  is also $\mathcal E$-convergent to the same limit $e$. For this, consider the compact sets $\tilde B_n$ defined in \eqref{e:tildeBn}, and
	observe that for any  $n' \ge n$  in $\mathbb N$, Definition \ref{D:Dm-def}\ref{eddom} and Lemma \ref{L:basic-facts}\ref{L:basic-mbound} imply that
	\begin{align*}
		\|e_{i}|_{B_n} - e|_{B_n}\|_\infty \le \|e_{i}|_{\tilde B_{n'}} - e|_{\tilde B_{n'}}\|_\infty + m^{2/3} (n')^{-1/3}
	\end{align*}
	The right-hand side above converges to $m^{2/3} (n')^{-1/3}$ as $i \to \infty$ since $e_{i} \to e$ in $\mathcal E_*$. Taking $n' \to \infty$ then gives that $e_{i} \to e$ in $\mathcal E$.
	
	For $\mathcal D_m$, let  $e_n$ be a sequence in $\mathcal D_m$ with a  limit point $e \in \mathcal D_m^*$. Then $e$ satisfies Definition \ref{D:Dm-def}\ref{edclose} by the lower semicontinuity of $\Theta$, Lemma \ref{L:basic-facts}\ref{L:basic-lsc}. Thus $e\in \mathcal D_m$.
\end{proof}
Now, let $e \in \mathcal E$, and let $\gamma:[a, b] \to \R$ be any continuous path. Let $\bar \ga(r) := (\ga(r),r).$ Recall from the introduction that the $e$-length of $\ga$ can be defined as follows. For a partition $\mathcal P = \{a = r_0 < \dots < r_k = b\}$ of $[a, b]$, let
\begin{equation*}
	|\ga|_{e, \mathcal P} = \sum_{i=1}^k e(\bar \ga(r_{i-1}); \bar \ga(r_i)), \qquad
	|\ga|_e = \inf_{\mathcal P \text{ a partition of } [a, b]} |\ga|_{e, \mathcal P},
\end{equation*}
A path from $p$ to $q$ is a geodesic if $|\ga|_e = e(p; q)$, and a metric $e$ defines a \textbf{geodesic space} if for every $(p, q) \in \R^4_\uparrow$ there is a geodesic from $p$ to $q$. The next lemma establishes that all metrics in $\mathcal D$ are geodesic spaces.
A geodesic $\gamma$ from $p$ to $q$ is called a {\bf rightmost geodesic} if every geodesic $\gamma'$ from $p$ to $q$ satisfies $\gamma'\le \gamma$. 
First, we construct a candidate for the rightmost geodesic.

\begin{lemma}
	\label{L:geodesic-space0}
	For every $e \in \mathcal D$, $u = (p, q) = (x, s; y, t) \in \R^4_\uparrow$, and $r\in [s,t]$ the function 
	\begin{equation}
		\label{E:z-max}
		z \mapsto e(p; z, r) + e(z, r;q).
	\end{equation} 
	has a rightmost maximizer  $z_u(r)$. Moreover, $z_u(\cdot)$ is continuous.
\end{lemma}

In \eqref{E:z-max} for $r \in \{s, t\}$, we have extended $e$ outside of $\Rd$ as discussed after Definition \ref{D:directed-metric}.

\begin{proof}
	A maximizer necessarily exists by Definition \ref{D:Dm-def}\ref{emet}, and the continuity of $e$ and Lemma \ref{L:basic-facts}\ref{L:basic-mbound} guarantee that a rightmost maximizer exists.
	
	Next, we claim that $r\mapsto z_u(r)$ is continuous. Indeed, the shape conditions in Definition \ref{D:Dm-def}\ref{eddom} and Lemma \ref{L:basic-facts}\ref{L:basic-mbound} guarantee that $z_u$ is continuous at $s, t$. The continuity of $e$ implies that $z_u$ is upper semicontinuous everywhere. Moreover, the metric composition law (Definition \ref{D:Dm-def}\ref{emet}) implies that for any $r' < r \in (s, t)$, the value $z_{(\bar z_u(r'); q)}(r)$ maximizes the function \eqref{E:z-max}, and so
	$
	z_{(\bar z_u(r'); q)}(r) \le z_{u}(r).
	$
	Therefore
	$$
	z_u(r') = \liminf_{r \downarrow r'} z_{(\bar z_u(r'); q)}(r) \le \liminf_{r \downarrow r'} z_{u}(r),
	$$
	where the equality uses that $z_{(\bar z_u(r'); q)}$ is continuous at $r'$, which we have just established. Combining this with the upper semicontinuity of $z_u$ yields that $z_u$ is right-continuous at $r'$. A symmetric argument gives left-continuity.
\end{proof}	

Next, we show the quadrangle inequality on $\mathcal D$.
\begin{lemma}
	\label{L:quadrangle}
	For $e \in \mathcal D$, and points $(p, q) = (x, s; y, t), (p', q') = (x', s'; y', t)$ with $x < x'$, and $y < y'$ we have
	$$
	e(p; q') + e(p'; q) \le e(p; q) + e(p'; q').
	$$
\end{lemma}

\begin{proof}	In the setting of  Lemma \ref{L:geodesic-space0}, since $x < x', y < y'$ and $z_{(p, q')}-z_{(p', q)}$ is continuous, by the intermediate value theorem there exists $r \in (s, t)$ with $z_{(p, q')}(r) =  z_{(p', q)}(r)$. Let $o=(z_{(p',q)}(r),r)$. By the triangle inequality, 
	\[
	e(p; q') + e(p'; q)
	= e(p; o) + e(o; q') + e(p';o) + e(o; q) \le e(p; q) + e(p'; q'). \qedhere
	\]
\end{proof}

\begin{lemma}  \label{L:geodesic-space} The function $z_u$ defined in Lemma \ref{L:geodesic-space0} is a rightmost geodesic from $p$ to $q$.
\end{lemma}
\begin{proof}
	We follow the argument from \cite{DOV18}, Lemma 13.2.
	Any $e$-geodesic $\ga$ from $p$ to $q$ satisfies $\ga \le z_u$ by definition so it is enough to show that $z_u$ is itself a geodesic. For this, is enough to show that for every pair of points $r_1 < r_2 \in [s, t]$ we have
	\begin{equation}
		\label{E:epq}
		e(p; q) = e(p; \bar z_u(r_1)) + e(\bar z_u(r_1); \bar z_u(r_2)) + e(\bar z_u(r_2); q).
	\end{equation}
	Letting $a = z_{(\bar z_u(r_1); q)}(r_2)$ and $b = z_{(p; \bar z_u(r_2)}(r_1)$, we have that
	\begin{align*}
		e(p; q) &= e(p; \bar z_u(r_1)) + e(\bar z_u(r_1); a, r_2) + e(a, r_2; q), \\
		e(p; q) &= e(p; b, r_1) + e(b, r_1; \bar z_u(r_2)) + e(\bar z_u(r_2); q).
	\end{align*}
	As in the proof of Lemma \ref{L:geodesic-space0}, $a \le z_u(r_2)$ and $b \le z_u(r_1)$. Therefore adding the above two equalities and applying the quadrangle inequality in Lemma \ref{L:quadrangle} gives that
	\begin{align}
		\nonumber
		2 e(p; q) & \le [e(p; \bar z_u(r_1)) + e(\bar z_u(r_1); \bar z_u(r_2)) + e(\bar z_u(r_2); q)] \\
		\nonumber
		& \hspace{1cm}+[e(p; b, r_1) + e(b, r_1; a, r_2) + e(a, r_2; q)].
	\end{align}
	The triangle inequality for $e$ implies that both of the bracketed expressions on the right-hand side above are less than or equal to $e(p, q)$. This forces \eqref{E:epq}.
\end{proof}

The remaining lemmas in this section aim to understand the structure of finite length paths in metrics $e \in \mathcal D$. We start with the special case when $e = d$, where path length is given the negative of the Dirichlet energy. 

\begin{lemma}(Lemma 5.1.6, \cite{dz})
	\label{L:dirichlet-chara}
	Let $\ga:[a, b] \to \R$. If $\ga$ is absolutely continuous, then
	\begin{equation}
		\label{E:ga-d}
		|\ga|_d = - \int_a^b |\ga'(t)|^2 dt,
	\end{equation}
	and if $\ga$ is not absolutely continuous then $|\ga|_d = -\infty$.
\end{lemma}

Going forward, we write $H^1$ for the space of finite Dirichlet energy paths (i.e. paths for which $|\ga|_d \ne - \infty$).
Lemma \ref{L:dirichlet-chara} implies that for all $p = (x, s), q = (y, t)$ the unique $d$-geodesic from $p$ to $q$ is the linear function. 
For general metrics in $\mathcal D$, geodesics cannot wander too far from the Dirichlet geodesics.
\begin{lemma}
	\label{L:geo-mod-cont}
	Let $(p;q)\in \mathbb R^d$ and let $\pi_d$ be the (linear) $d$-geodesic from $p$ to $q$. 
	Let $e \in \mathcal D_m$. 
	Let $\pi$ be a path between points $p = (x, s), q = (y, t)$, with 
	$|\pi|_e\ge |\pi|_d$. Then for any $r \in [s, t]$ 
	$$
	|\pi(r) - \pi_d(r)| \le 2^{1/3} (t-s)^{1/6} \sqrt{m [r \wedge (t-r)]}.
	$$
\end{lemma}

\begin{proof}
	By symmetries of ${\mathcal D}_m$ (Lemma \ref{L:basic-facts}\ref{L:basic-invariant}), it suffices to prove the bound when $x= s = y = 0$. In this case, $e(0,0; 0, t) \ge 0$, whereas by Lemma \ref{L:basic-facts}\ref{L:basic-mbound}
	\begin{align*}
		e(0,0; z, r) + e(z, r; 0, t) &\le - \frac{z^2}{r} - \frac{z^2}{t-r} + m r^{1/3} + m (t- r)^{1/3} \\
		&\le  - \frac{z^2}{r \wedge (t-r)} + 2^{2/3} m t^{1/3}.
	\end{align*}
	For $z$ to lie along $\pi$, the right-hand side above must greater than or equal to $0$, implying that $|z| \le 2^{1/3} \sqrt{m[r \wedge (t-r)]} t^{1/6}$.
\end{proof}

Lemma \ref{L:dirichlet-chara} also implies that a path has finite $d$-length if and only if it has finite Dirichlet energy. This also holds for general $e \in \mathcal D$ by the following lemma.

\begin{lemma}
	\label{L:dirichlet-close-paths}
	Let $e \in \mathcal D_m$, and let $\gamma:[s, t] \to \mathbb R$ be any path. Then
	$$
	|\ga|_d \le |\ga|_e \le |\ga|_d + m^{2/3}(t-s)^{1/3}.
	$$
\end{lemma}

To prove Lemma \ref{L:dirichlet-close-paths}, we start with a useful lemma based on Jensen's inequality.
\begin{lemma}
	\label{L:path-pushed}
	Suppose that $e, e_* \in \mathcal D$ with $e_* \ge e$, and consider a collection of points $U = \{u_i = (x_i, s_i; y_i, t_i) : i =1, \dots, k\} \subset \R^4_\uparrow$.
	Then
	$$
	\sum_{i=1}^k \Theta(e_*, u_i) - \Theta(e, u_i) \ge \left(\sum_{i=1}^k (e_*(u_i) - e(u_i)) \right)^{3/2} \left(\sum_{i=1}^k (t_i - s_i)\right)^{-1/2}
	$$
\end{lemma}

\begin{proof}
	We can write
	\begin{equation}
		\label{E:Ie_*}
		\begin{split}
			\sum_{i=1}^k \Theta(e_*, u_i) - \Theta(e, u_i) = \sum_{i=1}^k f(\ep_i, b_i, t_i - s_i) - f(0, b_i, t_i - s_i),
		\end{split}
	\end{equation}
	where
	$$
	\ep_i = e_*(u_i) - e(u_i), \qquad b_i = e(u_i) + \frac{(x_i - y_i)^2}{t_i - s_i}, \qquad f(\ep, b, \Delta) =  \left[\frac{b + \ep}{\Delta} \right]^{3/2} \Delta.
	$$
	Since $e_* \ge e$, we have $\ep_i \ge 0$ for all $i$. Moreover, Dirichlet dominance of $e$ ensures that $b_i \ge 0$ for all $i$. For fixed $\ep, \Delta \ge 0$, the difference $f(\ep, b, \Delta) - f(0, b, \Delta)$ is monotone increasing in $b$, so \eqref{E:Ie_*} is bounded below by
	$$
	\sum_{i=1}^k f(\ep_i, 0, t_i - s_i) = \sum_{i=1}^k (\ep_i/[t_i - s_i])^{3/2} [t_i - s_i] \ge \left(\sum_{i=1}^k \ep_i\right)^{3/2} \left(\sum_{i=1}^k (t_i-s_i)\right)^{-1/2}.
	$$
	The final inequality is Jensen's inequality.
\end{proof}
\begin{proof}[Proof of Lemma \ref{L:dirichlet-close-paths}]
	The bound $|\ga|_d \le |\ga|_e$ follows since $d \le e$. For the second bound, it suffices to show that
	$$
	|\ga|_{e, \pn} - |\ga|_{d, \pn} \le m^{2/3}(t-s)^{1/3}.
	$$
	for any partition $\pn$. Using Lemma \ref{L:path-pushed} with $e_* = e, e = d$ and the definition of $\Theta$ we have
	$$
	(|\ga|_{e, \pn} - |\ga|_{d, \pn})^{3/2} \le \Theta(e) (t-s)^{1/2},
	$$
	which gives the desired bound since $\Theta(e) \le m$.
\end{proof}

Next, for a path $\ga:[s, t] \to \R$ and a metric $e$ define the {\bf weight function} $w_{\ga, e}:[s, t] \to \R$ by $w_{\ga, e}(r) = |\ga|_{[s, r]}|_e$. Lemma \ref{L:dirichlet-chara} implies that for any path $\ga \in H^1$, that the weight function $w_{\ga, d}$ is absolutely continuous with derivative $w_{\ga, d}' = -|\ga'|^2$. This absolute continuity also holds for general metrics $e \in \mathcal D$.

\begin{lemma}
	\label{L:finite-weight}
	Let $\ga \in H^1$, $e \in \mathcal D$. Then $w_{\ga, e}$ is absolutely continuous and $w_{\ga, e}' \ge -|\ga'|^2$ Lebesgue-a.e. 
\end{lemma}

We define the {\bf excess density} of $e$ along  $\gamma$ as the function $\rho_{\gamma,e}:[a,b]\to \mathbb R$ given by
\begin{equation}
	\label{e:rhogamma}
	\rho_{\gamma,e}(r)=w'_{\ga,e} (r)+ \ga'(r)^2\ge 0.
\end{equation}

\begin{proof}
	Write $x_-=-\min(x,0)$ and $x_+=\max(x,0)$. We have $[w_{\ga, e}(t) - w_{\ga, e}(s)]_- \le |w_{\ga, d}(t) - w_{\ga, d}(s)|$ for any $s < t$ since $e \ge d$ and Dirichlet weight functions are non-increasing. Since the weight function $w_{\ga, d}$ is absolutely continuous, for $w_{\ga, e}$ to be absolutely continuous it is therefore enough to show that for any $\ep > 0$, there exists $\de > 0$ such that for a disjoint collection of intervals $(s_i, t_i), i = 1, \dots, \ell$ we have
	\begin{equation}
		\label{E:ws-+}
		\sum_{i=1}^\ell (t_i - s_i) < \de \quad \implies \quad \sum_{i=1}^\ell (w_{\ga, e}(t_i) - w_{\ga, e}(s_i))_+ < \ep.
	\end{equation}
	Using Lemma \ref{L:path-pushed}, we have
	\begin{align*}
		\Theta(e) \ge \sum_{i=1}^\ell \Theta(e, (\bar \ga(s_i); \bar \ga(t_i)) &\ge \left( \sum_{i=1}^\ell e(\bar \ga(s_i); \bar \ga(t_i)) - d((\bar \ga(s_i); \bar \ga(t_i)) \right)^{3/2} \de^{-1/2} \\
		&\ge \left( \sum_{i=1}^\ell (w_{\ga, e}(t_i) - w_{\ga, e}(s_i))_+ \right)^{3/2} \de^{-1/2}
	\end{align*}
	yielding \eqref{E:ws-+} as long as $\de < \ep^3 \Theta(e)^{-2}$. As $e\ge d$, $w'_{\ga,e} \ge -|\ga'|^2$ Lebesgue-a.e.
\end{proof}

We finish this section by recording a property of path length under approximation.
\begin{lemma}
	\label{L:length-lemma}
	If $\ga_n \to \ga$ pointwise, and $e_n \to e$ in $\mathcal D$, then $\limsup_{n \to \infty} |\ga_n|_{e_n} \le |\ga|_e$.
\end{lemma}
The proof of Lemma \ref{L:length-lemma} is immediate from the definition and we leave the details to the reader.

\section{The rate function}\label{sec:rate}
In this section, we give a path definition of the rate function $I$, and use this to prove that it is a good rate function. For $e \in \mathcal D$ and a path $\ga \in H^1$, recall the excess density $\rho_{\gamma,e}=w_{\gamma,e}'+|\gamma'|^2$ defined in \eqref{e:rhogamma} and define
\begin{equation}
	\label{E:ga-rate}
	I(\ga, e) = \tfrac{4}{3} \int_{a_\ga}^{b_\ga} \rho_{\ga, e}(t)^{3/2} dt.
\end{equation}
Here and throughout the paper we let $[a_\ga, b_\ga]$ be the domain of a path $\ga$.

We say that a finite or countable collection of Dirichlet paths $\Ga$ is a {\bf network} if it is {\bf internally disjoint}, i.e. for all $\ga, \ga' \in \Ga$ with domains $[a, b], [a', b']$ and $r \in (a, b) \cap (a', b')$ we have $\ga(r) \ne \ga'(r)$. We call the network {\bf disjoint} if this also holds on $[a, b] \cap [a', b']$.

\begin{definition}[Rate function]
	\label{D:path-rate}
	Define the rate function $I:\mathcal E \to [0, \infty]$ by setting $I(e) = \infty$ for $e \notin \mathcal D$ and for $e \in \mathcal D$, letting
	$$
	I(e) := \sup_{\Gamma} I(\Gamma, e), \qquad \text{ where } \qquad I(\Gamma, e) := \sum_{\ga \in \Gamma} I(\ga, e).
	$$
	Here the supremum is over all networks $\Gamma$.
\end{definition}

In Definition \ref{D:path-rate}, we can equivalently take the supremum over all finite networks. Also, by slightly truncating paths we can take the supremum over all finite disjoint networks.
Next, we record a few basic properties of the rate function.

\begin{proposition}
	\label{P:rate-function-properties}
	The function $I:\mathcal E \to [0, \infty]$ satisfies the following properties:
	\begin{enumerate}[label=(\roman*)]
		\item $I(d) = 0$.
		\item \label{sublevel} For any $\al < \infty$, the sub-level set $I^{-1}[0,\al]$ is contained in $\mathcal D_{3 \al/4}$.
		\item $I$ is lower semi-continuous: if $e_n \to e$ in $\mathcal E$, then $\liminf_{n \to \infty} I(e_n) \ge I(e)$.
		\item For all $\al < \infty$, $I^{-1}[0,\al]$ is compact in $\mathcal E$.
	\end{enumerate}
\end{proposition}

Properties (iii, iv) above together imply that $I$ defines a good rate function.

To prepare for the proof of the proposition, we need a lemma relating the integral expression of the rate $I(\ga, e)$ to a partition-based expression of the rate. This partition-based expression will be used to facilitate proofs throughout the paper. 

For $e \in \mathcal D$, a path $\ga:[a, b] \to \R$ with $\ga \in H^1$, and a partition $\pn = \{r_0 < r_1 < \dots < r_k\}$ of $[a, b]$ and with the notation $x_+ = \max(x, 0)$,
define
\begin{equation}
	\label{E:6.3}
	I(\ga, e, \pn) := \frac{4}{3} \sum_{i=1}^k \left(\frac{w_{\ga, e}(r_i) - w_{\ga, e}(r_{i-1})}{r_i - r_{i-1}} + \frac{(\ga(r_i) - \ga(r_{i-1}))^2}{(r_i - r_{i-1})^2} \right)^{3/2}_+ (r_i - r_{i-1}).
\end{equation}

\begin{lemma}
	\label{L:partition-approximation}
	Let $e \in \mathcal D$, $\ga \in H^1, \ga:[a, b] \to \R$ and let $\pn$ be a partition of $[a, b]$. Then
	$I(\ga, e, \pn) \le I(\ga, e)$. Moreover, for any sequence of partitions $\pn_n = \{r_{n, 0} = a <  \dots < r_{n, \ell(n)} = b\}$ such that the mesh size $m(\pn_n) := \min \{ r_{n, i} - r_{n, i-1} : i = 1, \dots, \ell(n)\}$ approaches $0$ with $n$, we have
	$$
	\lim_{n\to\infty} I(\ga, e, \pn_n) = I(\ga, e).
	$$
\end{lemma}

\begin{proof}
	For a function $f:[a, b] \to \R$ and a partition $\pn = \{r_0 < \dots < r_k\}$ of $[a, b]$, write $[f]_\pn$ for the function which for $s \in [r_i, r_{i+1})$ is given by the average of $f$ on that interval:
	$$
	[f]_\pn(s) = \frac{1}{r_{i+1} - r_i}\int_{r_i}^{r_{i+1}} f(s) ds.
	$$
	Then
	\begin{align*}
		I(\ga, e, \pn) &= \frac{4}{3} \int_a^b \left([w_{\ga, e}']_{\pn} + ([\ga']_{\pn})^2 \right)_+^{3/2} 
		\le \frac{4}{3} \int_a^b \left([w_{\ga, e}']_{\pn} + [(\ga')^2]_{\pn} \right)_+^{3/2} \\&= \frac{4}{3} \int_a^b \left([w_{\ga, e}' + (\ga')^2]_{\pn} \right)^{3/2} \le \frac{4}{3} \int_a^b \left(w_{\ga, e}' + (\ga')^2 \right)^{3/2} = I(\ga, e).
	\end{align*}
	Here the two inequalities both use Jensen's inequality, and the equality on the middle line uses that $[f + g]_\pn = [f]_\pn + [g]_\pn$ and that $w_{\ga, e}' + (\ga')^2 \ge 0$. This yields the first part of the lemma.
	
	For the `Moreover', observe that $[w_{\ga, e}']_{\pn_n} + ([\ga']_{\pn_n})^2 \to w_{\ga, e}' + (\ga')^2$ Lebesgue a.e. by the Lebesgue differentiation theorem. Therefore by Fatou's lemma,
	$$
	\liminf_{n \to \infty} I(\ga, e, \pn_n) \ge \frac{4}{3} \int_a^b (w_{\ga, e}' + (\ga')^2)^{3/2}_+ = I(\ga, e),
	$$
	which combined with the first part of the lemma yields the result.
\end{proof}

\begin{proof}[Proof of Proposition \ref{P:rate-function-properties}]
	Property (i) holds because $I(\ga, d) = 0$ for any $\ga \in H^1$. For part (ii), consider any $e\in \mathcal D$.
	For any path $\gamma$ with endpoints $u=(x,s;y,t)$ we have 
	$$
	\tfrac34 I(\gamma,e)=\int \rho_{\gamma,e}(r)^{3/2}dr \ge \frac{(\int_s^t \rho_{\gamma,e})^{3/2}}{\sqrt{t-s}}= \frac{(|\gamma|_e-|\gamma|_d)^{3/2}}{\sqrt{t-s}}\ge\frac{(|\gamma|_e-d(u))^{3/2}}{\sqrt{t-s}}
	$$
	Now let $U$ be any set of points defined on time intervals with disjoint interiors, and for $u \in U$ let $\ga_{u}$ be an $e$-geodesic between the endpoints of $u$. Then 
	$$
	I(e) \ge \sum_{u \in U} I(\ga_{u}, e)  
	\ge\tfrac43\sum_{u=(x,s;y,t) \in U} \frac{(|\gamma_u|_e-d(u))^{3/2}}{\sqrt{t-s}}
	=\tfrac43\sum_{u \in U}\Theta (e,u).
	$$
	So $\Theta(e) \le 3I(e)/4$, which yields (ii).
	
	For part (iii), let $e_{n(i)}$ be a subsequence where
	$\lim_{i \to \infty} I(e_{n(i)}) = \liminf_{n \to \infty} I(e_n).$ If this limit is infinite, then the claim is trivially true, so we may assume this limit is finite, in which case by part (ii) there exists $m > 0$ such that $e_{n(i)} \in \mathcal D_m$ for all large enough $i$. By compactness of $\mathcal D_m$, Proposition \ref{P:compactness}, $e \in {\mathcal D}_m$ as well, and so its rate function is given by maximizing over networks. 
	
	Therefore for $\alpha < I(e)$, by Lemma \ref{L:partition-approximation} we can find a disjoint network $\ga_j:[a_j, b_j] \to \R, j = 1 \dots, k$ and partitions $\pn_1, \dots, \pn_k$ of the domains of $\ga_1, \dots, \ga_k$ such that
	$
	\sum_{j=1}^k I(\ga_j, e, \pn_j) \ge \alpha.
	$
	Moreover, we can make the mesh sizes of $\pn_j$ arbitrarily small.
	
	For every $n$, let $\ga_j^{n}$ be the path from $\bar \ga_j(a_j)$ to $\bar \ga_j(b_j)$ which equals $\ga_j$ on the points of $\pn_j$ and is given by the rightmost $e_n$-geodesic in between points of the partition $\pn_j$. As long as the mesh sizes of $\pn_j$ are sufficiently small, since $e_{n(i)} \in \mathcal D_m$ for large enough $i$, by Lemma \ref{L:geo-mod-cont} the paths  $\ga^{n(i)}_1, \dots, \ga^{n(i)}_k$ are disjoint for large enough $i$. Therefore
	$$
	I(e_{n(i)}) \ge \sum_{j=1}^k I(\ga^{n(i)}_j, e_{n(i)}, \pn_j) \to_{n \to \infty} \sum_{i=1}^k I(\ga_j, e, \pn_j),
	$$
	and so $\liminf_{n \to \infty} I(e_n) \ge I(e)$.
	
	For part (iv), $\mathcal D_{3 \al/4}$ is compact by Proposition \ref{P:compactness} and $I^{-1}[0,\al]$ is closed by part (iii) and a subset of $\mathcal D_{3 \al/4}$ by part (ii), so $I^{-1}[0,\al]$ is compact.
\end{proof}

\section{Finite-rate metrics,  measures and Kruzhkov entropy}
\label{sec:plant}

In this section we give a full description of the structure of finite-rate metrics. As we will see,  metrics are in one-to-one correspondence with measures $\mu$ on $\mathbb R^2$ whose support is contained in a set of the form $\gr \Gamma := \bigcup_{\ga \in \Ga} \gr \gamma$ for some network $\Gamma$ and have finite {\bf Kruzhkov entropy.} We call such measures {\bf planted network measures}.

We first define the {\bf temporal density} $\rho_\mu(x,t)$ of $\mu$. For $\gamma\in \Gamma$, $(x,t)\in \gr\gamma$ define $\rho_\mu(x,t)$ as the Lebesgue density of the $t$-marginal of 
$\mu|_{\gr \gamma}$.  Set $\rho_\mu=0$ on $(\gr \Gamma)^c$. The temporal density $\rho_\mu$ is well-defined $\mu$-almost everywhere, i.e. it is independent of the network $\Ga$. Define the {\bf Kruzhkov entropy} by
$$
\mathcal K(\mu)=\int \sqrt{\rho_\mu} d\mu.
$$
The name is motivated by Kruzhkov's analysis of Burgers' equation and its generalizations, but as far as we know, this is the first time it is used to define an entropy of a measure. 

\begin{definition}[Measure to metric]\label{d:mutoe} Given a planted network measure $\mu$, define the metric
	$$
	e_\mu(p;q)=\sup_\pi \mu(\gr \pi)+|\pi|_d
	$$
	where the $\sup$ is over all $H^1$ paths $\pi$ from $p$ to $q$. We call such a metric a planted network metric.
\end{definition}

A priori, it is not clear that $e_\mu$ even lies in $\mathcal E$, let alone that $e_\mu$ produces a finite rate metric. We prove this on route to our main structure theorem.
Next, for a finite rate metric $e$, we say a network $\Gamma$ is $e$-complete if $I(e) = I(\Gamma, e)$.

\begin{definition}[Metric to measure] \label{d:etomu}
	For a finite-rate metric $e$ and a network $\Gamma$ define 
	$$\mu_{\Gamma,e}=\sum_{\gamma \in \Gamma} (\rho_{\gamma,e} dt)\circ \gamma^{-1}, 
	$$
	where the notation refers to the  pushforward measure, and $\rho_{\ga, e}$ is the excess density defined in \eqref{e:rhogamma}. Set $\mu_e=\mu_{\Gamma,e}$ for any $e$-complete network $\Gamma$.
\end{definition}

At this point, it is not clear that $\mu_e$ is well-defined. This is shown in Proposition \ref{p:etomu}.
However, $\mu_{\Gamma,e}$ is well defined for all $\Gamma$. Its temporal density is given by 
$$
\rho_{\Gamma,e}(x)=\begin{cases}\rho_
	{\gamma,e}(t) & \text{if }(\gamma(t),t)=x \quad  \text{ for some } \gamma\in \Gamma
	\\ 0 & \mbox{else}
\end{cases}.
$$ 
The change-of-variables formula applied to  \eqref{E:ga-rate} implies 
\begin{equation}\label{e:IandK}
	I(\Gamma,e)=\tfrac{4}{3}\int \sqrt{\rho_{\Gamma,e}}d\mu_{\Gamma,e}=\tfrac{4}{3}\,\mathcal K(\mu_{\Gamma,e}).
\end{equation}
This motivates the main structure theorem. 

\begin{theorem}\label{t:structure}
	Definitions \ref{d:mutoe} and \ref{d:etomu} give a one-to-one correspondence between finite rate metrics and planted network measures. 
	The maps $\mu\to e_\mu$ and $e\to \mu_e$ are inverses of each other. Moreover $I(e_\mu)=\frac43 \mathcal K(\mu)$.
\end{theorem}

An immediate corollary of Theorem \ref{t:structure} is the following. 
\begin{corollary}
	\label{c:monotone}
	$I$ is strictly monotone:  for distinct $e, e' \in \mathcal D$ with $e \le e'$, we have $I(e) <I(e')$.
\end{corollary}

To prove Theorem \ref{t:structure}  we need the following lemma. It will be useful for eliminating conflicts when defining $\mu_{\Gamma,e}$ for two different $\Gamma$.

\begin{lemma}
	\label{L:overlap-lemma}
	Let $e \in \mathcal D$, and let $\ga, \pi \in H^1$ be paths with $e$-weight functions $w_\ga, w_\pi$ whose domains overlap on an interval $[a, b]$. Let $K_0 = \{r \in [a, b] : \ga(r) = \pi(r)\}$. Then $w_\ga'(r) = w_\pi'(r)$ and $\rho_{\pi, e} = \rho_{\ga, e}$ for Lebesgue-a.e.\ $r \in K_0$.
\end{lemma}

\begin{proof}Let $K$ be a closed subset of $K_0$.
	Let $\ep>0$ and let $\pn$ be a partition of $[a,b]$ so that $$|\gamma_{|[a,b]}|> |\gamma_{|[a,b]}|_\pn-\ep,$$
	where we drop $e$ from the notation $|\ga|_e$ and $|\ga|_{e,\pn}$ for convenience. The open set $(a,b)\setminus K$ is a finite or countable union of disjoint open intervals $U_1,U_2\ldots$. Let $k$ be so that the finite set $\pn \cap ((a,b)\setminus K)$ is contained in $\bigcup_{i=1}^k U_i$, and that $K'=[a,b]\setminus \bigcup_{i=1}^k U_i$ satisfies 
	\begin{equation}\label{E:wp-eps}
		\ep+\int_K w_\gamma'> \int_{K'} w_\gamma', \qquad \int_K w_\pi'< \ep +\int_{K'} w_\pi'.
	\end{equation}
	The latter can be achieved by the dominated convergence theorem.  $K'$ is a finite union of disjoint closed intervals $V_0,\ldots, V_{k}$. Let $\pn'$ be the union of $\pn$ and the endpoints of $U_1,\ldots, U_k$.  By construction, $\pn'\setminus \pn\subset K$, and also $\pn\cap K'\subset K$, hence $\pn'\cap K'\subset K$, which implies $|\gamma_{|V_i}|_{\pn'}=|\pi_{|V_i}|_{\pn'}$ for all $i$. We  bound
	$$
	\int_{K'} w_\gamma'  = |\gamma_{|[a,b]}|-\sum_{i=1}^k |\gamma_{|\bar U_i}|> |\gamma_{|[a,b]}|_{\pn'}-\ep - \sum_{i=1}^k |\gamma_{|\bar U_i}|_{\pn'}
	$$
	where $\bar U_i$ is the closure of $U_i$. We have 
	$$
	|\gamma_{|[a,b]}|_{\pn'}- \sum_{i=1}^k |\gamma_{|\bar U_i|}|_{\pn'}=\sum_{i=0}^{k} |\gamma_{|V_i}|_{\pn'}=\sum_{i=0}^{k} |\pi_{|V_i}|_{\pn'}\ge \sum_{i=0}^{k} |\pi_{|V_i}|=\int_{K'} w'_{\pi}.
	$$
	With \eqref{E:wp-eps}, this gives $\int_K w_\gamma'>\int_K w_\pi'-3\ep$. Since $\ep>0$ was arbitrary, and the roles of $\gamma$ and $\pi$ are symmetric, $\int_K w_\gamma'=\int_K w'_\pi$. Since this holds for all closed  $K\subset K_0$, we have that $w_\ga'(r) = w_\pi'(r)$ for Lebesgue-a.e.\ $r \in K_0$. The claim for $\rho_{\ga, e}, \rho_{\pi, e}$ follows by applying this equality to both $e$ and $d$.
\end{proof}

To show that $e$-complete networks exist, we will use the following extension lemma. 

\begin{lemma}\label{l:network-extension}
	Let $\ep>0$, $e\in \mathcal D$, and  $\Gamma, \Pi$ be finite networks. There exists a finite network $\Gamma^+\supset \Gamma$ so that for every Borel $A\subset \mathbb R^2$
	$$
	\ep+\int_A \sqrt{\rho_{\Gamma^+,e}}d\mu_{\Gamma^+,e}>\int_A \sqrt{\rho_{\Pi,e}}d\mu_{\Pi,e}.
	$$
	In particular, $\frac43\ep+I(\Gamma^+,e)>I(\Pi,e)$. 
\end{lemma}
\begin{proof}
	For $\gamma\in \Gamma, \pi\in \Pi$, let $K_{\gamma,\pi}$ be the set of times $t$ where $\gamma(t)=\pi(t)$. The set $K_\pi=\bigcup_{\gamma\in \Gamma}K_{\gamma,\pi}$ is closed, so $(a_\pi,b_\pi)\setminus K_\pi$ is a countable union of open intervals. Let $D_\pi$ denote the set of paths given by $\pi$ restricted to the closure of one of these intervals. Let $D^{\ep}_\pi$ be a finite subset of $D_\pi$ so that $I(D^{\ep}_\pi,e)>I(D_{\pi},e)-\ep/n$, where $n= |\Pi|$.
	Let
	$$\Gamma^+=\Gamma \cup \bigcup_{\pi  \in \Pi} D^\ep_\pi.
	$$
	We have  
	\begin{equation}\label{E:pathterms}
		\int_A \sqrt{\rho_{\Gamma^+,e}}d\mu_{\Gamma^+,e}=
		\sum_{\gamma\in \Gamma} \int_{\bar \gamma^{-1}(A)}
		\rho_{\gamma,e}^{3/2}(t)dt
		+\sum_{\pi\in \Pi}
		\sum_{\eta\in D^{\ep}_{\pi}}
		\int_{\bar \eta^{-1}(A)}\rho_{\eta,e}^{3/2}(t)dt.
	\end{equation}
	The first sum in \eqref{E:pathterms} is bounded below by 
	$$\sum_{\gamma\in \Gamma, \pi\in \Pi} \int_{\bar \gamma^{-1}(A)\cap K_{\gamma,\pi}}
	\rho_{\gamma,e}^{3/2}(t)dt=\sum_{\gamma\in\Gamma,\pi\in \Pi} \int_{\bar \pi^{-1}(A)\cap K_{\gamma,\pi}}
	\rho_{\pi,e}^{3/2}(t)dt,
	$$
	where the equality uses Lemma \ref{L:overlap-lemma}. To bound the second sum in \eqref{E:pathterms} term we write
	$$
	\ep/n+\sum_{\eta\in D_{\pi}^{\ep}}\int_{\bar \eta^{-1}(A)}\rho_{\eta,e}^{3/2}(t)dt
	\ge \sum_{\eta\in D_{\pi}}
	\int_{\bar \eta^{-1}(A)}\rho_{\eta,e}^{3/2}(t)dt=\int_{\bar \pi^{-1}(A)\cap K^c_{\pi}}\rho_{\pi,e}^{3/2}(t)dt.
	$$
	Summing the above  bounds we get
	\[\ep+
	\int_A \sqrt{\rho_{\Gamma^+,e}}d\mu_{\Gamma^+,e}\ge \sum_{\pi\in \Pi} \left(\int_{\pi^{-1}(A)\cap K_{\pi}}\!\!
	\rho_{\pi,e}^{3/2}(t)dt+\int_{\pi^{-1}(A)\cap K^c_{\pi}}\!\!
	\rho_{\pi,e}^{3/2}(t)dt
	\right)=\int_A \sqrt{\rho_{\Pi,e}}d\mu_{\Pi,e}.\qedhere
	\]\end{proof}

The next proposition implies that $\mu_e$ is well-defined for finite rate metrics. It also sheds light to how lengths are measured in $e$. 
\begin{proposition}\label{p:etomu}
	Let $I(e)<\infty.$ Then 
	\begin{enumerate}
		\item Any finite  network is contained in an $e$-complete  network.
		\item For any two  $e$-complete networks $\Gamma, \Pi$, we have $\mu_{\Gamma,e}=\mu_{\Pi,e}=:\mu_e$. Also, $\rho_{\Gamma,e}=\rho_{\Pi,e}=:\rho_e$ almost everywhere with respect to $\mu_e$.
		\item $|\gamma|_e=\mu_e(\gr \gamma)+|\gamma|_d$ for $\gamma\in H^1$.
	\end{enumerate}
\end{proposition}

\begin{proof}
	Let $\Pi_n$ be a sequence of finite networks so that $I(\Pi_n,e)\to I(e)$. Let $\Gamma_1$ be the given finite network, and for $n\ge 2$ construct $\Gamma_n$ consecutively as follows. Use Lemma \ref{l:network-extension} to get $\Gamma_{n}\supset \Gamma_{n-1}$, a finite network with $I(\Gamma_n,e)\ge I(\Pi_n,e)-1/n$. Then $\Gamma=\bigcup_{n\ge 1} \Gamma_n$ is an $e$-complete network, showing 1. 
	
	For 2, let $\Gamma_0\subset \Gamma$ be a finite network with $I(\Gamma_0,e)>I(e)-\frac43\ep$, and define $\Pi_0$ similarly. Use the notation $\nu_\Gamma=\sqrt{\rho_{\Gamma,e}}\mu_{\Gamma,e}$. Let $\Gamma_1\supset \Gamma_0$ be a finite network satisfying $\ep+\nu_{\Gamma_1}(A)\ge \nu_{\Pi_0}(A)$ for all Borel $A$ as in Lemma \ref{l:network-extension}. Then for any Borel $A$ we have
	\begin{align*}
		\nu_\Gamma(A)&\le \nu_{\Gamma_0}(A)+\ep\le 
		\nu_{\Gamma_1}(A)+\ep=\nu_{\Gamma_1}(\mathbb R^2)-\nu_{\Gamma_1}(A^c)+\ep 
		\\&\le \tfrac34 I(e)-\nu_{\Pi_0}(A^c)+2\ep\le
		\tfrac34 I(e)-\nu_{\Pi}(A^c)+3\ep=\nu_{\Pi}(A)+3\ep.
	\end{align*}
	Since $A,\ep$ were arbitrary and the roles of $\nu_\Gamma$, $\nu_\Pi$ are symmetric, we get $\nu_\Gamma=\nu_\Pi=:\nu$.
	
	Use the notation $\gr \Gamma=\bigcup_{\gamma\in \Gamma}\gr \gamma$. By the definition, $\nu\left((\gr \Gamma)^c\right)=\nu\left((\gr \Pi)^c\right)=0$. Moreover, on $\gr \Gamma \cap \gr\Pi$, $\rho_{\Gamma,e}=\rho_{\Pi,e}$ by Lemma \ref{L:overlap-lemma}, and 
	so $\rho_{\Gamma,e}=\rho_{\Pi,e}$ $\nu$-almost everywhere. Since by definition $\mu_\Gamma(\{x:\rho_\Gamma(x)=0\})=0$, and the same holds for $\Pi$, we see that $\mu_{\Gamma,e}=\mu_{\Pi,e}:=\mu_e$, and $\rho_{\Gamma,e}=\rho_{\Pi,e}$ $\mu_e$-almost everywhere, proving 2. 
	
	For part $3$, there is an $e$-complete network $\Gamma$ containing $\gamma$, and then  $\mu_{\Gamma,e}(\gr \gamma)=|\gamma|_e-|\gamma|_d$ by the definition of $\mu_{\Gamma,e}$.
\end{proof}

Next, we analyze the map $\mu\to e_\mu$. We first check that $e_\mu \in \mathcal D$ for any measure $\mu$ with finite entropy.

\begin{lemma}
	\label{L:inD}
	Let $\mu$ be a measure supported on $\gr \Gamma$ for some network $\Gamma$ and assume that $\mathcal K(\mu)$ is finite. Then for any path $\ga$, we have 
	\begin{equation}
		\label{E:jensen-measure}
		\frac{\mu(\gr \ga)^{3/2}}{(b_\ga - a_\ga)^{1/2}}  \le \int_{\gr \ga} \sqrt{\rho_\mu} d\mu = \mathcal K(\mu|_{\gr \ga}) \le \mathcal K(\mu).
	\end{equation}
	Moreover, defining $e=e_\mu$ by
	$
	e(p;q)=\sup_\ga \mu(\gr \gamma)+|\gamma|_d  
	$
	over all paths $\gamma \in H^1$ from $p$ to $q$, we have that $e \in \cD_{\cK(\mu)}$.
\end{lemma}

\begin{proof}
	The bound \eqref{E:jensen-measure} is immediate from Jensen's inequality. We will use it to prove that $e \in \mathcal D_{\mathcal K(\mu)}$.
	
	A priori, it is not clear that $e$ is continuous. However, $e$ is Dirichlet-dominant by construction (Definition \ref{D:Dm-def}(i)) and by \eqref{E:jensen-measure}, we have that $\Theta(e) \le \cK(\mu)$ giving Definition \ref{D:Dm-def}(ii). From there continuity of $e$ follows exactly as in the proof of the equicontinuity for $\mathscr D_m^*$ in Lemma \ref{L:equicontinuity}/Proposition \ref{P:compactness}.
	Finally, from the path definition, for any $x, y \in \R, s < r < t$, 
	$$
	e(x, s; y, t) = \sup_{z \in \R} e(x, s; z, r) + e(z, r; y, t).
	$$  
	This supremum is unchanged if we replace $\R$ by a large finite interval since $\Theta(e) < \infty$. Continuity of $e$ then implies the metric composition law, Definition \ref{D:Dm-def}\ref{emet}.
\end{proof}

In general, when we try to define a metric through lengths of paths, it may not in the end give the desired lengths. The next proposition shows that $e_\mu$ is well-behaved in this sense. 

\begin{proposition}\label{p:mutoe}
	For any measure $\mu$ supported on $\gr \Gamma$ for some network $\Gamma$ with $\mathcal K(\mu) < \infty$, setting $e = e_\mu$, we have $|\gamma|_e=\mu(\gr \gamma)+|\gamma|_d$ for every Dirichlet path $\gamma$.
\end{proposition}

%

\begin{proof}
	For a fixed partition $\mathcal P$ on $[a_\ga, b_\ga]$, let $S(\ga, \pn)$ denote the set of paths $\pi \in H^1$ from $\bar \ga(a_\ga)$ to $\bar \ga(b_\ga)$ which equal $\ga$ on $\pn$.	
	The definition of path length and the definition of the distance $e$ can be combined to give 
	\begin{equation}
		\label{e:gainfsup}   |\gamma|_e=\inf_{\pn}|\gamma|_{e,\pn}=\inf_{\pn}\sup_{\pi\in S(\ga, \pn)} \mu(\gr \pi)+|\pi|_d,
	\end{equation}
	which shows $|\gamma|_e\ge \mu(\gr \gamma)+|\gamma|_d$. Now let $\mu' = \mu|_{(\gr \ga)^c}$, and let $O$ be any open set containing $\gr \ga$. Then for any partition $\pn$,
	\begin{align}
		\nonumber
		&\sup_{\pi\in S(\ga, \pn)} \mu(\gr \pi)+|\pi|_d \\
		\nonumber
		&\le \mu(\gr \ga) + \sup_{\pi\in S(\ga, \pn), \gr \pi \subset O} (\mu'(\gr \pi)+|\pi|_d) \vee \sup_{\pi\in S(\ga, \pn), \gr \pi \not \subset O} (\mu'(\gr \pi)+|\pi|_d) \\
		\label{E:final-calc}
		&\le \mu(\gr \ga) + \sup_{\pi\in S(\ga, \pn), \gr \pi \subset O} (c_\ga \mathcal K(\mu'|_O)^{2/3} +|\pi|_d) \vee \sup_{\pi\in S(\ga, \pn), \gr \pi \not \subset O} (c_\ga \mathcal K(\mu')^{2/3}+|\pi|_d),
	\end{align}
	where the second inequality uses \eqref{E:jensen-measure} applied to $\mu'$, where $c_\ga := (b_\ga - a_\ga)^{1/3}$. As the mesh size of $\mathcal P$ tends to $0$, we claim that the second supremum term in \eqref{E:final-calc} tends to $-\infty$. Indeed, the triangle inequality for the $L^2$-norm implies that for any $\pi \in H^1$ from $\bar \ga(a_\ga)$ to $\bar \ga(b_\ga)$ we have 
	$$\sqrt{-|\pi|_d}\ge \sqrt{-|\pi-\gamma|_d}-\sqrt{-|\gamma|_d},
	$$
	and for $\pi\in S(\ga, \pn), \gr \pi \not \subset O$, the first term on the right-hand side above blows up to $\infty$ as the mesh of $\mathcal P$ tends to $0$ whereas the second term is finite. On the other hand,
	$$
	\inf_\pn \sup_{\pi\in S(\ga, \pn), \gr \pi \subset O} (c_\ga \mathcal K(\mu'|_O)^{2/3} +|\pi|_d) = c_\ga \mathcal K(\mu'|_O)^{2/3} + |\ga|_d.
	$$
	As we let the open set $O$ decrease down to $\gr \ga$, $\mathcal K(\mu'|_O) \to \mathcal K(\mu'|_{\gr \ga}) = 0$ by the dominated convergence theorem.  This gives that $|\gamma|_e\le \mu(\gr \gamma)+|\gamma|_d$.
\end{proof}

We now have all the ingredients to prove our structure theorem. 

\begin{proof}[Proof of Theorem \ref{t:structure}]
	Start with a finite rate metric $e$, construct $\mu_e$, and let $\Gamma$ be an $e$-complete network. This exists by Proposition \ref{p:etomu}.3.
	Then $\mathcal K(\mu_e) = \mathcal K(\mu_{\Gamma,e})=\frac34I(\Gamma,e)=\tfrac34 I(e)$ by \eqref{e:IandK}. Now, for all $\ga \in H^1$, \begin{equation}
		\label{E:equal-lengths}
		|\gamma|_e=\mu_e(\gr\gamma)+|\gamma|_d = |\gamma|_{e_{\mu_e}}, 
	\end{equation}
	where the first equality uses Proposition \ref{p:etomu}.3 and the second uses Proposition \ref{p:mutoe}. Finally, $e \in \mathcal D$ since $e$ has finite rate, and $e_{\mu_e} \in \mathcal D$ by Lemma \ref{L:inD}, and so both $e, e_{\mu_e}$ are geodesic spaces by Lemma \ref{L:geodesic-space} with all geodesics in $H^1$ by Lemma \ref{L:dirichlet-close-paths}. Therefore \eqref{E:equal-lengths} implies $e = e_{\mu_e}$.
	
	Now start with $\mu$ supported on $\gr\Gamma$ with $\mathcal K(\mu)<\infty$. Then 
	\begin{equation}
		\label{e:gamu}
		|\gamma|_{e_\mu}=\mu(\gr\gamma)+|\gamma|_d
	\end{equation}
	for all $\gamma\in H^1$ by Proposition \ref{p:mutoe}. Applying \eqref{e:gamu} to the paths $\ga|_{[a_\ga, r]}$ for all $r \in [a_\ga, b_\ga]$ we see that $\rho_{\ga, e_\mu} = \rho_\mu$ for any path $\ga$ and hence $\mu_{\Pi, e_\mu} = \mu|_{\gr\Pi}$ for any network $\Pi$. Therefore
	$$
	I(e) = \sup_{\Pi} I(e, \Pi) = \sup_\Pi \int_{\gr \Pi} \sqrt{\rho_{\mu}} d_{\mu|_{\gr \Pi}} = \tfrac{4}3 \mathcal K(\mu_{\Ga, e_\mu}) = \tfrac{4}3 \mathcal K(\mu)
	$$ 
	by \eqref{e:IandK}, and so $e_\mu$ is a finite rate metric with $\mu = \mu_{e_\mu}$.
\end{proof} 

We finish this section with a few estimates that will be useful in proving the large deviation principle, Theorem \ref{t:main}. The first lemma will be useful in the large deviation upper bound on cones. It will be used to show that a metric is large enough if distances are large on a certain finite set.

\begin{lemma}\label{l:pieces-length}
	Assume $e=e_\mu$ is a finite rate metric and  $\mu$ is supported on $\gr \Gamma$ for a network $\Gamma$. By possibly replacing each of the paths $\ga \in \Ga$ with $\ga|_{[a_\ga, (a_\ga + b_\ga)/2]}$ and $\ga|_{[(a_\ga + b_\ga)/2, b_\ga]}$ we may assume that for any point pair $(p; q) = (x, s; y, t)$ there is at most one path $\ga \in \Ga$ with $\gamma(s)=x$, $\gamma(t)=y$.
	
	We can then unambiguously define $e^0(x,s;y,t)=|\gamma_{[s,t]}|_e$ if $\gamma(s)=x$, $\gamma(t)=y$ for some $\gamma \in \Gamma$,  and  $e^0(x,s;y,t)=-\infty$ otherwise. 
	Let 
	$$
	e'(u)=\sup_{\pn} \sup_\pi \sum_{i=1}^\ell(e^0\vee d)(\pi(t_i),\pi(t_{i+1}))
	$$
	Then $e'=e$. Moreover, the inner sup is non-decreasing as the partition is refined. 
\end{lemma}
\begin{proof}
	The refinement claim can be proven by induction as we add an extra time in the partition. The inequality to check is  $$
	\sup_x (e^0\vee d) (p;x,s)+(e^0\vee d)(x,s;q)\ge (e^0\vee d)(p;q).
	$$
	When $d(p;q)\ge e^0(p;q)$ then we can take $(x,s)$ to be the point on the line segment $pq$ and $d$ in both terms on the left. Otherwise, let $\gamma\in \Gamma$ so that $p,q\in \gr\gamma$ and take $x=\gamma(s)$ and $e^0$ in both terms on the left. 
	
	It follows from the definition that $d\le e'\le e$. This implies that $e'\in \mathcal D$ and is of finite rate. 
	By the definition of length, for any $\gamma\in \Gamma$ and $[s, t] \subset [a_\ga, b_\ga]$ we have $|\gamma|_{[s, t]}|_{e'}=|\gamma|_{[s, t]}|_e$. Therefore $\mu_{e'}\ge \mu_{e}$, and so $e'=e$.
\end{proof}

\begin{proposition} Suppose $\mu_n, \mu$ are planted network measures such that $\mu_n(A)\uparrow \mu(A)$ for every Borel set $A$. Then $e_{\mu_n}\to e_\mu$ in $\mathcal E$.
\end{proposition}
\begin{proof}
	Since $I(e_{\mu_n}) \le I(e_\mu) < \infty$ and sub-level sets of $I$ are compact, $e_{\mu_n}$ has a subsequential limit $e$. For any path $\gamma$ we have $|\gamma|_{e_{\mu_n}}=\mu_n(\gr \gamma)+|\gamma|_d\to \mu(\gr \gamma)+|\gamma|_d=|\gamma|_{e_\mu}$. Thus by Lemma \ref{L:length-lemma}, $|\gamma|_{e} \ge |\gamma|_{e_{\mu}}$. This shows $e\ge e_\mu$. On the other hand $e_{\mu_n} \le e$ for all $n$ since $\mu_n \le \mu$, so $e \le e_\mu$.
\end{proof}

Approximating a finite-rate planted network measure by its restrictions to finitely many slightly truncated paths we get the following.

\begin{corollary}\label{c:approx} Any finite-rate metric can be approximated from below by finite disjoint planted network  metrics.
\end{corollary}

\section{The large deviation upper bound}\label{sec:upbd}
In this section, we prove the ingredients needed for the large deviation upper bound. The large deviation upper bound will follow from a version of exponential tightness (Proposition \ref{P:landscape-in-E}) and an upper bound on small balls (Proposition \ref{p:upbd}). Both of these bounds use only the limited inputs of Theorem \ref{T:tracy-widom-tails}, Proposition \ref{l:4.4} and the topological framework of the previous sections. We start with exponential tightness.

\subsection{Exponential tightness}

Here it will be convenient to expand the space $\mathcal E$ to the space $\bar \cE$ comprising all continuous functions from $\Rd \to \R$ with the same metric $\fd$ introduced in Section \ref{sec:top}. 

\begin{proposition}
	\label{P:landscape-in-E*} For every $m > 0$, there exists a compact set $K_m \subset \bar \cE$ such that for all $\de > 0$,
	\begin{align}
		\label{e:exptight*}
		\limsup_{\ep \to 0} \ep^{3/2} \log \Pr(\fd (\mathcal L_\ep, K_m) \ge \delta) \le - \tfrac{4}3 m.
	\end{align}
\end{proposition}

\begin{proof}
	Throughout we assume $m \ge 1, \de \le 1$.	By definition \eqref{e:metric} of $\fd$, letting $\ell = 2 + \lceil \log_2(\de^{-1}) \rceil$, for $f \in \bar \cE$ we have
	$$
	\fd (\mathcal L_\ep, f) \ge \delta \qquad \implies \qquad \fd_\ell(\mathcal L_\ep,f) > \delta/2.
	$$
	Therefore it is enough to construct a compact set $K_m'$ in the continuous function space $C(B_\ell)$ of continuous functions from $B_\ell \to \R$ with the uniform norm, such that \eqref{e:exptight*} holds with $\fd_\ell, K_m', \de/2$ in place of $\fd, K_m, \de$. For this, define the events
	\begin{description}
		\item[$A_{\ep}$:] For all $u \in B_{2 \ell m}$, we have
		$
		d(u) - \de/10 \le \cL_\ep(u)\le d(u) + m^{2/3} (t - s)^{1/3} + \de/10.
		$ Here and throughout the proof, we let $s, t$ be the time coordinates of $u$ when the notation is unambiguous.
		\item[$A_{ \ep}'$:] For all $\cL_\ep$-geodesics $\pi$ with endpoints in $B_\ell$, we have $\gr \pi \in B_{2 \ell m}$.
	\end{description}
	We have $\Pr(A_\ep) \ge 1 - \exp(-\tfrac{4}{3} \ep^{3/2} (m + o(1)))$ as $\e \to 0$ by Propositions \ref{p:d-close} and  \ref{p:d-dominant}. Moreover, $A_\ep \subset A_\ep'$. Indeed, suppose there is a geodesic $\pi$ between points $p =(x, s), q = (y, t) \in [-\ell, \ell]^2$ that exits $B_{2 \ell m}$. Then there exists $r \in (s, t)$ and $z = \pm 2 m \ell$ such that 
	$
	\cL_\e(p; q) = \cL_\e(p; z, r) + \cL_\e(z, r; q),
	$
	and so by the condition on $A_\ep$ we have
	$$
	d(p; q) - d(p; z, r) - d(z, r; q) \le 2m^{2/3} \ell^{1/3} + 3\de/10.
	$$
	A quick computation shows this is not possible. Therefore to complete the proof, we just need to construct a compact set $K_m'$ so that on $A_\ep = A_\ep \cap A_\ep'$ we have $\|\mathcal L_\ep|_{B_\ell}-f \|_\infty \le \delta/2$ for some $f \in K_m'$.
	
	The event $A_\ep'$ implies that for any $(p; q) \in B_\ell$, we have the metric composition law 
	\begin{equation}
 \label{E:Lepq}
 \begin{aligned}
		\cL_\e(p; q) & = \max_{z \in [-2 m\ell, 2 m \ell]} \cL_\e(p; z,r)+ \cL_\e(z, r; q) \\ & =\max_{z_1,z_2 \in [-2 m\ell, 2 m \ell]} \cL_\e(p; z_1, r_1) +\cL_\e(z_1,r_1;z_2,r_2)+ \cL_\e(z_2, r_2; q).
 \end{aligned}	
	\end{equation}
	Therefore on $A_\e$, by Lemma \ref{L:equicontinuity} we have the estimate
	\begin{equation}
		\label{E:L-e}
		|\cL_\e(u_1) - \cL_\e(u_2)| \le 2\de/5 + 48 m^2 \ell^2 \|u_1 - u_2\|_\infty^{1/9}
	\end{equation}
	for all $u_1, u_2 \in  B_\ell$ with $\|u_1 - u_2\|_\infty < \de^9/(64 \cdot 10^{9} m^6)$ and $(t_1 - s_1) \vee (t_2 - s_2) \ge \de^3/(1000 m^2)$.
	We use this to construct an approximation to $\cL_\e$ with an explicit modulus of continuity. Fix $\beta = m^{-18} \ell^{-18} {48}^{-9} (\de/10)^9, \ga = \de^3/(2000 m^2)$, define $ B_\ell' = \{u \in B_\ell: t - s \ge 2 \ga\}$ and let $M = \tilde B_\ell \cap (\beta \Z)^4$. 
	Now for $u \in  B_\ell'$, let $F_\e:B_\ell' \to \R$ interpolate $\cL_\e|_M$ by setting
	\begin{equation}
		\label{E:Feu}
		F_\e(u) = \sum_{v \in M: \|v - u\|_\infty < \beta} \cL_\e(v)\frac{[\beta - \|v - u\|_\infty]}{\sum_{v \in M: \|v - u\|_\infty < \beta} [\beta - \|v - u\|_\infty]}.
	\end{equation}
 We extend $F_\e$ to all of $B_\ell$ as follows. For $u = (x, s; y, t)$ with $t -s < 2\ga$, define
 $$
 F_\e(u) = d(u) + [F_\e(x, s; y, s + 2\ga) - d(x, s; y, s + 2\ga)](t - s - \ga)^+.
 $$
By construction, $F_\e = d$ off of the compact subset $\hat B_\ell = \{u \in B_\ell: t - s \ge \ga\}$. Moreover, by the definition of $A_\ep$, $F_\e$ satisfies a uniform bound on $\hat B_\ell$ depending only on $\de, m, \ell$. Finally, on $\hat B_\ell$, $F_\e$ satisfies an explicit modulus of continuity depending only on $\de, m, \ell$. This uses \eqref{E:L-e}, along with the estimate in $A_\e$ comparing $\cL_\e$ and $d$, which handles $F_\e$ on $\hat B_\ell \setminus B_\ell'$. Therefore by the Arzel\`a-Ascoli theorem on $C(\hat B_\ell)$, there exists a compact set $K_m' \subset C(B_\ell)$ such that on $A_\ep$, we have $F_\e \in K_m'$. Moreover, by \eqref{E:L-e}, \eqref{E:Feu} we have that 
	$$
	|F_\ep(u) - \cL_\ep(u)| \le 2 \de/5 + \de/10 = \de/2
	$$
	when $t - s \ge 2 \ga$. When $t - s \le 2 \ga$, we can obtain the same estimate via the bounds in the definition of $A_\e$. Therefore $\|\mathcal L_\ep|_{B_\ell}- F_\e \|_\infty \le \delta/2$, as desired.
\end{proof}

Next, we upgrade Proposition \ref{P:landscape-in-E*} to move from the abstract compact set $K_m$ to the explicit compact set $\cD_m$. We will use the notation $\Theta$ introduced in Section \ref{sec:top}, and in the proof we write $B(e, \ep)$ for the open $\fd $-ball of radius $\ep$ about a function $e \in \bar \cE$, and $B(K, \ep) = \bigcup \{B(e, \ep) : e \in K\}$. We will use the following folklore lemma, whose proof we leave to the reader.

\begin{lemma}\label{l:CK} Let $C$ be a closed set and $K$ be a compact set in a metric space $(M,d)$ with $C \cap K = \emptyset$.  Then $d(C,K):=\inf_{C\times K} d > 0$.
\end{lemma}

\begin{proposition}
	\label{P:landscape-in-E} For every $\delta, m > 0$ we have
	\begin{align}
		\label{e:exptight}
		\limsup_{\ep \to 0} \ep^{3/2} \log \Pr(\fd (\mathcal L_\ep, \mathcal D_m) > \delta) \le - \tfrac{4}{3}m.
	\end{align}
\end{proposition}

	\begin{proof}
		Let $m,\delta>0$ and consider the sets 
		$$
		\mathcal C_{m, \de} = K_m \setminus B(\mathcal D_m,\delta).
		$$
		Every $e \in \mathcal C_{m, \de}$ satisfies one of the following four conditions:
		\begin{enumerate}[nosep]
			\item There are points $(p; q), (q; r) \in \Rd$ with $e(p; q) + e(q; r) > e(p; r)$.
			\item There exists $u \in \Rd$ with $e(u) < d(u)$.
			\item There exists $(p, q) = (x, s; y, t) \in \Rd$ and $r \in (s, t)$ with
			$$
			e(p; q) > \sup_{z \in \R} e(p; z, r) + e(z, r; q).
			$$
			\item There exists finite set of points $U_e \subset \Rd$ defined on disjoint time intervals such that $\sum_{u \in U_e} \Theta(e, u) > m$.
		\end{enumerate}
		We show that in all of these cases we can find $\ga_e > 0$ such that 
		\begin{equation}
			\label{E:de-e}
			\Pr(\fd (\mathcal L_\ep, e) < \ga_e) \le \exp(-\tfrac{4}{3}\ep^{3/2} (m + o(1)), \qquad \text{ as} \quad \e \to 0.
		\end{equation}
		If $e$ satisfies property $1$, then since $\cL_\e$ satisfies the reverse triangle inequality, for small enough $\ga$ we have $\Pr(\fd (\mathcal L_\ep, e) < \ga) = 0$ for all $\e > 0$. If $e$ satisfies property $2$, then by Theorem \ref{T:tracy-widom-tails}, for small enough $\ga$ we have $\Pr(\fd (\mathcal L_\ep, e) < \ga) \le 2\exp(- \ga \ep^3)$ for all $\e \in (0, 1)$. Next, arguing as in the proof of Proposition \ref{P:landscape-in-E*}, there exists $m, \ell \ge 0$ such that with probability at least $1 - \exp(-\tfrac{4}{3} \ep^{3/2}(m + o(1))$, \eqref{E:Lepq} holds. This implies that if $e$ satisfies property $3$, then \eqref{E:de-e} holds for small enough $\ga_e$. Finally, if $e$ satisfies property 4, then since $\Theta(\cdot, u)$ is continuous on $\bar {\mathcal E}$ for all $u$, for small enough $\ga$, if $\fd (\mathcal L_\ep, e) < \ga$ then
		$$
		\sum_{u \in U_e} \Theta(\cL_\e, u) \ge m.
		$$
		By Theorem \ref{T:tracy-widom-tails} and the temporal independence of $\cL_\e$, the probability of this event is at most $\exp(-\tfrac{4}{3}\ep^{3/2} (m + o(1))$.
		
		Now, since $K_m$ is compact, its closed subset 
		$\mathcal C_{m, \de}$ is also compact. Therefore the open cover $ \{O_e := B(e, \ga_e) : e \in \mathcal C_{m, \de}\}$  contains a finite subcover  $\{O_{e_1}, \dots, O_{e_n}\}$. By Lemma \ref{l:CK}, we have  $r:=(\fd((O_{e_1}\cup\ldots\cup O_{e_n})^c,\mathcal C_{m, \de})\wedge \delta)/2>0$. Therefore if $\fd (\mathcal L_\ep,  \mathcal D_m) > \delta$ then $\mathcal L_\ep \notin B(K_m, r )$ or $\mathcal L_\ep \in O_{e_i}$ for some $i = 1, \dots, n$. The first event is covered by Proposition \ref{P:landscape-in-E*}, and the second is covered by \eqref{E:de-e}.
	\end{proof}
	
	Next, we prove the large deviation upper bound on small balls in $\mathcal D_m$.
	\subsection{A bound on small balls}
	
	\begin{proposition}\label{p:upbd} For any $e \in \mathcal D$, we have
		\begin{align}\label{e:DLub}
			\lim_{\delta\downarrow 0}\limsup_{\e\downarrow 0}\e^{3/2}\log\Pr(\L_\e \in B(e, \delta)) \le -I(e).
		\end{align}
	\end{proposition}
	
	For the proof, we need a definition. For a network $\Gamma$, define the {\bf separation} of $\Gamma$ as
	\begin{equation}\label{e:def-delta-gamma}
		\Delta(\Ga) = \inf \{|\ga(r) - \ga'(r)| : \ga \ne \ga' \in \Ga, r\in [a_\ga, b_\ga] \cap [a_\ga', b_\ga'] \},
	\end{equation}
	the minimal gap between paths in $\Ga$. For a finite disjoint network, we have $\Delta(\Ga)>0$.
	
	\begin{proof}
		Let $\Ga$ be a finite disjoint network, and  let $S = \{a_\ga, b_\ga : \ga \in \Ga\}$ be the set of all endpoints of paths in $\Ga$. Define the partition $\pn_{n, \ga}= (S \cup n^{-1} \Z)\cap[a_\ga,b_\ga].$
		
		Write $\bigcup_{\gamma\in \Gamma}\mathcal P_{n,\gamma}=\{t_0 < t_1 < \dots < t_k\}$. For all $i = 1, \dots, k$, let $U_i$ be the (possibly empty) set of all points of the form $(\bar \ga(t_{i-1}); \bar \ga(t_{i}))$ where $\ga \in \Ga$, and let $U=\bigcup_{i=1}^k U_i$. 
		Then
		\begin{align}
			\label{E:Theta-bd-2}
			\Pr(\L_\e\in B(e, \de))\le \Pr\bigg(\bigcap_{u\in U}\{\mathcal \L(u)\ge e(u)-\delta\}\bigg)
			= \prod_{i=1}^{k} \Pr \bigg(
			\bigcap_{u\in U_i}\{\mathcal \L(u)\ge e(u)-\delta\}\bigg).
		\end{align}
		The last equality follows since $\cL$ has independent time increments. Proposition \ref{p:Airy-tails} together with \eqref{E:Theta-bd-2} gives that if $|\Ga| = 1$ or if
		\begin{equation}\label{e:delta-cond}\theta(
			\Gamma,n):=\sum_{u\in U}\Theta(e,u)< c_0\Delta(\Gamma)^3n^2,
		\end{equation}
		we have 
		$$
		\Pr(\L_\e\in B(e, \de))\le  \exp\bigg((o(1)-\tfrac43\sum_{u\in U_i} \Theta(e-\delta,u))\e^{-3/2}\bigg).
		$$
		In this case, letting $\delta\to 0$ and using that $\Theta(\cdot, u)$ is continuous shows that 
		$$
		\eta:=
		\lim_{\delta\downarrow 0}\limsup_{\e\downarrow 0}\e^{3/2}\log\Pr(\L_\e \in B(e, \de)) 
		\le-\tfrac43 \sum_{u\in U}\Theta(e,u)
		\le -\sum_{\ga \in \Ga} I(\ga, e, \pn_{n,\ga}) 
		$$
		where the last inequality follows from the definition \eqref{E:6.3}.
		By Lemma \ref{L:partition-approximation}, as $n\to\infty$  we see that $\frac43\theta(\Gamma,n)\to I(\Gamma,e)$ so when $|\Ga|= 1$ or \eqref{e:delta-cond} holds, the above calculation is valid for large enough $n$. Therefore the left-hand side of \eqref{e:DLub} is bounded above by 
		$$
		- \max \big(\sup\{I( \Ga,e) : |\Ga| = 1\} , \sup \{I(\Ga,e) : I(\Ga,\e) < \infty, \Ga \text{ finite, disjoint}\} \big).
		$$
		At this point, either the first term in the maximum is $\infty$, in which case $I(e) = \infty$ and we have the result, or else in the second term we can eliminate the constraint that $I(\Ga,e) < \infty$, and this term equals $I(e)$ by the discussion following Definition \ref{D:path-rate}. 
	\end{proof}
	Next, we assemble the parts to complete the proof of the large deviation upper bound.
	\subsection{Proof of the upper bound in Theorem \ref{t:main}} \label{sec:9.1}
	
	Fix any closed set $C \subset \mathcal E$ and let $\de,m > 0$. 
	By Proposition \ref{p:upbd}, for every $e \in C\cap \mathcal D_m$ we can find an open ball $B$ centered at $e$ such that
	\begin{equation}
		\label{E:limsup-est}
		\limsup_{\ep \downarrow 0} \ep^{3/2} \log \Pr(\cL_\ep \in B) \le - I(e) + \de.
	\end{equation}
	By Proposition \ref{P:compactness}, $\mathcal D_m \cap C$ is compact, so it is covered by a finite collection $\mathcal O$ of such balls. 
	Let $r$ be the distance of the compact set $C\cap \mathcal D_m$ and the closed set $(\bigcup \mathcal O)^c$. By Lemma \ref{l:CK}, $r>0$. Then 
	$$
	C \subset \big(\bigcup \mathcal O\big) \cup \{e \in \mathcal E: \fd (e, \mathcal D_m) > r/2\}.
	$$ 
	Therefore applying \eqref{E:limsup-est} together with Proposition \ref{P:landscape-in-E*} gives that
	$$
	\limsup_{\ep \downarrow 0} \ep^{3/2} \log \Pr(\cL_\ep \in C) \le - \min \left(\inf_{e \in C} I(e) + \de, \frac{4m}{3}\right).
	$$
	Taking $m \to \infty$ and then $\de \to 0$ completes the proof.

	\section{The large deviation lower bound}\label{sec:lwbd}
	
	\subsection{A bound on cones}
	
	Given $e \in \mathcal E$, define the cone with apex $e$ by
	\begin{align}\label{def:cone}
		\m{Cone}_e:=\big\{e' \in \mathcal E: e'\ge e \big\}.\end{align}
	
	In this section, we prove a large deviation lower bound on neighborhoods of  cones. This will be combined with strict monotonicity of $I$ and a topological argument to give the full large deviation lower bound. 
	\begin{theorem}\label{t:lbd_apart0}
		Let $e$ be a finite disjoint planted network metric so that $\mu_e$ is supported on the graph $\gr \Gamma$ of  a  finite disjoint network $\Gamma$. For all $\delta>0$ we have
		\begin{align}
			\label{e:lbDL00}
			\liminf_{\e\downarrow 0}\e^{3/2}\log\Pr\big( \L_\e \in B(\m{Cone}_e,\delta)\big) \ge -I(e).
		\end{align}
	\end{theorem}

	\begin{proof}
		We may assume that $I(e)<\tfrac43m<\infty$, or else there is nothing to prove. Our first goal is to reduce the problem to having to control only finitely many values of $\L_\ep$. Towards this end, write
		$$\m{Cone}_e= \bigcap_{r\in(0,\infty),u\in\Rd}
		G(u,e-r), \qquad G(u,e)=\{e'\in \mathcal E:e'(u)\ge e(u)\}.
		$$
		so that  
		$$
		\mathcal D_m\cap  B(\m{Cone}_e,\delta/2)^c\subset \m{Cone}_e^c =\bigcup _{r\in(0,\infty),u\in \Rd} G(u,e-r)^c,
		$$
		an open cover. 
		Since $\mathcal D_m$ is compact (Proposition \ref{P:compactness}) so is $\mathcal D_m\cap B(\m{Cone}_e,\delta/2)^c$, and therefore we can find a finite subcover 
		$\{G(u,e-r_u)^c:u\in Q\}$. Let $r=(\delta \wedge \min_{u\in Q} r_u)/2$.
		Then $\{G(u,e-2r)^c:u\in Q\}$ is also a subcover, giving
		\begin{equation}
			\label{e:bd2}
			\mathcal B(\m{Cone}_e,\delta/2) \supset  \mathcal D_m
			\cap\bigcap _{u\in Q} G(u,e-2r).
		\end{equation} 
		The set $\mathcal D_m$ is too small for our bounds. To fix this, we claim that 
		\begin{equation}\label{e:lb-twoset}
			\mathcal B(\m{Cone}_e,\delta) \supset B(\mathcal D_m,r)\cap \bigcap _{u\in Q} G(u,e-r).
		\end{equation}
		Indeed, let $e_1$ be in the set on the right. Then there exists $e_2\in \mathcal D_m$ with $\fd(e_1,e_2)<r$ so that 
		$e_2\in \bigcap_{u\in Q} G(u,e-2r)$. By \eqref{e:bd2} then $e_2\in B(\m{Cone}_e,\delta/2)$. Then $e_1\in B(\m{Cone}_e,\delta)$ since $\delta\ge r+\delta/2$, giving \eqref{e:lb-twoset}.
		
		Let $S=\bigcup_{(x,s;y,t)\in Q} \{s,t\}$, the set containing all time coordinates in $Q$.
		By Lemma \ref{l:pieces-length}, for $\tau>0$ we can find a partition $\pn=\{t_0<t_1<\ldots< t_k\}$  containing $S$ so that the following holds. First, $t_{i+1}-t_i\le  \tau$ for all $i$. Second, for every $u=(x,t_i;y,t_j)\in Q$ there exists points $(x,t_i)=z_i, z_{i+1},\ldots, z_j=(y,t_j)$ with increasing and consecutive time coordinates $t_i,t_{i+1},\ldots ,t_j$ so that 
		\begin{align}\label{e.lwsumbd}
			\sum_{v\in V_u} e^0\vee d (v)>e(u)-r, \qquad \text{ where }V_u=\{(z_{\ell-1},z_{\ell}):\ell=i+1,\ldots ,j\}.
		\end{align}
		Here the function $e^0$ is defined in Lemma \ref{l:pieces-length}.
		Let $U =\bigcup_{u\in Q}V_u$. By \eqref{e.lwsumbd} and the triangle inequality for $\L_\e$, 
		\begin{align}\label{e:lb-p1}
			P&\Big(\cL_e\in \bigcap _{u\in Q} G(u,e-r)\Big)
			\ge \Pr\Big(\bigcap_{u \in U} \{\cL_\ep(u) \ge e^0 \vee d(u)\}\Big).
		\end{align}
		Write $U=\bigcup_{i=1}^k U_i$ with $U_i=\{(x,t_{i-1};y,t_i)\in U\}$.   By the independent increment property of the directed landscape and the upper tail bound of Proposition \ref{p:Airy-tails}, the right hand side of \eqref{e:lb-p1} equals
		\begin{align}\label{e:lb-p2}
			\prod_{i=1}^{k}
			P\Big(\bigcap_{u \in U_i} \{\cL_\ep(u) \ge e^0 \vee d(u)\}\Big)
			= \exp \Big(-\sum_{u \in U} \tfrac{4}{3} \ep^{-3/2}(\Theta(e^0\vee d, u) + o(1)) \Big)
		\end{align}
		as long as  the minimal gap $\Delta(\Gamma)$ defined  in \eqref{e:def-delta-gamma} satisfies
		\begin{equation}\label{e:lb-gap-cond}
			\sum_{u \in U}\Theta(e^0\vee d, u)<c_0\Delta(\Gamma)^3/\tau^2. \end{equation}
		Finally, by Jensen's inequality, 
		\begin{align}\label{e:lb-p3}
			\sum_{u \in U}  \Theta(e^0\vee d, u) = \sum_
			{
				\genfrac{}{}{0pt}{}
				{
					\gamma\in\Gamma,(x,s;y,t)\in U:}{\gamma(s)=x,\gamma(t)=y} 
			}
			\frac{(\int_s^t\rho_{\gamma,e}(r)dr)^{3/2}}{\sqrt{t-s}}\le\sum_{\gamma\in\Gamma} \int \rho_{\gamma,e}(r)^{3/2}dr=\tfrac34I(e),
		\end{align}
		so in particular $\eqref{e:lb-gap-cond}$ holds for $\tau$ small enough. 
		By Proposition \ref{P:landscape-in-E*} $$P(\L_\e \in B(\mathcal D_m,r)^c)\le e^{- (\frac43m+o(1))\e^{-3/2}}.$$
		This together with \eqref{e:lb-twoset} and \eqref{e:lb-p1}-\eqref{e:lb-p3} imply the claim. 
	\end{proof}

	Next, we conclude the proof of the large deviation lower bound. 
	
	\subsection{Proof of the lower bound in Theorem \ref{t:main}}
	Let  $O=A^\circ$, an open set. If $O$ does not contain finite rate metrics, there is nothing to prove. So let $e_0\in O$ with $I(e_0)<\infty$. Then $e_0\in \mathcal D$, and by Corollary \ref{c:approx}, we can find a finite disjoint planted network metric $e\le e_0$ so that $e\in O$. It suffices to show that 
	\begin{align}\label{rtg}
		\liminf_{\e\downarrow 0}\e^{3/2}\log\Pr(\L_\e \in O) \ge -I(e).
	\end{align}
	Recall the set $\m{Cone}_e$ from Theorem \ref{t:lbd_apart0}. We bound
	$
	\Pr(\cL_\ep \in O)$ below by 
	$$
	\Pr(\cL_\ep \in O \cap B(\m{Cone}_e,\delta)) = \Pr(\cL_\ep \in B(\m{Cone}_e,\delta))- \Pr(\cL_\ep \in B(\m{Cone}_e,\delta) \cap O^c),
	$$
	and so in view of Theorem \ref{t:lbd_apart0}, to prove \eqref{rtg} it is enough to show that for some $\delta> 0$,
	\begin{equation}
		\label{E:cone-intersection}
		\limsup_{\ep \downarrow 0} \ep^{3/2} \log\Pr(\cL_\ep \in B(\m{Cone}_e,\delta) \cap O^c) < - I(e).
	\end{equation}
	The closure of $B( \m{Cone}_e,\delta)\cap O^c$ is contained in $B(\m{Cone}_e,2\delta)\cap O^c$. By the large deviation upper bound, it suffices to show that  
	$$
	s=\sup_{\delta>0}\; \inf_{e' \in B(\m{Cone}_e,2\delta) \cap O^c} I(e') >I(e).
	$$
	If $s=\infty$, there is nothing to prove. Otherwise we can find $e_n \in B(\m{Cone}_e,1/n)\cap O^c$ with $I(e_n)\to s$. By Proposition \ref{P:rate-function-properties}, $I$ is lower semi-continuous with compact sub-level sets, so $e_n$ has a subsequential limit $e^*\in \m{Cone}_e\cap O^c$ and $s=\lim I(e_n) \ge I(e^*)$. Thus  $e^* \ge e$ and $e^*\not=e$.  Strict monotonicity of $I$, Corollary \ref{c:monotone}, gives $I(e^*) > I(e)$,  as required.

	\section{Large deviations of the directed geodesic}\label{s:gamma-LDP}
	
	In this section, we use the results of Section \ref{sec:app} to prove Theorem \ref{T:dg-ldp}.  By symmetries of $\cL$, it suffices to prove the theorem when $(p; q) = (0,0; 0, 1)$.
	
	\begin{proof}[Proof of Theorem \ref{T:dg-ldp}]\ 
		
		{\bf $J$ is lower semicontinuous.}  Consider $f_n \to f$ uniformly, and suppose that $J(f_n) $ converges to a finite limit, or else there is nothing to prove.  By Lemma \ref{l:j-facts}, we can find $e_n$ achieving the infimum \eqref{e:J-def} for each $f_n$, and since $I$ is a good rate function, there is a subsequential limit $e$ of the $e_n$, and $I(e) \le \liminf_{n \to \infty} I(e_n)$. Let $u=(0,0;0,1)$, then $|f_n|_{e_n}=e_n(u)\to e(u)$. By Lemma \ref{L:length-lemma}, $|f|_e\ge \liminf |f_n|_{e_n}=\lim e_n(u)=e(u)$, so $f$ is a geodesic in $e$. Thus $J(f)\le I(e)$, as required. 
		
		{\bf Sub-level sets are compact.}
		By Lemma \ref{l:J-norms}, for $\al > 0$, we have that $J^{-1}[0, \al]$ is contained in the set
		$$
		\Big\{f :[0, 1] \to \R, f(0) = f(1) = 0, \int_0^1 |f'(t)|^{3/2} dt \le \sqrt{3 \al/4} \Big\},
		$$
		which is compact in the uniform norm. Since $J$ is lower semicontinuous, $J^{-1}[0,a]$ is also closed. Therefore sub-level sets are compact, and $J$ is a good rate function.
		
		{\bf The large deviation upper bound.} Let $\mathcal G\subset \mathcal E$ denote the set of geodesic spaces, see the definition prior to Lemma \ref{L:geodesic-space0}. Since $\mathcal D\subset \mathcal G$ by Lemma \ref{L:geodesic-space}, $I=\infty$  outside $\mathcal G$. Since also $\L_\e\in\mathcal G$ almost surely, it follows that $\L_\ep$ satisfies the large deviation principle, Theorem \ref{t:main}, on $\mathcal G$ replacing  $\mathcal E$. 
		
		Let $A$ be a Borel  subset of $C([0, 1])$, and let $S_A \subset \mathcal G$ be the set of  metrics that have a geodesic from $(0, 0)$ to $(0, 1)$ which lies in $A$. Then
		\begin{equation}
			\label{E:Pr-translation}
			\Pr(\ep \ga \in A) = \Pr(\cL_{\ep^2} \in A).
		\end{equation}
		Moreover, if $A$ is closed then $S_A$ is also closed by Lemma \ref{L:length-lemma}, and so by the large deviation upper bound for $\cL_\ep$ we have
		$$
		\limsup_{\ep \to 0} \ep^{-3} \log\Pr(\ep \ga \in A) \le -\inf_{S_A} I = -\inf_{A} J.
		$$
		%
		%
		
		{\bf The large deviation lower bound.}
		Suppose that $A$ is open in $C([0, 1])$, and assume that ${J}(f) < \infty$ for some $f \in A$. By \eqref{E:Pr-translation} and the large deviation lower bound for $\cL_\ep$ on $\mathcal G$ we have
		$$
		\liminf_{\ep \to 0} \ep^{-3} \log\Pr(\ep \ga \in A) =-\inf_{ S_A^{\circ}} I,
		$$
		where here the interior is defined with respect to $\mathcal G$.
		Thus it is enough to show that for any $f \in A$ with $J(f) < \alpha<\infty$, we can find   $e \in S_A^\circ$ with $I(e) < \alpha$. 
		
		Let $\mu$ be a planted network measure supported on $\gr f$ so that $e_{\mu} 
		\in \mathcal D(f)$ and $I(e_\mu) = J(f)$, and let $\lambda$ be the measure supported on $\gr f$ with time marginal given by Lebesgue  measure on $[0,1]$. Let $\kappa>0$ and $e=e_{\mu+\kappa\lambda}$. 
		By the triangle inequality for $L^{3/2}$-norms, 
		$I(e)^{2/3}\le I(e_\mu)^{2/3}+\kappa$, and  we can set $\kappa$ small enough so that $I(e)< \alpha$.
		
		Since $f$ is an $e_\mu$-geodesic for $u$, $f$ must be the unique $e$-geodesic for $u$. This implies
		\begin{equation}
			\label{E:dfwn}
			\max_{a\in \{\delta,-\delta\}}\sup_{r \in [0, 1]} \Big(e(p; f(r) +a, r) + e(f(r) +a, r; q) \Big) < e(u).
		\end{equation}
		Let $\delta>0$ so that $\{g:\|f-g\|_\infty<\delta\}\subset A$ and so $S_{\{g:\|f-g\|_\infty<\delta\}}\subset S_A$.
		
		If $e_n \in \mathcal G$ and $e_n \to e$ uniformly on bounded sets, then for all large $n$ \eqref{E:dfwn}  holds for $e_n$ as well. Any $e_n$-geodesic $g$ is continuous and by \eqref{E:dfwn} cannot  intersect $f+\delta$ or $f-\delta$, hence $\|f-g\|_\infty<\delta$ and $e_n\in S^\circ_A$. Since for any $e_n\to e$ we have $e_n\in S_A$ for large enough $n$, we have $e\in S_A^{\circ}$, as required. 
	\end{proof}
	
	\section{The rate function as an integrated Dirichlet energy}
	\label{sec:dirichlet}
	
	The goal of this section is to give an alternate characterization of the rate function that is based on differentiating the metric directly. To motivate this form of the rate function, we first focus on understanding the one-dimensional spatial marginals $x \mapsto \cL(x, s; y, t)$ and $y \mapsto \cL(x, s; y, t)$. All of these marginals are rescaled versions of the \textbf{parabolic Airy process} $\mathfrak A_1(x) = \cL(0,0; x, 1)$.
	
	\subsection{Conjectured large deviations for the Airy process}
	
	The parabolic Airy process is the top line in a random sequence of functions $(\mathfrak A_i : \R\to \R, i \in \N)$ known as the parabolic Airy line ensemble. The parabolic Airy line ensemble can be loosely thought of as a system of infinitely many Brownian motions (of diffusion coefficient $2$) conditioned on the non-intersection event $\mathfrak A_1 > \mathfrak A_2 > \dots $ and with the boundary condition $\mathfrak A_i(x) \sim -x^2$ as $|x| \to \infty$. This idea can be made precise through the Brownian Gibbs resampling property \cite{corwin2014brownian}. When restricted to the top line, this property says that given the function $\mathfrak A_2$ and the values of $\mathfrak A_1$ outside of an interval $[a, b]$, the function $\mathfrak A_1|_{[a, b]}$ is simply a Brownian bridge between the points $(a, \mathfrak A_1(a))$ and $(b, \mathfrak A_1(b))$ conditioned on the event $\mathfrak A_1 > \mathfrak A_2$.
	
	The shifted process $\mathfrak A_2(x) + x^2$ is stationary, and it is difficult to push $\mathfrak A_2$ far below the parabola $-x^2$, so the Brownian Gibbs resampling property suggests that a good proxy for $\mathfrak A_1$ is simply a Brownian motion $\mathfrak B$ (of diffusion coefficient $2$) conditioned on the event $\mathfrak B(x) > -x^2$ for all $x$ and given the boundary condition $\mathfrak B(x) \sim - x^2$ as $|x| \to \infty$ \footnote{To construct $\mathfrak B$ precisely, take the limit in law as $a \to \infty$ of a Brownian bridge $W_a:[-a, a] \to \R$ with $W_a(a) = W_a(-a) = -a^2 + 1$ and conditioned on the event $W_a(x) > -x^2$ for all $x$.}. We can hope to understand the large deviations for $\mathfrak A_1$ by first developing a large deviation principle  for $\mathfrak B$. This follows from Schilder's large deviation principle for Brownian motion.
	
	Let $\mathfrak B_\ep(x) = \ep \mathfrak B(\ep^{-1/2} x)$, so that $\mathfrak B_\ep$ converges in law to the parabola $g(x) = -x^2$ as $\ep \to 0$. For any absolutely continuous function $f:\R \to \R$ with $f(x) \ge g(x)$ for all $x$ and
	$$
	\lim_{x \to \pm \infty} f(x) - x^2 = 0,
	$$
	define
	\begin{equation}
		\label{E:Dirichlet-energy}
		\pj (f) = \frac{1}{4} \int_\R (f'(x)^2 - g'(x)^2) dx= \frac{1}{4} \int_\R (f'(x)^2 - 4 x^2) dx,
	\end{equation} 
	and set $\mathcal I(f) = \infty$ for any other function $f$. Then for $f:\R^2 \to \R$, as $\ep \to 0$,  we should have a large deviation principle which makes precise the statement that
	\begin{equation}
		\label{E:P-B-ep}
		\Pr  (\mathfrak B_\ep \approx f) = \exp(-\ep^2[\pj(f) + o(1)]).
	\end{equation}
	The large deviation principle \eqref{E:P-B-ep} should also hold for the parabolic Airy process. While this result is not straightforward due to the more delicate nature of $\mathfrak A_1$, it seems within reach of current methods, e.g. see \cite{ganguly2022sharp, dauvergne2023wiener} for recent work on related problems. Such a large deviation principle for the parabolic Airy process can be rephrased in terms of our rate function $I$. We state this as an open problem.
	\begin{conjecture}
		\label{P:minimizer-characterization}
		For any function $f:\R\to \R$ we have
		\begin{align}
			\label{E:DE-st}
			\min \{I(e) : e \in \mathcal D, e(0, 0; \cdot, 1) = f \} &= \pj(f).
		\end{align}
	\end{conjecture}
	It would also be quite interesting to find a proof of Conjecture \ref{P:minimizer-characterization} using only our definition of $I$. While we attempted this, we were only able to verify \eqref{E:DE-st} for a few straightforward choices of the function $f$, see Examples \ref{e:one-point} and \ref{e:two-point}.
	
	\subsection{The rate function in terms of gradients of directed metrics}
	
	The above discussion suggests that we may be able to \textit{build} up the rate function $I$ through a kind of integrated Dirichlet energy. This is the goal of the present section. Along the way, we will prove the differentiation formula \eqref{E:sup-theta-R} for the planted network measure in terms of the metric. First, for a metric $e \in \mathcal D$, define its \textbf{gradient field} $D e:\R^2 \times \R \to \R, (q, \theta) \mapsto D_q e (\theta)$ by letting
	$$
	D_q e (\theta) = \liminf_{t \to 0^+} \frac{e(q; q + (t \theta, t))}{t}.
	$$
	As we will see, the gradient field $D e$ typically behaves much more nicely than the metric $e$ itself. In particular, the part of the rate function of $e$ coming from the infinitesimal part of $e$ around $p$ can be expressed through \eqref{E:Dirichlet-energy}. More precisely, we have the following theorem.
	\begin{theorem}
		\label{T:tangent-space-gives-the-metric}
		Let $e \in \mathcal D$ be a finite rate metric and let $\mu = \mu_e$. Then there is a set $S \subset \R$ whose complement has Lebesgue measure $0$ such that for all $q \in \R \times S$, the limit $D_qe(\theta)$ exists for all $\theta \in \R$ and 
		$$
		\pj(D_q e) = \frac{4}{3} \rho_\mu(q)^{3/2}, \qquad \rho_\mu(q) = \sup_{\theta \in \R} D_qe (\theta) - D_q d(\theta).
		$$
	\end{theorem}
	Theorem \ref{T:tangent-space-gives-the-metric} gives the following corollary.
	\begin{corollary}
		\label{C:I-is-defined}
		For any $e \in \mathcal D$ of finite rate, we can write
		$$
		I(e) = \int_\R \Big(\sum_{x \in \R} \pj(D_{(x, s)} e) \Big) ds.
		$$
	\end{corollary}
	
	The inner sum in Corollary \ref{C:I-is-defined} is over uncountably many elements, but by Theorem \ref{T:tangent-space-gives-the-metric}, for Lebesgue a.e.~$s \in \R$ there are only countably many non-zero terms. One way to think of Corollary \ref{C:I-is-defined} is by viewing the gradient field $D_q e$ as the `Riemannian metric tensor' for the directed metric $e$. With this language, Corollary \ref{C:I-is-defined} states that we can build the rate function for the metric by integrating the rate function for the metric tensor over the whole space. 
	
	\begin{proof}[Proof of Theorem \ref{T:tangent-space-gives-the-metric}]
		Let $\Ga$ be a network such that $\gr \Ga$ contains the support of $\mu$, and for every $\ga \in \Ga$ define the excess density $\rho_\ga = \rho_{\ga, e}$ as in \eqref{e:rhogamma}, which recall is related to $\rho_\mu$ by Definition \ref{d:etomu}.
		Let $S \subset \R$ to be the set of points where
		\begin{itemize}[nosep]
			\item $a_\ga, b_\ga \notin S$ for all $\ga \in \Ga$.
			\item $s$ is a Lebesgue point for $\ga', |\ga'|^2, \rho_\ga, \rho_\ga^{3/2}$ for all $\ga \in \Gamma$ with $s \in (a_\ga, b_\ga)$, and also a Lebesgue point for
			$$
			g(z) := \max_{\ga \in \Ga} \rho_\ga^{3/2} \mathbf{1}(z \in [a_\ga, b_\ga]).
			$$
		\end{itemize}	
		Then $S^c$ has Lebesgue measure $0$, by the Lebesgue differentiation theorem, using the integrability of $|\ga'|^2, \rho_\ga^{3/2}$ and $g$ (whose integral over $\R$ is bounded above by $\mathcal K(\mu))$.
		
		Now fix $q = (x, s) \in S \times \R, \theta \in \R$ and let $\pi_r$ denote the rightmost geodesic from $q$ to $q_r := q + (\theta r, r)$. First suppose $q \notin \gr \Ga$ and hence $\rho_\mu(q) = 0$.
		By Lemma \ref{L:geo-mod-cont} there exists a constant $c > 0$ such that $|\pi_r(h) - x| \le c r^{2/3}$ for all $0 < h < r \le 1$. Therefore by Proposition \ref{p:etomu} we can write
		\begin{equation}
			\label{E:pi-r-e}
			\frac{1}{r} (|\pi_r|_e - |\pi_r|_d) \le \frac{1}{r} \int_s^{s + r} \max_{\ga \in \Ga} \rho_\ga(z)\mathbf{1}(z \in [a_\ga, b_\ga]), |\ga(z) - x| \le  c r^{2/3}) dz
		\end{equation}
		By Jensen's inequality, the right-hand side of \eqref{E:pi-r-e} is bounded above by
		$$
		\left(\frac{1}{r} \int_s^{s+r} \max_{\ga \in \Ga} \rho_\ga^{3/2}(z) \mathbf{1}(z \in [a_\ga, b_\ga], |\ga(z) - x| \le  c r^{2/3}) dz \right)^{2/3},
		$$
		whose limit as $r \to 0$ is $0$ since $s$ is a Lebesgue point for $g$ and $\ga', \rho_\ga^{3/2}$ for all $\ga$. Therefore $e(q; q_r) \le d(q; q_r) + o(r)$ as $r \to 0$. The opposite inequality also holds by Dirichlet dominance of $e$, yielding the theorem in this case.
		
		Now suppose $q \in \gr \Ga$. By our assumptions on $S$, there is a unique $\ga \in \Ga$ such that $q \in \gr \ga$. Decompose $\mu = \mu_1 + \mu_2$, where $\mu_1 = \mu|_{\gr \ga}, \mu_2 = \mu|_{(\gr \ga)^c}$. We have that
		$$
		e_{\mu_1}(q; q_r) \le e(q; q_r) \le e_{\mu_1}(q; q_r) + e_{\mu_2}(q; q_r) - d(q; q_r),
		$$
		and from the previous case $e_{\mu_2}(q; q_r) - d(q; q_r) = o(r)$ as $r \to 0$. Therefore it suffices to prove this case for $\mu_1$. Letting $w' = \rho_\ga - |\ga'|^2$, our goal will be to show that
		\begin{equation}
			\label{E:Dp-desire}
			D_q e (\theta) = \begin{cases}
				-\theta^2,  &|\theta - \ga_1'(s)| \ge \sqrt{\rho_\ga(s)}, \\
				w'(s) - 2 \sqrt{\rho_\ga(s)}|\theta - \ga'(s)| + 2\ga'(s)(\theta - w'(s)),  &|\theta - \ga'(s)| \le \sqrt{\rho_\ga(s)}.
			\end{cases}
		\end{equation}
		This function is the concave majorant of the parabola $-\theta^2$ and a single spike of height $w'(s)$ at location $\ga'(s)$.
		From here a quick calculation proves the two claims in the theorem. To prove \eqref{E:Dp-desire}, observe that
		\begin{equation}
			\label{E:emu1}
			e_{\mu_1}(q; q_r) = \sup_{\pi:q \to q_r} \mu_1(\pi \cap \ga) + |\pi|_d = o(r) + \sup_{\pi:q \to q_r} \rho_\ga(s)\lambda\{t: \pi(t) = \ga(t)\}+ |\pi|_d,
		\end{equation}
		where here $\lambda$ denotes Lebesgue measure, and in the equality we have used that $s$ is a Lebesgue point of $\rho_\ga$. Now, since $s$ is a Lebesgue point of $\ga', |\ga'|^2$ we have $ \ga(s + h) = \ga(s) + h \ga'(s) + o(h)$ and $|\ga|_{[s, s + h]}|_d + o(r) = h |\ga'(s)|^2 + o(h)$ as $h \to 0$. Therefore in \eqref{E:emu1}, up to a $o(r)$-term it suffices to optimize over paths $\pi$ that equal $\ga$ on an initial interval $[s, s + hr]$ and are straight lines afterwards. Hence \eqref{E:emu1} equals
		$$
		o(r) + \sup_{h \in [0, 1]} \rho_\ga(s) h r + h r|\ga'(s)|^2  +  \frac{r(\ga'(s) h - \theta)^2}{1 - h},
		$$
		from which the formula \eqref{E:Dp-desire} easily follows after taking $r \to 0$.
	\end{proof}

	\bibliographystyle{dcu}
	\bibliography{ldp}
	\bigskip
	
	\noindent
	Sayan Das. 	Department of Mathematics, University of Chicago,
	5734 S. University Avenue, Room 108
	Chicago, IL, 60637.  sayan.das@columbia.edu
	
	\medskip
	\noindent Duncan Dauvergne. 	Departments of Mathematics, University of Toronto,
	40 St. George St., Toronto, Ontario, M5S2E4. duncan.dauvergne@utoronto.ca
	
	\medskip
	\noindent B\'alint Vir\'ag.
	Departments of Mathematics, University of Toronto, 40 St. George St., Toronto, Ontario, M5S2E4. balint@math.toronto.edu
\end{document}